\documentclass[11pt, reqno]{amsart}
\usepackage{amssymb,latexsym,amsmath,amsfonts}
\usepackage{latexsym}
\usepackage[mathscr]{eucal}
\usepackage{bm}

\def ~{\hspace{1mm}}

\def\textmatrix#1&#2\\#3&#4\\{\bigl({#1 \atop #3}\ {#2 \atop #4}\bigr)}
\def\dispmatrix#1&#2\\#3&#4\\{\left({#1 \atop #3}\ {#2 \atop #4}\right)}
\newcommand{\beg}{\begin{equation}}
\newcommand{\eeg}{\end{equation}}
\newcommand{\ben}{\begin{eqnarray*}}
\newcommand{\een}{\end{eqnarray*}}

\newtheorem{thm}{Theorem}[section]

\newtheorem{lem}[thm]{Lemma}

\newtheorem{prop}[thm]{Proposition}
\numberwithin{equation}{section} \theoremstyle{definition}
\newtheorem{defn}[thm]{Definition}
\newtheorem{rem}[thm]{Remark}
\newtheorem{note}[thm]{Note}
\newtheorem{eg}[thm]{Example}

\def\textmatrix#1&#2\\#3&#4\\{\bigl({#1 \atop #3}\ {#2 \atop #4}\bigr)}
\def\dispmatrix#1&#2\\#3&#4\\{\left({#1 \atop #3}\ {#2 \atop #4}\right)}

\begin{document}

\title[Rational dilation on $\Gamma_3$]
{Rational dilation on the symmetrized tridisc: failure, success
and unknown}

\author[Sourav Pal]{Sourav Pal}
%\dedicatory{\textup{(Dedicated to Tapasundar Bandyopadhyay)}}

\address[ Sourav Pal]{Department of Mathematics, Indian Institute of Technology Bombay,
Powai, Mumbai - 400076, India. } \email{sourav@math.iitb.ac.in,
souravmaths@gmail.com}



\thanks{The author is supported by Seed Grant of IIT Bombay, CPDA and INSPIRE Faculty Award
(Award No. DST/INSPIRE/04/2014/001462) of DST, India. A part of
this work was done when the author was visiting the Department of
Mathematics, Ben-Gurion University of the Negev, Israel and was
supported in part by the Skirball Foundation via the Center for
Advanced Studies in Mathematics at Ben-Gurion University of the
Negev, Israel.}

\begin{abstract}
The closed symmetrized tridisc $\Gamma_3$ and its distinguished
boundary $b\Gamma_3$ are the sets
\begin{gather*}
\Gamma_3=\{ (z_1+z_2+z_3,z_1z_2+z_2z_3+z_3z_1,z_1z_2z_3):
\,|z_i|\leq 1, i=1,2,3 \}\subseteq \mathbb C^3 \\ b\Gamma_3=\{
(z_1+z_2+z_3,z_1z_2+z_2z_3+z_3z_1,z_1z_2z_3): \,|z_i|= 1, i=1,2,3
\}\subseteq \Gamma_3.
\end{gather*}
A triple of commuting operators $(S_1,S_2,P)$ defined on a Hilbert
space $\mathcal H$ for which $\Gamma_3$ is a spectral set is
called a $\Gamma_3$-contraction. In this article we show by a
counter example that there are $\Gamma_3$-contractions which do
not dilate. It is also shown that under certain conditions a
$\Gamma_3$-contraction can have normal $b\Gamma_3$ dilation. We
determine several classes of $\Gamma_3$-contractions which dilate
and show explicit construction of their dilations. A concrete
functional model is provided for the $\Gamma_3$-contractions which
dilate. Various characterizations for $\Gamma_3$-unitaries and
$\Gamma_3$-isometries are obtained; the classes of
$\Gamma_3$-unitaries and $\Gamma_3$-isometries are analogous to
the unitaries and isometries in one variable operator theory. Also
we find out a model for the class of pure $\Gamma_3$-isometries.
En route we study the geometry of the sets $\Gamma_3$ and
$b\Gamma_3$ and provide variety of characterizations for them.

\end{abstract}

\maketitle \tableofcontents

\section{Introduction}

Throughout the paper all operators are bounded linear operators
defined on complex Hilbert spaces unless and otherwise stated. A
contraction is an operator with norm not greater than one. We
define spectral set, complete spectral set, distinguished boundary
and rational dilation in Section \ref{background}.

\subsection{Motivation}

\noindent The aim of dilation roughly speaking is to model a given
tuple of commuting operators as a compression of a tuple of
commuting normal operators. For a commuting tuple
$(T_1,\cdots,T_d)$ associated with a domain in $\mathbb C^d$,
where each $T_i$ is defined on a Hilbert space $\mathcal H$, the
purpose of dilation is to find out a normal tuple $(N_1,\cdots,
T_d)$ associated with the boundary of the domain such that for
each $i$
\[
T_i=P_{\mathcal H} N_i|_{\mathcal H}\,,
\]
where each $N_i$ is defined on a bigger Hilbert space $\mathcal K$
and $P_{\mathcal H}$ is the orthogonal projection of $\mathcal K$
onto $\mathcal H$. In 1951, von Neumann, \cite{von-Neumann},
introduced the notion of spectral set for an operator which turned
our attention, when studying an operator, to an underlying compact
subset of $\mathbb C$. The notion was appealing in the sense that
it could describe all contractions as operators having the closed
unit disk $\overline{\mathbb D}$ as a spectral set.

\begin{thm}[von Neumann, 1951]
An operator $T$ is a contraction if and only if $\overline{\mathbb
D}$ is a spectral set for $T$.
\end{thm}

Later the notion of spectral set was extended for any finite
number of commuting operators and beautiful interplays were
witnessed between the operators having a particular domain in
$\mathbb C^d$ as a spectral set and the complex geometry and
function theory of the associated domain, \cite{paulsen, AgMcC}.
In 1953, Sz.-Nagy published a very influential paper \cite{nagy1}
studying a contraction and establishing the following theorem
whose impact is extraordinary till date.

\begin{thm}[Sz.-Nagy, 1953]
If $T$ is a contraction acting on a Hilbert space $\mathcal H$,
then there exists a Hilbert space $\mathcal K \supseteq \mathcal
H$ and a unitary $U$ on $\mathcal K$ such that
\[
p(T)=P_{\mathcal H}p(U)|_{\mathcal H}
\]
for every polynomial $p$ in one complex variable.
\end{thm}
Sz.-Nagy's dilation theorem removed much of the mystery of one
variable operator theory by expressing an abstract object like an
arbitrary contraction as a compression of a more well known
object, a unitary. Since every operator is nothing but a scalar
time a contraction, Sz.-Nagy's result provided a subtle way of
modelling an operator in terms of a normal operator or more
precisely a scalar time a unitary.\\

Sz-Nagy dilation theorem even holds for all rational functions
with poles off $\overline{\mathbb D}$, which actually establishes
the success of rational dilation on the closed disk. By von
Neumann's theorem, a contraction is an operator that lives inside
$\overline{\mathbb D}$ and Sz.-Nagy's theorem provides a normal
dilation that lives in the boundary of $\overline{\mathbb D}$. In
higher dimensions, in the context of rational dilation, the role
of bounary is replaced by a more refined distinguished boundary.
The success of rational dilation on the closed disk prompted a
number of mathematicians to ask the same question for an arbitrary
compact subset of $\mathbb C^d$, that is, for a commuting tuple of
operators $(T_1,\dots,T_d)$ acting on a Hilbert space $\mathcal H$
for which a given compact subset $K$ of $\mathbb C^d$ is a
spectral set whether or not we can find out a commuting normal
tuple $(N_1,\dots,N_d)$ acting on a bigger Hilbert space $\mathcal
K \supseteq \mathcal H$ and having the distinguished boundary $bK$
as a spectral set such that
\[
f(T_1,\dots,T_d)=P_{\mathcal H}f(N_1,\dots,N_d)|_{\mathcal H}\,,
\]
for all rational functions $f$ in complex $d$-variables with poles
off $K$.\\

In 1985, Jim Agler found positive answer to this question for an
annulus. He constructed dilation for an annulus using an
innovative technique, \cite{agler-ann}. In 2005, Dritschel and
McCulough \cite{DM} resolved this issue for a triply connected
domain with a negative answer. The failure of rational dilation on
a triply connected domain was also shown independently by Jim
Agler and his collaborators, \cite{ahr}. In higher dimensions we
have success of rational dilation on the closed bidisk
$\overline{\mathbb D^2}$ by T. Ando \cite{ando}, which is known as
Ando's inequality. Also we have failure on the closed tridisk
$\overline{\mathbb D^3}$ \cite{parrott}, and on the tetrablock,
\cite{sourav1}. Till date we have few instances where rational
dilation succeeds or fails but the issue of characterizing all
compact subsets of $\mathbb C^d$ where rational
dilation succeeds is still unsettled.\\

In recent past, Jim Agler and Nicholas Young established the
success of rational dilation on the closed symmetrized bidisc
\[
\Gamma_2 = \{ (z_1+z_2,z_1z_2): z_1,z_2\in \overline{\mathbb D}
\}\,,
\]
\cite{ay-jfa}, by showing the existence of dilation using
Stinespring's dilation theorem (see \cite{paulsen} for
Stinespring's theorem). Also the author of this paper and his
collaborators constructed such a dilation independently in
\cite{tirtha-sourav, sourav}. The symmetrized polydisc is a well
studied domain in past two decades and we will see many references
in the next subsection. Since rational dilation succeeds on the
symmetrized bidisc, there are subtleties in asking if it succeeds
on its higher dimensional analogues. In this article, we show by a
counter example that it fails in dimension three, that is on the
closed symmetrized tridisc

\[
\Gamma_3=\{  (z_1+z_2+z_3,z_1z_2+z_2z_2+z_3z_1,z_1z_2z_3):\, z_i
\in\overline{\mathbb D} \}\,.
\]

This essentially terminates our interest of proceeding further to
the higher dimensions. Also we study and find various occasions
when it succeeds conditionally. In such cases, we construct
explicit dilations. En route, we find a variety of
characterizations for the set $\Gamma_3$, its distinguished
boundary $b\Gamma_3$ and for several important classes of operator
triples having $\Gamma_3$ as a spectral set.

\subsection{Literature and a brief description of the main results}

For $n\geq 2$, the symmetrization map in $n$-complex variables
$z=(z_1,\dots,z_n)$ is the following
\[
\pi_n(z)=(s_1(z),\dots, s_{n-1}(z), p(z))
\]
 where
 \[
s_i(z)= \sum_{1\leq k_1 \leq k_2 \dots \leq k_i \leq n-1}
z_{k_1}\dots z_{k_i} \quad \text{ and } p(z)=\prod_{i=1}^{n}z_i\,.
 \]
The map $\pi_n$ is a proper holomorphic map, \cite{rudin}. The
closed \textit{symmetrized} $n$-\textit{disk} (or simply closed
\textit{symmetrized polydisc}) is the image of the closed $n$-disc
$\overline{\mathbb D^n}$ under the symmetrization map $\pi_n$,
that is, $\Gamma_n := \pi_n(\overline{\mathbb D^n})$. Similarly
the open symmetrized polydisc is defined as the image of the open
polydisc $\mathbb D^n$ under $\pi_n$. The set $\Gamma_n$ is
polynomially convex but not convex (see \cite{edi-zwo, BSR}). So
in particular the closed and open symmetrized tridisc are the sets
\begin{align*}
\Gamma_3 &=\{ (z_1+z_2+z_3,z_1z_2+z_2z_3+z_3z_1,z_1z_2z_3):
\,|z_i|\leq 1, i=1,2,3 \} \subseteq \mathbb C^3 \\ \mathbb G_3 &
=\{ (z_1+z_2+z_3,z_1z_2+z_2z_3+z_3z_1,z_1z_2z_3): \,|z_i|< 1,
i=1,2,3 \}\subseteq \Gamma_3.
\end{align*}

We obtain from the literature (see \cite{edi-zwo, BSR}) that the
distinguished boundary of the symmetrized polydisc is the
symmetrization of the distinguished boundary of the
$n$-dimensional polydisc, which is $n$-torus $\mathbb T^n$. Hence
the distinguished boundary for $\Gamma_3$ is the set
\begin{align*}
b\Gamma_3 &=\{(z_1+z_2+z_3,z_1z_2+z_2z_3+z_3z_1,z_1z_2z_3):
\,|z_i|=1, i=1,2,3\}.
\end{align*}

The symmetrized polydiscs in several dimensions have attracted
considerable attention in past two decades because of its rich
function theory \cite{ALY12, AY06, AY08, costara1, jarnicki, KZ,
MSZ, NiPa, PZ}, complex geometry \cite{AY04, costara, edi-zwo, GS,
Le86, NiPfZw}, associated operator theory \cite{ay-jfa, ay-jot,
AY, tirtha-sourav, tirtha-sourav1, BSR, sourav, pal-shalit,
sarkar} and its connection with the most appealing and difficult
problem of $\mu$-synthesis \cite{ALY13, hari, NJY11}, which arises
in the $H^{\infty}$ approach to the problem of robust control
\cite{JCD, JCD1}. Operator theory on the symmetrized bidisc has
numerous applications to its complex geometry and function theory,
see classic \cite{AgMcC}. Nevertheless, operator theory on a
domain is always of independent interest even without considering
any connection with complex geometry or function theory of the
domain.\\

\begin{defn}
A triple of commuting operators $(S_1,S_2,P)$ on a Hilbert space
$\mathcal H$ for which $\Gamma_3$ is a spectral set (complete
spectral set) is called a $\Gamma_3$-$contraction$
(\textit{complete} $\Gamma_3$-\textit{contraction}).
\end{defn}

If $(S_1,S_2,P)$ is a $\Gamma_3$-contraction then so are
$(S_2,S_1,P)$ and $(S_1^*,S_2^*,P^*)$. We shall see a proof of
this result in section 4. Also it is obvious from the definition
that if $(S_1,S_2,P)$ is a $\Gamma_3$-contraction then $S_1,S_2$
have norms not greater than $3$ and $P$ is a contraction.
Unitaries, isometries and co-isometries are important special
classes of contractions. There are natural analogues of these
classes for $\Gamma_3$-contractions.

\begin{defn}
Let $S_1,S_2,P$ be commuting operators on a Hilbert space
$\mathcal H$. We say that $(S_1,S_2,P)$ is
\begin{itemize}
\item [(i)] a $\Gamma_3$-\textit{unitary} if $S_1,S_2,P$ are
normal operators and the Taylor joint spectrum
$\sigma_T(S_1,S_2,P)$ is contained in $b\Gamma_3$ ; \item [(ii)] a
$\Gamma_3$-\textit{isometry} if there exists a Hilbert space
$\mathcal K$ containing $\mathcal H$ and a $\Gamma_3$-unitary
$(\tilde{S_1},\tilde{S_2},\tilde{P})$ on $\mathcal K$ such that
$\mathcal H$ is a common invariant subspace for
$\tilde{S_1},\tilde{S_2},\tilde{P}$ and that
$S_i=\tilde{S_i}|_{\mathcal H}$ for $i=1,2$ and
$\tilde{P}|_{\mathcal H}=P$; \item [(iii)] a
$\Gamma_3$-\textit{co-isometry} if $(S_1^*,S_2^*,P^*)$ is a
$\Gamma_3$-isometry.
\end{itemize}
\end{defn}

\begin{defn}
A $\Gamma_3$-isometry $(S_1,S_2,P)$ is said to be \textit{pure} if
$P$ is a pure isometry, i.e, if ${P^*}^n \rightarrow 0$ strongly
as $n \rightarrow \infty$.
\end{defn}

It is evident from the definitions (see Section \ref{background})
that rational dilation of a $\Gamma_3$-contraction $T=(S_1,S_2,P)$
is nothing but a $\Gamma_3$-unitary dilation of $T$, that is, a
$\Gamma_3$-unitary $N=(N_1,N_2,N_3)$ that dilates $T$ by
satisfying (\ref{rational-dilation}). Similarly a
$\Gamma_3$-isometric dilation of $T=(S_1,S_2,P)$ is a
$\Gamma_3$-isometry $V=(V_1,V_2,V_3)$ that satisfies
(\ref{rational-dilation}). Clearly a $\Gamma_3$-unitary dilation
is necessarily a $\Gamma_3$-isometric dilation.\\

For a commuting triple $(S_1,S_2,P)$ with $P$ being a contraction,
we consider the following two operator equations
\[
S_1-S_2^*P=D_{P}X_1D_{P}\, , \; S_2-S_1^*P=D_{P}X_2D_{P}
\]
where $D_{P}=(I-P^*P)^{\frac{1}{2}}$ and $\mathcal
D_{P}=\overline{Ran}\,D_{P}$. We call them \textit{fundamental
equations} for such a triple. In Theorem
\ref{existence-uniqueness} we show that if $(S_1,S_2,P)$ is a
$\Gamma_3$-contraction, then the fundamental equations have unique
solutions $F_1,F_2$ respectively. The unique pair $(F_1,F_2)$ is
called the \textit{fundamental operator pair} of $(S_1,S_2,P)$. We
shall write FOP for fundamental operator pair. The FOP plays
crucial role in determining the failure of rational dilation on
$\Gamma_3$ and also in finding out a major class of
$\Gamma_3$-contractions which dilate to $\Gamma_3$-unitaries.

\begin{defn}
A pair of operators $(T_1,T_2)$ acting on $\mathcal H$ is said to
be {\em almost normal} if $T_1,T_2$ commute and
$T_1^*T_1-T_1T_1^*=T_2^*T_2-T_2T_2^*$.
\end{defn}

In Section \ref{necessary-condition}, we produce a necessary
condition for the existence of rational dilation for a class of
$\Gamma_3$-contractions. Indeed, in Proposition \ref{ultimate}, we
show that if $(S_1,S_2,P)$ is a $\Gamma_3$-contraction on
$\mathcal H_1 \oplus \mathcal H_1$ for some Hilbert space
$\mathcal H_1$, satisfying
\begin{itemize}
 \item[(i)] $ Ker (D_{P})=\mathcal H_1\oplus \{0\} \textup{ and } \mathcal D_{P} =
 \{0\}\oplus \mathcal H_1$
\item[(ii)] $P(\mathcal D_{P})=\{0\}$ \text{and} $P
Ker(D_{P})\subseteq \mathcal D_{P}$ ,
\end{itemize}
then for the existence of a $\Gamma_3$-isometric dilation of
$(S_1^*,S_2^*,P^*)$ it is necessary that the FOP $(F_1,F_2)$ of
$(S_1,S_2,P)$ is almost normal. In section \ref{example}, we
construct an example of a $\Gamma_3$-contraction that satisfies
the hypotheses of Proposition \ref{ultimate} but fails to possess
almost normal FOP. This proves the failure of
rational dilation on the symmetrized tridisc.\\

In spite of such failure in general, we become able to show that
if the FOP of $(S_1,S_2,P)$ is almost normal, then $(S_1,S_2,P)$
dilates to a $\Gamma_3$-unitary. For such a $\Gamma_3$-contraction
$(S_1,S_2,P)$, Theorem \ref{main-dilation-theorem} describes an
explicit construction of a $\Gamma_3$-unitary dilation
$(R_1,R_2,U)$. Though in this theorem the almost normality of the
FOP of $(S_1^*,S_2^*,P^*)$ is assumed but it is required to
construct such a dilation, an existence of dilation demands only
the almost normality of $(F_1,F_2)$. This is obvious from Theorem
\ref{isometric-dilation}, where a $\Gamma_3$-isometric dilation to
$(S_1,S_2,P)$ is obtained. The dilation operators are constructed
with the help of the FOPs of $(S_1,S_2,P)$ and
$(S_1^*,S_2^*,P^*)$. It is remarkable in this construction that
$U$ is the minimal unitary dilation of $P$. This leads to the
conclusion that the dilation space is same as the minimal unitary
dilation space of
$P$. Naturally the dilation becomes minimal.\\

In Theorem \ref{model2}, we find a concrete functional model for a
$\Gamma_3$-contraction $(S_1,S_2,P)$ whose adjoint has almost
normal FOP. This result also guarantees that almost normality of
FOP of any one of $(S_1,S_2,P)$ or $(S_1^*,S_2^*,P^*)$ is
sufficient for $(S_1,S_2,P)$ to dilate.\\

So far we have witnessed that for a $\Gamma_3$-contraction, the
almost normality of the FOP is sufficient but not necessary for
the existence of a rational dilation. In Section
\ref{further-dilation} we shall see that there are
$\Gamma_3$-contractions which, without having almost normal FOP,
dilate to $\Gamma_3$-unitaries. In the same section we obtain few
other classes of $\Gamma_3$-contractions and find their dilations.
However, we are unable to determine the entire class of
$\Gamma_3$-contractions which dilate even without having almost
normal FOPs.\\

One can ask a very natural question. Given a commuting triple
$(S_1,S_2,P)$ with $\|S_i\|\leq 3$ and $\|P\|\leq 1$, when can we
declare that $(S_1,S_2,P)$ is a $\Gamma_3$-contraction ? We answer
this question partially in the end of section
\ref{conditional-dilation}. In Theorem \ref{sufficient1}, we show
that the existence of solutions to the fundamental equations of a
commuting triple $(S_1,S_2,P)$ and the almost normality of the
solution pair are sufficient for $(S_1,S_2,P)$ to become a
$\Gamma_3$-contraction. In fact under such
conditions $(S_1,S_2,P)$ becomes a complete $\Gamma_3$-contraction.\\

Throughout the whole program, the FOP of a $\Gamma_3$-contraction
plays the main role. So it is worth having a further study on
these operator pairs. In Section \ref{basicproperties}, we make
reasonable progress in this matter and discover beautiful
properties of FOPs. The remarkable among them is Proposition
\ref{end-prop}, which states that if two $\Gamma_3$-contractions
are unitarily equivalent then so are their FOPs. Also Theorem
\ref{converse} provides a partial converse to the
existence-uniqueness theorem (Theorem \ref{existence-uniqueness})
of FOP. This result ascertains that for every almost normal
operator pair $(F_1,F_2)$ with certain norm condition, there
exists a $\Gamma_3$-contraction for which $(F_1,F_2)$ is the FOP.\\

The proof of the existence-uniqueness theorem of FOP requires the
positivity of two operator pencils $\Phi_{13},\Phi_{23}$ which we
have defined in section \ref{basicproperties} (see (\ref{eq:1a}),
(\ref{eq:1b})). Proposition \ref{prop:POS} shows the positivity of
$\Phi_{1k},\Phi_{2k}$ for $k\geq 3$. The positivity of $\Phi_{13}$
and $\Phi_{23}$ in scalar case (see equations (\ref{eqn:2a}),
(\ref{eqn:2b})) is a necessary and sufficient condition for a
point $(s_1,s_2,p)$ to belong to $\Gamma_3$ which is proved in
Theorem \ref{thm:SC1}. In other words, the positivity of
$\Phi_{13}$ and $\Phi_{23}$ determines the success of rational
dilation in scalar case. However the positivity of these two
operator pencils is necessary but not sufficient in operator case.
The question of determining whether $\Gamma_3$ is a spectral set
or complete spectral set for a point $(s_1,s_2,p)$ in $\mathbb
C^3$ is closely related to Schur's theorem on location of zeros of
a polynomial, \cite{schur, ptak-young}. We shall apply this
theorem to prove Theorem \ref{thm:SC1} which characterizes a point
of $\Gamma_3$ in several ways including the success of rational
dilation in scalar case. Also Theorem \ref{thm:DB} provides a
variety of
characterizations for the distinguished boundary of $\Gamma_3$.\\

We have already seen that the classes of $\Gamma_3$-unitaries and
$\Gamma_3$-isometries are analogues of unitaries and isometries in
one variable operator theory. Theorem \ref{G-unitary} describes
the structure of a $\Gamma_3$-unitary by a set of
characterizations. It shows that a $\Gamma_3$-unitary is nothing
but the symmetrization of three commuting unitaries which
parallels the scalar case. However this is no longer true in
general for every $\Gamma_3$-contraction. In Remark 2.11 in
\cite{BSR} , Biswas and Shyam Roy have established by an example
that symmetrization of three commuting contractions may not be a
$\Gamma_3$-contraction. Theorem \ref{G-isometry} provides a set of
characterizations for the $\Gamma_3$-isometries. It shows that
every $\Gamma_3$-isometry admits an Wold-type decomposition that
splits a $\Gamma_3$-isometry orthogonally into two parts of which
one is a $\Gamma_3$-unitary and the other is a pure
$\Gamma_3$-isometry.\\

As a $\Gamma_3$-isometry is an orthogonal direct sum of a
$\Gamma_3$-unitary and a pure $\Gamma_3$-isometry and since
Theorem \ref{G-unitary} describes the structure of a
$\Gamma_3$-unitary, one needs a concrete model of pure
$\Gamma_3$-isometries to get a complete description of a
$\Gamma_3$-isometry. In Theorem \ref{model1}, we show that a pure
$\Gamma_3$-isometry $(\hat{S_1},\hat{S_2},\hat{P})$ can be
modelled as a commuting triple of Toeplitz operators
$(T_{A+Bz},T_{B^*+A^*z},T_z)$ on the vectorial Hardy space
$H^2(\mathcal D_{\hat{P^*}})$, where $(A^*,B)$ is the FOP of the
$\Gamma_3$-co-isometry $(\hat{S_1}^*,\hat{S_2}^*,\hat{P}^*)$. The
converse is also true, that is, every such triple of commuting
contractions $(T_{A+Bz},T_{B^*+A^*z},T_z)$ on a vectorial Hardy
space is a pure
$\Gamma_3$-isometry.\\

\noindent \textbf{Note.} Part $(1)\Leftrightarrow (4)$ of Theorem
\ref{thm:DB} and part $(1)\Leftrightarrow (7)$ of Theorem
\ref{thm:SC1} were obtained independently by Costara in
\cite{costara} and later by Gorai and Sarkar, \cite{GS}. Also
parts of Theorem \ref{G-unitary}, Theorem \ref{G-isometry} and
Theorem \ref{model1} were proved by Biswas and Shyam Roy
\cite{BSR}, in a more general setting. However, the proofs given
here to most of these results are different.

\section{Spectral set, complete spectral set and rational
dilation}\label{background}

\subsection{The Taylor joint spectrum}

Let $\Lambda$ be the exterior algebra on $n$ generators
$e_1,...e_n$, with identity $e_0\equiv 1$. $\Lambda$ is the
algebra of forms in $e_1,...e_n$ with complex coefficients,
subject to the collapsing property $e_ie_j+e_je_i=0$ ($1\leq i,j
\leq n$). Let $E_i: \Lambda \rightarrow \Lambda$ denote the
creation operator, given by $E_i \xi = e_i \xi $ ($\xi \in
\Lambda, 1 \leq i \leq n$).
 If we declare $ \{ e_{i_1}... e_{i_k} : 1 \leq i_1 < ... < i_k \leq n \}$ to be an
 orthonormal basis, the exterior algebra $\Lambda$ becomes a Hilbert space,
 admitting an orthogonal decomposition $\Lambda = \oplus_{k=1} ^n \Lambda^k$
 where $\dim \Lambda ^k = {n \choose k}$. Thus, each $\xi \in \Lambda$ admits
 a unique orthogonal decomposition
 $ \xi = e_i \xi' + \xi''$, where $\xi'$ and $\xi ''$ have no $e_i$
contribution. It then follows that that  $ E_i ^{*} \xi = \xi' $,
and we have that each $E_i$ is a partial isometry, satisfying
$E_i^*E_j+E_jE_i^*=\delta_{i,j}$. Let $\mathcal X$ be a normed
space, let $\underline{T}=(T_1,...,T_n)$ be a commuting $n$-tuple
of bounded operators on $\mathcal X$ and set $\Lambda(\mathcal
X)=\mathcal X\otimes_{\mathbb{C}} \Lambda$. We define
$D_{\underline T}: \Lambda (\mathcal X) \rightarrow \Lambda
(\mathcal X)$ by

\[
D_{\underline T} = \sum_{i=1}^n T_i \otimes E_i .
\]

Then it is easy to see $D_{\underline T}^2=0$, so $Ran
D_{\underline T} \subset Ker D_{\underline T}$. The commuting
$n$-tuple is said to be \textit{non-singular} on $\mathcal X$ if
$Ran D_{\underline T}=Ker D_{\underline T}$.
\begin{defn}
The Taylor joint spectrum of ${\underline T}$ on $\mathcal X$ is
the set
\[
\sigma_T({\underline T},\mathcal X) = \{
\lambda=(\lambda_1,...,\lambda_n)\in \mathbb{C}^n : {\underline
T}-\lambda \text{ is singular} \}.
\]
\end{defn}
\begin{rem}
The decomposition $\Lambda=\oplus_{k=1}^n \Lambda^k$ gives rise to
a cochain complex $K({\underline T},\mathcal X)$, known as the
Koszul complex associated to ${\underline T}$ on $\mathcal X$, as
follows:
\[
K({\underline T},\mathcal X):0 \rightarrow \Lambda^0(\mathcal
X)\xrightarrow{D_{\underline T}^0}... \xrightarrow{D_{\underline
T}^{n-1}} \Lambda^n(\mathcal X) \rightarrow 0 ,
\]
where $D_{\underline T}^{k}$ denotes the restriction of
$D_{\underline T}$ to the subspace $\Lambda^k(\mathcal X)$. Thus,
\[
\sigma_T({\underline T},\mathcal X) = \{ \lambda\in \mathbb{C}^n :
K({\underline T}-\lambda ,\mathcal X)\text{ is not exact} \}.
\]
\end{rem}
For a further reading on Taylor joint spectrum an interested
reader is referred to Taylor's works, \cite{Taylor, Taylor1}.

\subsection{Spectral and complete spectral set}

We shall follow Arveson's terminologies about the spectral and
complete spectral sets. Let $X$ be a compact subset of $\mathbb
C^n$ and let $\mathcal R(X)$ denote the algebra of all rational
functions on $X$, that is, all quotients $p/q$ of polynomials
$p,q$ for which $q$ has no zeros in $X$. The norm of an element
$f$ in $\mathcal R(X)$ is defined as
$$\|f\|_{\infty, X}=\sup \{|f(\xi)|\;:\; \xi \in X  \}. $$
Also for each $k\geq 1$, let $\mathcal R_k(X)$ denote the algebra
of all $k \times k$ matrices over $\mathcal R(X)$. Obviously each
element in $\mathcal R_k(X)$ is a $k\times k$ matrix of rational
functions $F=(f_{i,j})$ and we can define a norm on $\mathcal
R_k(X)$ in the canonical way
$$ \|F\|=\sup \{ \|F(\xi)\|\;:\; \xi\in X \}, $$ thereby making
$\mathcal R_k(X)$ into a non-commutative normed algebra. Let
$\underline{T}=(T_1,\cdots,T_n)$ be an $n$-tuple of commuting
operators on a Hilbert space $\mathcal H$. The set $X$ is said to
be a \textit{spectral set} for $\underline T$ if the Taylor joint
spectrum $\sigma_T (\underline T)$ of $\underline T$ is a subset
of $X$ and
\begin{equation}\label{defn1}
\|f(\underline T)\|\leq \|f\|_{\infty, X}\,, \textup{ for every }
f\in \mathcal R(X).
\end{equation}
Here $f(\underline T)$ can be interpreted as $p(\underline
T)q(\underline T)^{-1}$ when $f=p/q$. Moreover, $X$ is said to be
a \textit{complete spectral set} if $\|F(\underline T)\|\leq
\|F\|$ for every $F$ in $\mathcal R_k(X)$, $k=1,2,\cdots$.

\subsection{The distinguished boundary and rational dilation}

Let $\mathcal A(X)$ be an algebra of continuous complex-valued
functions on $X$ which separates the points of $X$. A
\textit{boundary} for $\mathcal A(X)$ is a closed subset $\partial
X$ of $X$ such that every function in $\mathcal A(X)$ attains its
maximum modulus on $\partial X$. It follows from the theory of
uniform algebras that the intersection of all the boundaries of
$X$ is also a boundary for $\mathcal A(X)$ (see Theorem 9.1 of
\cite{wermer}). This smallest boundary is called the
$\check{\textup{S}}$\textit{ilov boundary} for $\mathcal A(X)$.
When $\mathcal A(X)$ is the algebra of rational functions which
are continuous on $X$, the $\check{\textup{S}}$\textit{ilov
boundary} for $\mathcal A(X)$ is called the \textit{distinguished
boundary} of $X$ and is denoted by $bX$.\\

A commuting $n$-tuple of operators $\underline T$ on a Hilbert
space $\mathcal H$, having $X$ as a spectral set, is said to have
a \textit{rational dilation} or \textit{normal}
$bX$-\textit{dilation} if there exists a Hilbert space $\mathcal
K$, an isometry $V:\mathcal H \rightarrow \mathcal K$ and an
$n$-tuple of commuting normal operators $\underline
N=(N_1,\cdots,N_n)$ on $\mathcal K$ with $\sigma_T(\underline
N)\subseteq bX$ such that

\begin{equation}\label{rational-dilation}
f(\underline T)=V^*f(\underline N)V, \textup{ for every } f\in
\mathcal R(X),
\end{equation}

or, in other words for every $f\in \mathcal R(X)$

\begin{equation}
f(\underline T)=P_{\mathcal H}f(\underline N)|_{\mathcal H}\,,
\end{equation}

when $\mathcal H$ is considered as a closed linear subspace of
$\mathcal K$. Moreover, the dilation is called {\em minimal} if
\[
\mathcal K=\overline{\textup{span}}\{ f(\underline N) h\,:\;
h\in\mathcal H \textup{ and } f\in \mathcal R(K) \}.
\]

It obvious that if $X$ is a complete spectral set for $\underline
T$ then $X$ is a spectral set for $\underline T$. A celebrated
theorem of Arveson states that $\underline T$ has a normal
$bX$-dilation if and only if $X$ is a complete spectral set of
$\underline T$ (Theorem 1.2.2 and its corollary, \cite{sub2}).
Therefore, the success or failure of rational dilation is
equivalent to asking whether $X$ is a spectral set for $\underline
T$ implies that $X$ is a complete spectral set for $\underline
T$.\\

Arveson \cite{sub2} profoundly reformulated the rational dilation
problem in terms of contractive and completely contractive
representations. A tuple $\underline{T}$ acting on the Hilbert
space $\mathcal H$ with Taylor joint spectrum in $X$ determines a
unital representation $\pi_{\underline T}$ of $\mathcal R(X)$ on
$\mathcal H$ via the map $\pi_{\underline T}(f) = f(\underline T)$
and the condition that $X$ is a spectral set for $\underline T$ is
equivalent to the condition that this representation is
contractive. Recall that a representation $\pi$ of $\mathcal R(X)$
is \textit{contractive} if for all $f\in \mathcal R(X)$
\[
\|\pi(f)\|\leq \|f\|_{\infty, X}
\]
and \textit{completely contractive} if for all $n$ and all $F$ in
$M_n(\mathcal R(X)), \pi^{(n)}(F) := (\pi(F_{i,j}))$ is
contractive. Arveson showed that $\underline T$ dilates to a tuple
$\underline N$ with spectrum in the distinguished boundary of X
($\check{\textup{S}}$ilov boundary related to $\mathcal R(X)$) if
and only if $\pi_{\underline T}$ is completely contractive. Thus
the rational dilation problem can be reformulated. Namely, is
every contractive representation of $\mathcal R(X)$ completely
contractive?

\section{Schur's criterion and geometry of the symmetrized tridisc
}\label{scalarcase}

We begin this section by recalling the symmetrization map
\begin{align}
\nonumber \pi_3:& \mathbb C^3 \rightarrow \mathbb C^3 \\ &
(z_1,z_2,z_3) \rightarrow
(z_1+z_2+z_3,z_1z_2+z_2z_3+z_3z_1,z_1z_2z_3).
\end{align}
Clearly the set $\Gamma_3$ is the image of the closed tridisc
$\overline{\mathbb D^3}$ under the symmetrization map. We shall
mention few useful results about the symmetrized bidisc
$\Gamma_2$, which in literature is denoted by $\Gamma$ also.

\begin{thm}\label{thm:21}
Let $(s,p)\in\mathbb C^2$. Then $(s,p)\in\Gamma_2$ if and only if
$|s|\leq 2$ and $|s-\bar s p|\leq 1-|p|^2$
\end{thm}
For a proof see Theorem 1.1 in \cite{ay-jot}.

\begin{lem}\label{lem:22}
Let $(s,p)\in\Gamma_2$. Then $|s|-|p|\leq 1$.
\end{lem}
\begin{proof}
We have
\[
|s|-|sp|\leq |s-\bar s p|\leq 1-|p|^2
\]
which implies that $|s|(1-|p|)\leq (1+|p|)(1-|p|)$ and hence
\[
|s|-|p|\leq 1\,.
\]

\end{proof}

When a commuting operator triple $(S_1,S_2,P)$ is a scalar triple
it is natural to ask whether $\Gamma_3$ is a spectral set for
$(S_1,S_2,P)$. Given $(s_1,s_2,p)$ one can determine whether
$(s_1,s_2,p)\in\Gamma_3$ by solving the cubic equation
\[
z^3-s_1z^2+s_2z-p=0
\]
and by verifying whether the roots lie within the closed disc
$\overline{\mathbb D}$. But this algorithm needs brute forces and
does not work in higher dimensions. Schur had a more subtle
approach towards the problem of finding the locations of zeros of
a polynomial, \cite{schur}. Schur's theorem is a standard and
effective test to ascertain whether the zeros of a polynomial lie
in the open disc. A simple proof to this result may be found in
\cite{ptak-young}. The question of determining whether $\Gamma_3$
is a spectral set for $(s_1,s_2,p)$ is closely related to Schur's
theorem, we shall see this in Theorems \ref{thm:DB} and \ref{thm:SC1}.\\

For $k\geq 3$ we define two operator pencils $\Phi_{1k},\Phi_{2k}$
for a commuting operator triple. We shall see in the coming
sections that these two operator pencils play pivotal role in
determining the structures of different classes of
$\Gamma_3$-contractions.
\begin{align}
\Phi_{1k}(S_1,S_2,P) &= (k-S_1)^*(k-S_1)-(kP-S_2)^*(kP-S_2) \notag
\\&
\label{eq:1a} =k^2(I-P^*P)+(S_1^*S_1-S_2^*S_2)-k(S_1-S_2^*P)-k(S_1^*-P^*S_2) \\
\Phi_{2k}(S_1,S_2,P) &= (k-S_2)^*(k-S_2)-(kP-S_1)^*(kP-S_1) \notag \\
& \label{eq:1b}
=k^2(I-P^*P)+(S_2^*S_2-S_1^*S_1)-k(S_2-S_1^*P)-k(S_2^*-P^*S_1)
\end{align}
So in particular when $S_1,S_2$ and $P$ are scalars, i.e, points
in $\Gamma_3$, the above two operator pencils take the following
form.
\begin{align}
\Phi_{1k}(s_1,s_2,p) &=
k^2(1-|p|^2)+(|s_1|^2-|s_2|^2)-k(s_1-\bar{s_2}p)-k(\bar{s_1}-\bar{p}s_2) \label{eqn:2a} \\
\Phi_{2k}(s_1,s_2,p) &=
k^2(1-|p|^2)+(|s_2|^2-|s_1|^2)-k(s_2-\bar{s_1}p)-k(\bar{s_2}-\bar{p}s_1).\label{eqn:2b}
\end{align}

Before we proceed to characterize the points of $\Gamma_3$ let us
state two results from \cite{BSR} which we need in sequel.

\begin{lem}\label{lem:BSR1}
If $(s_1,s_2,p)\in \Gamma_3$ then
$(\dfrac{2}{3}s_1,\dfrac{1}{3}s_2)\in\Gamma_2$.
\end{lem}
For a proof see Lemma 2.5 in \cite{BSR}. We now present a set of
characterizations for the points in $\Gamma_3$.

\begin{lem}\label{lem:BSR2}
If $(S_1,S_2,P)$ is a $\Gamma_3$-contraction, then
$(\dfrac{2}{3}S_1,\dfrac{1}{3}S_2)$ is a $\Gamma_2$-contraction.
\end{lem}
See Lemma 2.7 in \cite{BSR} for a proof.

\begin{thm}\label{thm:SC1}
Let $(s_1,s_2,p)\in\mathbb C^3$. Then the following are
equivalent:
\begin{enumerate}
\item $(s_1,s_2,p)\in\Gamma_3$\,; \item $\Gamma_3$ is a complete
spectral set for $(s_1,s_2,p)$\,; \item $\Phi_{ik}(\alpha s_1,
\alpha^2 s_2,\alpha^3 p)\geq 0$ for all
$\alpha\in\overline{\mathbb D}$, $i=1,2$ and $(s_1,s_2,p)$
satisfies $P$\,; \item $\left| {k\alpha^3 p-\alpha^2 s_2}
\right|\leq |{k-\alpha s_1}|$, $\left| {k\alpha^3 p-\alpha s_1}
\right|\leq |{k-\alpha^2 s_2}|$ for all
$\alpha\in\overline{\mathbb D}$ and $(s_1,s_2,p)$ satisfies $P$
\,; \item $|s_1-\bar{s_2}p|+|s_2-\bar{s_1}p|\leq 3(1-|p|^2)$,
$(\dfrac{2}{3}s_1,\dfrac{1}{3}s_2)\in\Gamma_2$ and $(s_1,s_2,p)$
satisfies $P$ \,; \item $|p|\leq 1$ and there exists
$(c_1,c_2)\in\Gamma_2$ such that $s_1=c_1+\bar{c_2}p$ and
$s_2=c_2+\bar{c_1}p$.
\end{enumerate}

\end{thm}

\begin{proof}
We shall prove: $(1)\Rightarrow(6)\Rightarrow (2)\Rightarrow (1)$,
$(1)\Rightarrow (3)\Rightarrow(5)\Rightarrow(6)$
and $(3)\Leftrightarrow(4)$.\\

\noindent $(1)\Rightarrow (6).$ Let $(s_1,s_2,p)\in\Gamma_3$. Then
$|p|\leq 1$ and there are complex numbers $z_1,z_2,z_3$ from the
closed unit disc $\overline{\mathbb D}$ such that
\[
s_1=z_1+z_2+z_3\,,\, s_2=z_1z_2+z_2z_3+z_3z_1 \text{ and }
p=z_1z_2z_3\,.
\]
If $|p|=1$ then by Theorem \ref{thm:DB},
\[
s_1=c_1+\overline{c_2}p \,,\, s_2=c_2+\overline{c_1}p\,,
\]
for some $c_1,c_2$ with $|c_1|+|c_2|\leq 3$. If $|p|<1$, we
consider the polynomials
\[
f(z)=z^3-s_1z^2+s_2z-p \text{ and } f_1(z)= z^2-c_1z+c_2\,.
\]
We show that if the zeros of $f$ lie in $\overline{\mathbb D}$,
then the zeros of $f_1$ also lie in $\overline{\mathbb D}$. Now
\begin{align*}
f(z) & =z^3-s_1z^2+s_2z-p \\&
=z^3-(c_1+\bar{c_2}p)z^2+(c_2+\bar{c_1}p)z-p \,.
\end{align*}
Therefore,
\begin{equation}\label{eqn:001}
|f(z)-zf_1(z)|=|p||c_2\bar{z}^2-c_1\bar{z}+1|\,.
\end{equation}
Taking restriction on $\mathbb T$ we get
\[
|f(z)-zf_1(z)|=|p||z^2-c_1z+c_2|=|p||f_1(z)| = |p||zf_1(z)| <
|zf_1(z)|\,.
\]
So, by Rouch$\acute{e}$ 's Theorem $f$ and $zf_1$ have same number
of zeros inside $\mathbb D$. Now if $f(\omega)=0$ for some $\omega
\in \mathbb T$, then from equation (\ref{eqn:001}) we get
\[
|f_1(\omega)|=|p||f_1(\omega)|
\]
which by virtue of $|p|<1$ implies that $f_1(\omega)=0$. Thus, if
the zeros of $f$ lie in $\overline{\mathbb D}$, then the zeros of
$f_1$ also lie in $\overline{\mathbb D}$. Since here the zeros of
$f$ lie in $\overline{\mathbb D}$ as $(s_1,s_2,p)\in \Gamma_3$,
the zeros of $f_1$ are also in $\overline{\mathbb D}$. Therefore,
there exist $\alpha_1, \alpha_2 \in \overline{\mathbb D}$ such
that
\[
\alpha_1+\alpha_2 =c_1 \,,\; \; \alpha_1\alpha_2 =c_2\,.
\]
Hence $(c_1,c_2)\in\Gamma_2$.\\

\noindent $(6)\Rightarrow (2).$ This follows from a result of a
subsequent section. Indeed, in Theorem \ref{sufficient1}, we shall
see that if a commuting operator pair $(F_1,F_2)$ is almost normal
and that
\begin{align*}
S_1-S_2^*P &=D_PF_1D_P \;,\\
S_2-S_1^*P &=D_PF_2D_P\;,
\end{align*}
then $\Gamma_3$ is a complete spectral set for $(S_1,S_2,P)$
provided that $\|S_i\|\leq 3$ for $i=1,2$. It is evident that in
scalar case when $s_1=c_1+\bar{c_2}p$ and $s_2=c_2+ \bar{c_1}p$,
then $|s_i|\leq 3$ as $(c_1,c_2)\in\Gamma_2$ and $|c_1|+|c_2|\leq
3$. Also $(c_1,c_2)$ plays the role of the fundamental operator
pair $(F_1,F_2)$ as in Theorem \ref{sufficient1} to make
$\Gamma_3$ a complete spectral set for $(s_1,s_2,p)$. Because in
this scalar case the existence of the FOP (which is $(c_1,c_2)$)
needs no proof and thus we can construct the explicit
$\Gamma_3$-unitary dilation which is described in Theorem
\ref{main-dilation-theorem}. Hence
$(6)\Rightarrow (2)$.\\

\noindent $(2)\Rightarrow (1).$ Obvious.\\

\noindent $(1)\Rightarrow(3).$ Let $(s_1,s_2,p)$ be a point in
$\Gamma_3$. Then $(\alpha s_1,{\alpha}^2s_2,{\alpha}^3p)\in\mathbb
G_3$, for any $\alpha$ in the unit disc. Let us consider the
polynomial
\[
f(z)=z^3-\alpha s_1z^2+{\alpha}^2s_2z-{\alpha}^3p.
\]
If $z_1,z_2,z_3$ are the roots of $f(z)=0$ then
\[
\sum_{i=1}^3z_i=\alpha s_1\,,\, \sum_{1\leq i<j\leq
3}z_iz_j=\alpha^2 s_2 \quad \text{and } \prod_{i=1}^3z_i =
\alpha^3 p \,.
\]
This is same as saying that $\pi_3(z_1,z_2,z_3)=(\alpha
s_1,{\alpha}^2s_2,{\alpha}^3p)$ and thus $(z_1,z_2,z_3)\in \mathbb
D^3$ as $(\alpha s_1,{\alpha}^2s_2,{\alpha}^3p)\in \mathbb G_3$.
Therefore, by Schur's theorem (see \cite{ptak-young}), the matrix
\[
H=f_*(A)^*f_*(A)-f(A)^*f(A) >0 \,,
\]
where
\begin{gather*}
f_*(z)=-{\bar{\alpha}}^3\bar p
z^3+{\bar{\alpha}}^2\bar{s_2}z^2-\bar{\alpha}\bar{s_1}z+1
\\
\text{ and } \quad A=\begin{pmatrix}
0 & 1 & 0 \\
0 & 0 & 1 \\
0 & 0 & 0
\end{pmatrix}.
\end{gather*}
A simple computation shows that
\[
f(A)=\begin{pmatrix}
-{\alpha}^3p & {\alpha}^2s_2 & -{\alpha}s_1 \\
0 & -{\alpha}^3p & {\alpha}^2s_2 \\
0 & 0 & -{\alpha}^3p
\end{pmatrix}
\text{ and } f_*(A)=
\begin{pmatrix}
1 & -\bar{\alpha}\bar{s_1} & {\bar{\alpha}}^2\bar{s_2} \\
0 & 1 & -\bar{\alpha}\bar{s_1} \\
0 & 0 & 1
\end{pmatrix}.
\]
So we have
\begin{align*}
f(A)^*f(A) &= \begin{pmatrix}
-{\bar{\alpha}}^3\bar{p} & 0 & 0 \\
{\bar{\alpha}}^2\bar{s_2} & -{\bar{\alpha}}^3\bar{p} & 0 \\
-\bar{\alpha}\bar{s_1} & {\bar{\alpha}}^2\bar{s_2} &
-{\bar{\alpha}}^3\bar{p}
\end{pmatrix}
\begin{pmatrix}
-{\alpha}^3p & {\alpha}^2s_2 & -{\alpha}s_1 \\
0 & -{\alpha}^3p & {\alpha}^2s_2 \\
0 & 0 & -{\alpha}^3p
\end{pmatrix} \\
&=\left(
\begin{array}{ccc}
|\alpha^3p|^2 & -|\alpha|^4\bar{\alpha}s_2\bar{p} & |\alpha|^2{\bar{\alpha}}^2s_1\bar{p} \\
-|\alpha|^4\alpha \bar{s_2}p & |\alpha^3p|^2+|\alpha^2s_2|^2 & -|\alpha|^2\bar{\alpha}s_1\bar{s_2}-|\alpha|^4\bar{\alpha}s_2\bar{p} \\
|\alpha|^2\alpha^2\bar{s_1}p & -|\alpha|^2\alpha
\bar{s_1}s_2-|\alpha|^4\alpha \bar{s_2}p &
|\alpha^3p|^2+|\alpha^2s_2|^2+|\alpha s_1|^2
\end{array}
\right)
\end{align*}
and
\begin{align*}
f_*(A)^*f_*(A) &=
\begin{pmatrix}
1 & 0 & 0 \\
-\alpha s_1 & 1 & 0 \\
{\alpha}^2s_2 & -\alpha s_1 & 1
\end{pmatrix}
\begin{pmatrix}
1 & -\bar{\alpha}\bar{s_1} & {\bar{\alpha}}^2\bar{s_2} \\
0 & 1 & -\bar{\alpha}\bar{s_1} \\
0 & 0 & 1
\end{pmatrix} \\
&=\begin{pmatrix}
1 & -\bar{\alpha}\bar{s_1} & \bar{\alpha}^2\bar{s_2} \\
-\alpha s_1 & 1+|\alpha s_1|^2 & -\bar{\alpha}\bar{s_1}-|\alpha|^2\bar{\alpha}s_1\bar{s_2} \\
\alpha^2s_2 & -\alpha s_1 -|\alpha|^2\alpha \bar{s_1}s_2 &
1+|\alpha s_1|^2+|\alpha^2s_2|^2
\end{pmatrix}.
\end{align*}
Therefore,
\begin{align*}
H &= f_*(A)^*f_*(A)-f(A)^*f(A) \\
& = \small \begin{pmatrix}
1-|\alpha^3 p|^2 & -\bar{\alpha}(\bar{s_1}-|\alpha|^4s_2\bar{p}) & \bar{\alpha}^2(\bar{s_2}-|\alpha|^2s_1\bar{p}) \\
-\alpha(s_1-|\alpha|^4\bar{s_2}p) & (1-|\alpha^3p|^2)+(|\alpha s_1|^2-|\alpha^2 s_2|^2) & -\bar{\alpha}(\bar{s_1}-|\alpha|^4s_2\bar{p}) \\
\alpha^2(s_2-|\alpha|^2\bar{s_1}p) &
-\alpha(s_1-|\alpha|^4\bar{s_2}p) & 1-|\alpha^3 p|^2
\end{pmatrix}.
\end{align*}
We introduce few notations which are same as they are in Horn and
Johnsons' Matrix Analysis, \cite{horn}. Let $A\in M_{m,n}(\mathbb
C)$. For index sets $\Lambda \subseteq \{1,\hdots,m \}$ and
$\Omega \subseteq \{1,\hdots,n \}$, we denote by $A(\Lambda,
\Omega)$ the submatrix of $A$ that lies in the rows of $A$ indexed
by $\Lambda$ and the columns indexed by $\Omega$. If $\Lambda
=\Omega$, the square submatrix $A(\Lambda,\Lambda)$ is called a
{\em principal submatrix} of $A$ and is abbreviated $A(\Lambda)$.
The determinant of a square submatrix of $A$ is called a {\em
minor} of $A$. If the submatrix is a principal submatrix, then the
minor is a {\em principal minor}.

It is well known that if $A$ is positive definite, then all
principal minors of $A$ are positive. See Theorem 7.2.5 in
\cite{horn}, for a more general and finer version of this
statement. Now since $H$ is positive definite, the principal minor
obtained from the principal submatrix $H(\{1,3\})$ is positive.
So, we have
\[
\det
\begin{pmatrix}
1-|\alpha^3p|^2 & \bar{\alpha}^2(\bar{s_2}-|\alpha|^2\bar p s_1)\\
\alpha^2(s_2-|\alpha|^2\bar{s_1}p) & 1-|\alpha^3p|^2
\end{pmatrix} >0\,.
\]
If we denote
\begin{equation*}
H(\{1,3\}) = \begin{pmatrix} a & c\\
\bar c & a
\end{pmatrix},  \quad \text{where }
\begin{cases}
a &= 1-|\alpha^3p|^2 \\
c & = \bar{\alpha}^2(\bar{s_2}-|\alpha|^2\bar p s_1)
\end{cases}
\end{equation*}
then $a>|c|$ as $a>0$ and with $k\geq 3$ it implies that

\begin{equation}\label{eq:4a}
(k^2-3)a >2k|c| \,.
\end{equation}

Let us denote $m=|\alpha s_1|^2-|\alpha^2 s_2|^2$. Our goal is to
show that for all $k\geq 3$
\[
k^2a-m>2k|c|\,.
\]
We assume here that $m\geq 0$ because otherwise $k^2a-m>2k|c|$ is
obvious. We first show that $3a-m\geq 0$. We apply $(1)\Rightarrow
(6)$ to get $(c_1,c_2)\in\Gamma_2$ such that
\[
\alpha s_1=c_1+\bar{c}_2(\alpha^3 p)\;,\; \alpha^2
s_2=c_2+\bar{c}_1(\alpha^3p)\,.
\]
Now
\begin{align*}
m &= |\alpha s_1|^2-|\alpha^2 s_2|^2 \\
&= |c_1+\bar{c}_2(\alpha^3 p)|^2-|c_2+\bar{c}_1(\alpha^3p)|^2 \\
&=(|c_1|^2+|c_2(\alpha^3 p)|^2+c_1c_2(\overline{\alpha^3 p})+\bar{c}_1\bar{c}_2(\alpha^3 p)) \\
&\quad - (|c_2|^2+|c_1(\alpha^3 p)|^2+c_1c_2(\overline{\alpha^3
p})+\bar{c}_1\bar{c}_2(\alpha^3 p))\\
&= (|c_1|^2-|c_2|^2)(1-|\alpha^3 p|^2)\\
&=(|c_1|+|c_2|)(|c_1|-|c_2|)(1-|\alpha^3 p|^2)\\
& \leq (|c_1|+|c_2|)(1-|\alpha^3 p|^2) \\
& \leq 3a \,.
\end{align*}
The last two inequalities follow from Lemma \ref{lem:22} and
Theorem \ref{thm:21} respectively as $(c_1,c_2)\in\Gamma_2$. Thus
$3a-m\geq 0$ and consequently from (\ref{eq:4a}) we have that
\begin{equation}\label{eq:5}
(k^2-3)a+(3a-m)> 2k|c|\,,
\end{equation}
for all $k\geq 3$. Therefore
\[
k^2a-m > 2k|c|.
\]
Now using the fact that
\begin{equation}\label{eq:4}
x>|y|\Leftrightarrow x> \text{Re } \omega y\,, \quad \text{ for
all }\omega\in\mathbb T\,,
\end{equation}
we have that
\begin{align}
k^2a-m & > 2k \text{ Re } \omega c \notag \\
& \label{eq:6} = k \omega c +k \bar{\omega}\bar c, \quad \text{
for all } \omega\in\mathbb T.
\end{align}
Choosing $\omega=1$ and substituting the values of $a,m,c$ in
(\ref{eq:6}) we get
\begin{align*}
\Phi_{2k}(\alpha s_1,\alpha^2 s_2, \alpha^3 p) &= k^2(1-|\alpha^3
p|^2)+(|\alpha^2 s_2|^2-|\alpha s_1|^2)\\
& \quad
-k\alpha^2(s_2-|\alpha|^2\bar{s_1}p)-k\bar{\alpha}^2(\bar{s_2}-|\alpha|^2\bar
p s_1) \\& \; >0 .
\end{align*}
By continuity, we have that $\Phi_{2k}(\alpha s_1, \alpha^2
s_2,\alpha^3 p)\geq 0$, for all $\alpha \in\overline{\mathbb
D}$.\\

\noindent We now show that
\[
\Phi_{1k}(\alpha s_1, \alpha^2 s_2,\alpha^3 p)\geq 0, \text{ for
all } \alpha \in\overline{\mathbb D}.
\]
As in the case of $\Phi_{2k}$, here our target is to establish
\[
k^2a+m \geq 2k|b|\, , \text{ where } b= -\bar
\alpha(\bar{s_1}-|\alpha|^4s_2\bar p).
\]
We have that
\[
m=(|c_1|^2-|c_2|^2)a \text{ and } |b| = |c_1|a.
\]
The expression of $m$ follows from the proof of $m\leq 3a$ and
\begin{align*}
|b| &=|\alpha s_1 -\overline{(\alpha^2 s_2)}(\alpha^3 p)| \\
& = |(c_1+\bar c_2 {\alpha^3 p}) - (\bar c_2 + c_1
\overline{\alpha ^3 p})\alpha^3 p| \\
& = |c_1|(1- |\alpha^3 p|^2)\\
& = |c_1|a.
\end{align*}
Therefore, we have
\begin{align*}
k^2a+m-2k|b| &= k^2a+(|c_1|^2-|c_2|^2)a - 2k|c_1|a \\
&= \{ (k-|c_1|)^2-|c_2|^2 \}a \\
& = (k-|c_1|+|c_2|)(k-|c_1|-|c_2|)a \\
& \geq 0.
\end{align*}
The last inequality follows from the facts that $a>0, k\geq 3$ and
$|c_1|+|c_2| \leq 3$. Therefore,
\[
k^2a+m\geq 2k|b|.
\]
From here we have that

\begin{align}
k^2a+m & \geq 2k \text{ Re } \omega b \notag \\
& \notag = k \omega b +k \bar{\omega}\bar b, \quad \text{ for all
} \omega\in\mathbb T.
\end{align}
Choosing $\omega=1$ and substituting the values of $a,m,b$ we get
\begin{align*}
\Phi_{1k}(\alpha s_1,\alpha^2 s_2, \alpha^3 p) &= k^2(1-|\alpha^3
p|^2)+(|\alpha s_1|^2-|\alpha^2 s_2|^2)\\
& \quad -k(\alpha s_1-\overline{(\alpha^2 s_2)}(\alpha^3
p))-k\overline{(\alpha s_1-\overline{(\alpha^2 s_2)}(\alpha^3 p))}
\\& \;
\geq 0 .
\end{align*}
By continuity, we have that $\Phi_{1k}(\alpha s_1, \alpha^2
s_2,\alpha^3 p)\geq 0$, for all $\alpha \in\overline{\mathbb
D}$.\\

\noindent $(3)\Rightarrow (5).$ First we assume that $s_1\neq
\bar{s_2}p$ and $s_2\neq \bar{s_1}p$. For $\omega_1 ,\omega_2
\in\mathbb T$, we have
\[
\Phi_{1k}(\omega_1 s_1,{\omega_1}^2
s_2,{\omega_1}^3p)+\Phi_{2k}(\omega_2 s_1,{\omega_2}^2
s_2,{\omega_2}^3p) \geq 0
\]
that is
\[
k^2(1-|p|^2)\geq \text{Re
}k[\omega_1(s_1-\bar{s_2}p)+{\omega_2}^2(s_2-\bar{s_1}p)].
\]
Since $s_1\neq \bar{s_2}p$ and $s_2\neq \bar{s_1}p$, we choose
\[
\omega_1 =\frac{|s_1-\bar{s_2}p|}{s_1-\bar{s_2}p} \text{ and }
\omega_2=\sqrt{\frac{|s_2-\bar{s_1}p|}{s_2-\bar{s_1}p}}
\]
and get
\[
k(1-|p|^2)\geq |s_1-\bar{s_2}p|+|s_2-\bar{s_1}p|\,,\quad k\geq
3\,.
\]
If anyone or both of $s_1= \bar{s_2}p$ and $s_2= \bar{s_1}p$ hold
then the above inequality is obvious.\\

\noindent $(5)\Rightarrow (6).$ Let $(5)$ holds. Then $|p|\leq 1$
as $k> 0$. If $|p|=1$ then $s_1=\bar{s_2}p$ and $s_2=\bar{s_1}p$.
We choose $c_1=s_1$ and $c_2=0$. Clearly $|c_1|+|c_2|=|s_1|\leq 3$
and
\[
s_1=c_1+\bar{c_2}p \quad \text{ and }\quad s_2=c_2+ \bar{c_1}p\,.
\]
When $|p|<1$ we choose
\[
c_1=\frac{s_1-\bar{s_2}p}{1-|p|^2} \quad \text{ and } \quad
c_2=\frac{s_2-\bar{s_1}p}{1-|p|^2}\,.
\]
It is evident that $s_1=c_1+\bar{c_2}p$ and $s_2=c_2+ \bar{c_1}p$.
Also since $(5)$ holds for $k=3$, we have
\[
|c_1|+|c_2|\leq 3.
\]
\\

\noindent $(3)\Leftrightarrow (4).$ From (\ref{eq:1a}) we have
that
\[
\Phi_{1k}(\alpha s_1,\alpha^2 s_2,\alpha^3 p) =(k-\bar{\alpha}
\bar{s_1})(k-\alpha s_1)-(k\bar{\alpha}^3\bar
p-\bar{\alpha}^2\bar{s_2})(k\alpha^3 p-\alpha^2 s_2).
\]
Therefore,
\begin{align*}
&\Phi_{1k}(\alpha s_1,\alpha^2 s_2,\alpha^3 p) \geq 0 \\
\Leftrightarrow &(k-\bar{\alpha} \bar{s_1})(k-\alpha
s_1)-(k\bar{\alpha}^3\bar p-\bar{\alpha}^2\bar{s_2})(k\alpha^3
p-\alpha^2 s_2) \geq 0 \\
\Leftrightarrow &\left | {k\alpha^3 p-\alpha^2 s_2} \right | \leq
|{k-\alpha s_1}|\,.
\end{align*}
The proof of
\[
\Phi_{2k}(\alpha s_1,\alpha^2 s_2,\alpha^3 p) \geq 0
\Leftrightarrow \left | {k\alpha^3 p-\alpha^2 s_1} \right | \leq
|{k-\alpha s_2}|
\]
is similar.

\end{proof}

\section{$\Gamma_3$-contractions and their fundamental operator
pairs}\label{basicproperties}

We recall that a commuting triple $(S_1,S_2,P)$ for which
$\Gamma_3$ is a spectral set is called a $\Gamma_3$-contraction.
The following result shows that the definition of
$\Gamma_3$-contraction can be made more precise by using the
polynomial convexity of $\Gamma_3$.
\begin{lem}
A commuting triple of bounded operators $(S_1,S_2,P)$ is a
$\Gamma_3$-contraction if and only if \begin{equation} \label{pT}
\| p (S_1,S_2,P) \| \leq \| p \|_{\infty, \;
\Gamma_3}\end{equation} for all holomorphic polynomial $p$ in
three variables.
\end{lem}

\begin{proof}

If $(S_1,S_2,P)$ is a $\Gamma_3$-contraction, then of course
(\ref{pT}) just follows from definition.

The converse proof can be easily done by using polynomial
convexity of $\Gamma_3$. Indeed, if $\sigma_T(S_1,S_2,P)$ is not
contained in $\Gamma_3$, then there is a point $(s_1, s_2, p)$ in
$\sigma_T (S_1,S_2,P)$ that is not in $\Gamma_3$. By polynomial
convexity of $\Gamma_3$, there is a polynomial $p$ such that $ |
p(s_1, s_2, p) | > \| p \|_{\infty, \Gamma_3}$. By polynomial
spectral mapping theorem,
\[
\sigma_T (p(S_1,S_2,P)) = \{ p(s_1, s_2, p) : (s_1, s_2, p) \in
\sigma_T (S_1,S_2,P) \}
\]
and hence the spectral radius of $p(S_1,S_2,P)$ is bigger than $\|
p \|_{\infty, \Gamma_3}$. But then
\[
\| p(S_1,S_2,P )\|> \| p \|_{\infty, \Gamma_3},
\]
contradicting the fact that $\Gamma_3$ is a spectral set for
$(S_1,S_2,P)$.
\end{proof}

Unlike scalars, the symmetrization of any three commuting
contractions may not be a $\Gamma_3$-contraction. See Remark 2.11
in \cite{BSR} for an example. Here we provide few properties of
$\Gamma_3$-contractions.

\begin{lem}
If $(S_1,S_2,P)$ is a $\Gamma_3$-contraction then $\|S_i\|\leq 3$
for $i=1,2$ and $\|P\|\leq 1$.
\end{lem}
\begin{proof}
This follows from the definition of $\Gamma_3$-contraction if we
consider the co-ordinate polynomials.
\end{proof}

\begin{prop}
If $(S_1,S_2,P)$ is a $\Gamma_3$-contraction then so are
$(S_2,S_1,P)$ and $(S_1^*,S_2^*,P^*)$.
\end{prop}
\begin{proof}
We know that $(S_1,S_2,P)$ is a $\Gamma_3$-contraction if
\[
\|f(S_1,S_2,P)\|\leq \|f(\boldsymbol
z)\|_{\infty,\Gamma_3}=\displaystyle \sup_{\boldsymbol z \in
\Gamma_3}|f(\boldsymbol z)|,
\]
for every polynomial $f$ in $3$-variables. Let $f$ be a polynomial
in $z_1,z_2,z_3$ and let $g(z_1,z_2,z_3)=f(z_2,z_1,z_3)$. Then by
virtue of $\Gamma_3$ being a symmetric domain, we have
\[
\|f\|_{\infty, \Gamma_3}=\|g\|_{\infty, \Gamma_3}\,.
\]
Therefore,
\[
\|f(S_2,S_1,P)\|=\|g(S_1,S_2,P)\|\leq \|g\|_{\infty,
\Gamma_3}=\|f\|_{\infty, \Gamma_3}\,
\]
and consequently $(S_2,S_1,P)$ is a $\Gamma_3$-contraction. Again,
let $p$ be an arbitrary polynomial in $3$-variables. Then the
polynomial $p_*$ whose coefficients are the conjugates of the
corresponding coefficients of $p$ has the same supremum norm as
that of $p$ over $\Gamma_3$. Clearly
\[
\|p(S_1^*,S_2^*,P^*)\| =\|(p_*(S_1,S_2,P))^*\| =\|p_*(S_1,S_2,P)\|
\leq \|p_*\|_{\infty, \Gamma_3} =\|p\|_{\infty, \Gamma_3}.
\]
Hence $(S_1^*,S_2^*,P^*)$ is a $\Gamma$-contraction.
\end{proof}

\begin{prop}\label{prop:POS}
Let $(S_1,S_2,P)$ be a $\Gamma_3$-contraction. Then for $i=1,2$
and for all $\alpha \in \overline{\mathbb D}$, $ \Phi_{ik}(\alpha
S_1,{\alpha}^2S_2,{\alpha}^3P)\geq 0 $ for all $k\geq 3$.
\end{prop}
\begin{proof}
Since $(S_1,S_2,P)$ is a $\Gamma_3$-contraction,
$\sigma_T(S_1,S_2,P)\subseteq \Gamma_3$. Let $f$ be a holomorphic
function in a neighbourhood of $\Gamma_3$. Since $\Gamma_3$ is
polynomially convex, by Oka-Weil Theorem (see \cite{Gamelin},
Theorem 5.1) there is a sequence of polynomials $\{p_n\}$ in
$3$-variables such that $p_n\rightarrow f$ uniformly over
$\Gamma_3$. Therefore, by Theorem 9.9 of CH-III in
\cite{vasilescu},
\[
p_n(S_1,S_2,P)\rightarrow f(S_1,S_2,P)
\]
which by the virtue of $(S_1,S_2,P)$ being a
$\Gamma_3$-contraction implies that
\[
\| f(S_1,S_2,P) \|=\displaystyle \lim_{n\rightarrow \infty}\|
p_n(S_1,S_2,P) \|\leq \displaystyle \lim_{n\rightarrow
\infty}\|p_n\|_{\infty,\Gamma_3}=\|f\|_{\infty,\Gamma_3}.
\]
We fix $\alpha \in \mathbb D$ and choose
\[
f(s_1,s_2,p)=\frac{k\alpha^3p-\alpha^2s_2}{k-\alpha s_1}\,.
\]
Since $k\geq 3$, $f$ is well-defined and is holomorphic in a
neighborhood of $\Gamma_3$ and has norm not greater than $1$, by
part-(5) of Theorem \ref{thm:SC1}. So we get
\[
\|(k\alpha^3P-\alpha^2S_2)(k-\alpha S_1)^{-1} \|\leq
\|f\|_{\infty,\Gamma_3}\leq 1.
\]
Thus
\[
{(k-\alpha
S_1)^*}^{-1}(k\alpha^3P-\alpha^2S_2)^*(k\alpha^3P-\alpha^2S_2)(k-\alpha
S_1)^{-1}\leq I
\]
which implies and is implied by
\[
(k-\alpha S_1)^*(k-\alpha S_1)\geq
(k\alpha^3P-\alpha^2S_2)^*(k\alpha^3P-\alpha^2S_2).
\]
By the definition of $\Phi_{1k}$, this is same as saying that
$\Phi_{1k}(\alpha S_1,\alpha^2 S_2,\alpha^3 P)\geq 0$ for all
$\alpha \in \mathbb D$. By continuity we have that
\[
\Phi_{1k}(\alpha S_1,\alpha^2 S_2,\alpha^3P)\geq 0 \quad \text{
for all } \alpha\in \overline{\mathbb D}.
\]
The proof of $\Phi_{2k}(\alpha S_1,\alpha^2S_2,\alpha^3P)\geq 0$
is similar.
\end{proof}

Let $S_1,S_2,P$ be commuting operators on a Hilbert space
$\mathcal H$ with $\|P\|\leq 1$. We denote by $D_P$ and $\mathcal
D_P$ the defect operator $(I-P^*P)^{\frac{1}{2}}$ and its range
closure respectively. The \textit{fundamental equations} for the
triple $(S_1,S_2,P)$ are defined in the following way:
\begin{equation}\label{fundeqn}
S_1-S_2^*P=D_PX_1D_P \; \textup{ and } S_2-S_1^*P=D_PX_2D_P, \quad
X\in \mathcal L(\mathcal D_P)
\end{equation}
We shall see below that when $(S_1,S_2,P)$ is a
$\Gamma_3$-contraction the fundamental equations have unique
solutions which together will be called the \textit{fundamental
operator pair} of $(S_1,S_2,P)$.\\

Let us recall that the {\em numerical radius} of an operator $A$
on a Hilbert space $\mathcal{H}$ is defined by
\[
\omega(A) = \sup \{|\langle Ax,x \rangle|\; : \;
\|x\|_{\mathcal{H}}= 1\}.
\]
It is well known that
\begin{eqnarray}\label{nradius}
r(A)\leq \omega(A)\leq \|A\| \textup{ and } \frac{1}{2}\|A\|\leq
\omega(A)\leq \|A\|, \end{eqnarray} where $r(A)$ is the spectral
radius of $A$. We state a basic lemma on numerical radius whose
proof is a routine exercise. We shall use this lemma in sequel.

\begin{lem} \label{basicnrlemma}
The numerical radius of an operator $A$ is not greater than $n$ if
and only if  Re $\alpha A \leq nI$ for all complex numbers
$\alpha$ of modulus $1$.
\end{lem}

\begin{lem}\label{basicnrlemma1}
Let $A_1,A_2$ be two bounded operators such that
$\omega(A_1+A_2z)\leq n$ for all $z\in\mathbb T$. Then
$\omega(A_1+zA_2^*)\leq n$ and $\omega(A_1^*+A_2z)\leq n$ for all
$z\in\mathbb T$.

\end{lem}

\begin{proof}

We have that $\omega(A_1+zA_2)\leq n$ for every $z\in\mathbb T$,
which is same as saying that $\omega(z_1A_1+z_2A_2)\leq n$ for all
complex numbers $z_1,z_2$ of unit modulus. Thus by Lemma
\ref{basicnrlemma},
\[
(z_1A_1+z_2A_2)+(z_1A_1+z_2A_2)^*\leq 2nI,
\]
that is
\[
(z_1A_1+\bar{z_2}A_2^*)+(z_1A_1+\bar{z_2}A_2^*)^* \leq 2nI.
\]
Therefore, $z_1(A_1+zA_2^*)+\bar{z_1}(A_1+zA_2^*)^*\leq 2nI$ for
all $z,z_1 \in\mathbb T$. This is same as saying that
\[
\text{Re }z_1(A_1+zA_2^*)\leq nI, \textup{ for all } z,z_1
\in\mathbb T.
\]
Therefore, by Lemma \ref{basicnrlemma} again
$\omega(A_1+zA_2^*)\leq n$ for any $z$ in $\mathbb T$. The proof
of $\omega(A_1^*+A_2z)\leq n$ is similar.

\end{proof}

\begin{lem}\label{vital}
Let $A_1,A_2$ be two bounded operators on a Hilbert space
$\mathcal H$. Then the following are equivalent:

\begin{enumerate}
\item $(A_1,A_2)$ is almost normal ; \item $A_1^*+A_2z$ and
$A_2^*+A_1z$ commute for every $z$ of unit modulus ; \item
$A_2^*+A_1z$ is a normal operator for every $z\in\mathbb T$ ;
\item $A_1^*+A_2z$ is a normal operator for every $z\in\mathbb T$.
\end{enumerate}

\end{lem}

\begin{proof}
$(1)\Leftrightarrow (2)$. Since $(A_1,A_2)$ is almost  normal, we
have that
\[
[A_1,A_2]=0 \text{ and } [A_1^*,A_1]=[A_2^*,A_2]\;,
\]
where $[A_1,A_2]$ is the commutator $A_1A_2-A_2A_1$. Again
$A_1^*+A_2z$ and $A_2^*+A_1z$ commute if and only if for every
$z\in\mathbb T$,
\[
[A_1^*,A_2^*]+([A_1^*,A_1]-[A_2^*,A_2])z+([A_2,A_1])z^2=0\,.
\]
Therefore, $(1)\Rightarrow (2)$ follows. Again putting $z=\pm 1$
and $z=\pm i$ we get $[A_1,A_2]=0$ which further implies that
$[A_1^*,A_1]=[A_2^*,A_2]$. Hence $(2)\Rightarrow (1)$.\\

The proofs of $(1)\Leftrightarrow (3)$ and $(1)\Leftrightarrow
(4)$ are similar.

\end{proof}

\begin{thm}{\bf (Existence and
Uniqueness).}\label{existence-uniqueness} Let $(S_1,S_2,P)$ be a
$\Gamma_3$-contraction on a Hilbert space $\mathcal H$. Then there
are unique operators $F_1,F_2\in\mathcal L(\mathcal D_P)$ such
that $S_1-S_2^*P=D_PF_1D_P$ and $S_2-S_1^*P=D_PF_2D_P $. Moreover,
$\omega (F_1+F_2z)\leq 3$ for all $z\in\mathbb T$.
\end{thm}
\begin{proof}
We apply Proposition \ref{prop:POS} for $k\geq 3$ to $(S_1,S_2,P)$
to get
\begin{gather}
\Phi_{1k}(\alpha S_1,\alpha^2S_2,\alpha^3P)\geq 0 \,,\label{eqn11}\\
\Phi_{2k}(\alpha S_1,\alpha^2S_2,\alpha^3P)\geq0. \label{eqn12}
\end{gather}
for all $\alpha\in\overline{\mathbb D}$. Therefore, in particular
for $\omega\in\mathbb T$ we have from (\ref{eqn11}) and
(\ref{eqn12}) respectively
\begin{align*}
&
k^2(I-P^*P)+(S_1^*S_1-S_2^*S_2)-k\beta(S_1-S_2^*P)-k\bar{\beta}(S_1^*-P^*S_2)\geq0
\,,
\\&
k^2(I-P^*P)+(S_2^*S_2-S_1^*S_1)-k\beta^2(S_2-S_1^*P)-k\bar{\beta}^2(S_2^*-P^*S_1)\geq0.
\end{align*}
On addition we get
\begin{align}
2k(I-P^*P) & \geq \beta(S_1-S_2^*P)+\bar{\beta}(S_1^*-P^*S_2)
\notag \\& \label{10a} \quad +
{\beta}^2(S_2-S_1^*P)+\bar{\beta}^2(S_2^*-P^*S_1).
\end{align}
This shows that the Laurent polynomial
\begin{align}
\xi(z)=& 2k(I-P^*P)-\beta(S_1-S_2^*P)-\bar{\beta}(S_1^*-P^*S_2)
\notag \\& \label{11a}
-{\beta}^2(S_2-S_1^*P)-\bar{\beta}^2(S_2^*-P^*S_1)
\end{align}
is non-negative for all $z\in\mathbb T$. Therefore, by Operator
Fejer-Riesz Theorem (see Theorem 1.2 in \cite{DW}) there is a
polynomial of degree $2$ say $P(z)=X_0+X_1z+X_2z^2$ such that for
all $z\in\mathbb T$,
\begin{align}
\xi(z)=p(z)^*p(z)
&=(X_0^*+X_1^*\bar{z}+X_2^*\bar{z}^2)(X_0+X_1z+X_2z^2)\notag \\&=
(X_0^*X_0+X_1^*X_1+X_2^*X_2)+(X_0^*X_1+X_1^*X_2)z \notag \\
& \label{11b} \quad
+(X_1^*X_0+X_2^*X_1)\bar{z}+(X_0^*X_2)z^2+(X_2^*X_0)\bar{z}^2.
\end{align}
Comparing (\ref{11a}) and (\ref{11b}) we get
\begin{gather}
\label{41} 2kD_P^2 =X_0^*X_0+X_1^*X_1+X_2^*X_2 \\
\label{42} S_1-S_2^*P =-(X_0^*X_1+X_1^*X_2) \\
\label{43} S_2-S_1^*P =-(X_0^*X_2).
\end{gather}
From (\ref{41}) we get
\[
2kD_P^2\geq X_0^*X_0, \; 2kD_P^2\geq X_1^*X_1 \text{ and }
2kD_P^2\geq X_2^*X_2\,.
\]
Therefore by Douglas's lemma (see Lemma 2.1 in \cite{DMP}) there
are contractions $Z_0,Z_1,Z_2$ such that
\[
X_0^*=\sqrt{2k}D_PZ_0,\; X_1^*=\sqrt{2k}D_PZ_1 \text{ and }
X_2^*=\sqrt{2k}D_PZ_2.
\]
Putting these values in (\ref{42}) and in (\ref{43}) we get
\begin{align*}
S_1-S_2^*P &=D_P[-2k(Z_0Z_1^*+Z_1Z_2^*)]D_P \,,\\
S_2-S_1^*P &=D_P(-2kZ_0Z_2^*)D_P.
\end{align*}
We denote
\[
F_1=P_{\mathcal D_P}[-2k(Z_0Z_1^*+Z_1Z_2^*)]|_{\mathcal D_P} \;,\;
F_2=P_{\mathcal D_P}(-2kZ_0Z_2^*)|_{\mathcal D_P}\,.
\]
It is now evident that $F_1,F_2$ are solutions to the equations
$S_1-S_2^*P=D_PX_1D_P$ and $S_2-S_1^*P=D_PX_2D_P$ respectively.\\

\noindent {\bf Uniqueness.} Let there be two solutions $F_1,G_1$
of the equation $S_1-S_2^*P=D_PX_1D_P$. Then $D_P(F_1-G_1)D_P=0$
which shows that $F_1-G_1=0$ as $F_1-G_1$ is defined on $\mathcal
D_P$. Thus $F_1$ is unique and similarly $F_2$.\\

\noindent From (\ref{10a}) we have that
\begin{align*}
2kD_P^2& \geq 2 \text{ Re }\beta[(S_1-S_2^*P)+\beta (S_2-S_1^*P)]
\\& =2 \text{ Re }\beta[D_PF_1D_P+\beta D_PF_2D_P].
\end{align*}
Therefore,
\[
D_P^2\geq \text{ Re } \beta D_PF(\beta)D_P\,,
\]
where $F(\beta)=\dfrac{1}{k}(F_1+\beta F_2)$. So we have
\[
D_P[I_{\mathcal D_P}-\text{Re }\beta F(\beta)]D_P \geq 0\,.
\]
This implies that
\[
I_{\mathcal D_P}-\text{Re }\beta F(\beta) \geq 0
\]
because $F(\beta)$ is defined on $\mathcal D_P$. Therefore, Re
$\beta F(\beta)\leq I_{\mathcal D_P}$ and consequently by Lemma
\ref{basicnrlemma}, we have
\[
\omega(F(\beta))\leq 1\,.
\]
This implies that
\[
\omega(F_1+F_2z)\leq k \quad \text{ for all } z\in\mathbb T \text{
and for all } k\geq 3 \,.
\]
Therefore,
\[
\omega(F_1+F_2z)\leq 3 \quad \text{ for all } z\in\mathbb T \,.
\]

\end{proof}

\begin{rem}
We shall see in the next section a partial converse to the
existence-uniqueness of FOP. If an almost normal operator pair
$(F_1,F_2)$ satisfies $\omega(F_1+F_2z)\leq 3$ for all
$z\in\mathbb T$, then there exists a $\Gamma_3$-contraction for
which $(F_1,F_2)$ is the FOP (see Theorem \ref{converse}). We
mention here that not every FOP is almost normal. We shall see
such an example in section \ref{example}.
\end{rem}

\begin{rem}
Since the FOP is defined on the space $\mathcal D_P$, it is
evident that $S_1-S_2^*P$ is equal to $0$ on the orthogonal
complement of $\mathcal D_P$ in $\mathcal H$.
\end{rem}

\begin{note}
The FOP of a $\Gamma_3$-isometry or a $\Gamma_3$-unitary
$(S_1,S_2,P)$ is defined to be $(0,0)$ because the FOP is defined
on the space $\mathcal D_P$ and in such cases $\mathcal
D_P=\{0\}$.
\end{note}

\begin{prop}\label{end-prop}
If two $\Gamma_3$-contractions are unitarily equivalent then so
are their FOPs.
\end{prop}
\begin{proof}
Let $(S_{11},S_{12},P_1)$ and $(S_{21},S_{22},P_2)$ be two
unitarily equivalent $\Gamma_3$-contractions on  Hilbert spaces
$\mathcal H_1$ and $\mathcal H_2$ respectively with FOPs
$(F_1,F_2)$ and $(G_1,G_2)$. Then there is a unitary $U$ from
$\mathcal H_1$ to $\mathcal H_2$ such that
\[
US_{11}=S_{21}U \;, \; US_{12}=S_{22}U \text{ and }\; UP_1=P_2U\,.
\]
Obviously $UP_1^*=P_2^*U$ and consequently
\[
UD_{P_1}^2=U(I-P_1^*P_1)=U-P_2^*P_2U=D_{P_2}^2U\,.
\]
Therefore, $UD_{P_1}=D_{P_2}U$. Let $V=U|_{\mathcal D_{P_1}}$.
Then $V\in\mathcal L(\mathcal D_{P_1},\mathcal D_{P_2})$ and
$VD_{P_1}=D_{P_2}V$. Thus, using the fact that
$S_{11}-S_{12}^*P_1$ and $S_{21}-S_{22}^*P_2$ on the orthogonal
complement of $\mathcal D_{P_1}$ and $\mathcal D_{P_2}$
respectively we have
\begin{align*}
D_{P_2}VF_1V^*D_{P_2} &= VD_{P_1}F_1D_{P_1}V^* \\& =
V(S_{11}-S_{12}^*P_1)V^* \\& = S_{21}-S_{22}^*P_2 \\&
=D_{P_2}G_1D_{P_2}\,.
\end{align*}
So, $F_1$ and $G_1$ are unitarily equivalent. Similarly, $F_2,
G_2$ are unitarily equivalent by the same unitary $V$. Hence the
result.
\end{proof}

\begin{rem}
The converse to the above result does not hold, i.e, two
non-unitarily equivalent $\Gamma_3$-contractions can have
unitarily equivalent FOPs. For example if we consider a
$\Gamma_3$-isometry on a Hilbert space which is not a
$\Gamma_3$-unitary, then its FOP is $(0,0)$ which is same as the
FOP of any $\Gamma_3$-unitary on the same Hilbert space.
\end{rem}

\section{The $\Gamma_3$-unitaries and
$\Gamma_3$-isometries}\label{unitaries-isometries}

\subsection{Structure theorem for $\Gamma_3$-unitaries}

We recall that a $\Gamma_3$-unitary is a commuting triple of
operators $(S_1,S_2,P)$ whose Taylor joint spectrum lies in the
distinguished boundary of $\Gamma_3$ and a $\Gamma_3$-isometry is
the restriction of a $\Gamma_3$-unitary to a joint invariant
subspace of $S_1,S_2$ and $P$. In this section we shall provide
several characterizations for the $\Gamma_3$-unitaries and
$\Gamma_3$-isometries. Also with the aid of a characterization of
$\Gamma_3$-isometries, we shall establish a partial converse to
the existence-uniqueness theorem for fundamental operator pair. We
shall state a lemma first whose proof is elementary and thus we
skip.

\begin{lem} \label{basiclemma} Let $T$ be a bounded operator on a Hilbert space
$\mathcal H$. If Re $\beta T \le 0$ for all complex numbers
$\beta$ of modulus $1$, then $T = 0$.
\end{lem}

\begin{thm} \label{G-unitary}
Let $(S_1,S_2,P)$ be a commuting operator triple defined on a
Hilbert space $\mathcal{H}.$ Then the following are equivalent:
\begin{enumerate}

\item $(S_1,S_2,P)$ is a $\Gamma_3$-unitary \, ; \item there exist
commuting unitary operators $U_{1},U_{2},U_3$ on $\mathcal{H}$
such that
\[
S_1= U_{1}+U_{2}+U_3,S_2=U_1U_2+U_2U_3+U_3U_1, \textup{ and } P=
U_{1}U_{2}U_3 \,;
\]
\item $P$ is unitary, $S_1=S_2^*P$ and
$\left(\dfrac{2}{3}S_1,\dfrac{1}{3}S_2\right)$ is a
$\Gamma_2$-contraction \,; \item $(S_1,S_2,P)$ is a
$\Gamma_3$-contraction and $P$ is a unitary \,; \item $P$ is
unitary and there exists a $\Gamma_2$-unitary $(C_1,C_2)$ on
$\mathcal H$ such that $C_1,C_2$ commute with $P$ and
\[
S_1=C_1+C_2^*P \,\,,\,\, S_2=C_2+C_1^*P\,.
\]
\end{enumerate}
\end{thm}

\begin{proof}
We shall show:\\

$(1)\Rightarrow (2) \Rightarrow (3) \Rightarrow (1)$,
$(2)\Rightarrow (4)\Rightarrow (3)$ and $(2)\Rightarrow (5)
\Rightarrow (3)$.\\

$(1)\Rightarrow (2)$. Let $(S_1,S_2,P)$ be a $\Gamma_3$-unitary.
Then by spectral theorem for commuting normal operators there
exists a spectral measure say $Q(.)$ on
$\sigma_T=\sigma_T(S_1,S_2,P)$ such that
\[
S_1= \int_{\sigma_T}\; p_{1}(z)\,Q(dz),\quad S_2=\int_{\sigma_T}\;
p_{2}(z)\,Q(dz),\text{ and }P= \int_{\sigma_T}\; p_{3}(z)\,Q(dz),
\]
where $p_{1}, p_{2}, p_3$ are the co-ordinate functions on
$\mathbb{C}^3.$ Now choose a measurable right inverse $\beta$ of
the restriction of the function $\pi_3$ to $\mathbb{T}^3$ so that
$\beta$ maps the distinguished boundary $b\Gamma_3$ of $\Gamma_3$
to $\mathbb{T}^3.$ Let $\beta=(\beta_{1},\beta_{2},\beta_3)$ and
$U_{j}=\int_{\sigma_T}\; \beta_{j}(z)Q(dz),\quad j=1,2,3.$ Then
$U_{1}, U_{2}, U_3$ are commuting unitary operators on
$\mathcal{H}$ and
\[
U_{1}+U_{2}+U_3=\int_{\sigma_T}\;(\beta_{1}+\beta_{2}+\beta_3)(z)\,Q(dz)=
\int_{\sigma_T}\; p_{1}(z)\,Q(dz)=S_1.
\]
Similarly
\[
U_{1}U_{2}+U_2U_3+U_3U_1=S_2 \quad \text{ and }\quad U_1U_2U_3=P .
\]

\noindent (2)$\Rightarrow(3)$ Since $(S_1,S_2,P)$ is a
$\Gamma_3$-contraction because it is symmetrization of three
commuting contractions. So, by Lemma \ref{lem:BSR2},
$(\dfrac{2}{3}S_1,\dfrac{1}{3}S_2)$ is a $\Gamma_2$-contraction.
The rest of the proof is straight forward.\\

$(3)\Rightarrow(1)$ Since $(3)$ holds, we have that
\begin{align*}
P^*P &=PP^*=I  & S_1^* &=P^*S_2\\
r(S_i)&\leq3,\quad i=1,2 & S_2^* &=P^*S_1.
\end{align*}
Since $P$, being a normal operator, commutes with $S_1,S_2$, by
Fudlege's Theorem on commutativity of normal operators, (see
\cite{Fuglede} for details) which states that if a normal operator
commutes with an operator $T$ then it commutes with $T^*$ too, it
commutes with $S_1^*,S_2^*$. Therefore we have
\[
S_1^*S_1=S_1^*PP^*S_1=S_2P^*S_1=S_1P^*S_2=S_1S_1^*.
\]
Here we have used the relations $S_1^*=P^*S_2, \; S_2^*=P^*S_1$
and their adjoint and the fact that $P$ commutes with $S_1,S_2$
and their adjoint. Similarly we can prove that
$S_2^*S_2=S_2S_2^*$. Therefore, $S_1,S_2,P$ are commuting normal
operators and hence $r(S_i)=\|S_i\|, \; i=1,2$.

Let $C^*(S_1,S_2,P)$ be the commutative $C^*$-algebra generated by
$S_1,S_2,P$. By general theory of joint spectrum (see p-27,
Proposition 1.2 in \cite{curto}),
\[
\sigma_T(S_1,S_2,P) = \{(\varphi(S_1),\varphi(S_2),\varphi(P)):
\varphi \in\mathcal{M}\},
\]
where $\mathcal{M}$ is the maximal ideal space of
$C^*(S_1,S_2,P).$ Suppose that
\[
(s_1,s_2,p)=(\psi(S_1),\psi(S_2),\psi(P))\in\sigma_T(S_1,S_2,P)
\quad \text{for some } \psi\in\mathcal{M}.
\]
Then
\begin{gather*}
|p|^2 = \overline{p}p = \overline{\psi(p)}\psi(p) =
\psi(P^*)\psi(P) = \psi(P^*P) = \psi(I)=1 \\
\overline{s_1}p = \overline{\psi(S_1)}\psi(P) = \psi(S_1^*P) =
\psi(S_2) = s_2 \quad \text{ and }\\ \overline{s_2}p =
\overline{\psi(S_2)}\psi(P) = \psi(S_2^*P) = \psi(S_1) = s_1.
\end{gather*}
Also since $\left(\dfrac{2}{3}S_1,\dfrac{1}{3}S_2 \right)$ is a
$\Gamma_2$-contraction,
$(\dfrac{2}{3}s_1,\dfrac{1}{3}s_2)\in\Gamma_2$. Therefore, by
Theorem \ref{thm:DB}, $(s_1,s_2,p)\in b\Gamma_3$ i.e,
$\sigma_T(S_1,S_2,P)\subseteq b\Gamma_3$. Hence $(S_1,S_2,P)$ is a
$\Gamma_3$-unitary. Hence (3) $\Rightarrow$ (1). Thus (1), (2) and (3) are equivalent.\\

\noindent (2) $\Rightarrow$ (4) is trivial.\\

\noindent (4) $\Rightarrow$ (3). Let $(S_1,S_2,P)$ be a
$\Gamma_3$-contraction and $P$ is a unitary. Since $(S_1,S_2,P)$
is a $\Gamma_3$-contraction, by Proposition \ref{prop:POS},
\[
\Phi_{i3}(\omega S_1,\omega^2 S_2,\omega^3 P)\geq 0 \text{ for all
} \omega \in \mathbb T \;,i=1,2.
\]

Now $\Phi_{13}(\omega
S_1,\omega^2 S_2,\omega^3 P)\geq 0$ gives
\[
9(I-P^*P)+(S_1^*S_1-S_2^*S_2)-3\omega(S_1-S_2^*P)-3\bar{\omega}(S_1^*-P^*S_2)\geq
0,
\]
which along with the fact that $P$ is a unitary implies that
\begin{equation}\label{13}
(S_1^*S_1-S_2^*S_2)-3\omega(S_1-S_2^*P)-3\bar{\omega}(S_1^*-P^*S_2)\geq
0 \quad \forall \omega \in \mathbb T.
\end{equation}
Putting $\omega =1$ and $-1$ respectively in (\ref{13}) and adding
them up we get
\begin{equation}\label{14}
S_1^*S_1-S_2^*S_2 \geq 0.
\end{equation}
Again from $\Phi_{23}(\omega S_1,\omega^2S_2,\omega^3P)\geq 0$ by
using the fact that $P$ is unitary, we get
\[
(S_2^*S_2-S_1^*S_1)-3\omega^2(S_2-S_1^*P)-3{\bar{\omega}^2}
(S_2^*-P^*S_1)\geq 0 \quad \text{for all } \omega \in \mathbb T.
\]
Putting $\omega=1$ and $i$ respectively we obtain that
\begin{equation}\label{15}
S_2^*S_2-S_1^*S_1\geq 0.
\end{equation}
Thus (\ref{14}) and (\ref{15}) implies that $S_1^*S_1=S_2^*S_2$.
Therefore from (\ref{13}) we have that Re $\omega (S_1-S_2^*P)\leq
0$ for all $\omega\in\mathbb T$. By Lemma \ref{basiclemma},
$S_1=S_2^*P$. Again since $(S_1,S_2,P)$ is a
$\Gamma_3$-contraction by Lemma \ref{lem:BSR2},
$\left(\dfrac{2}{3}S_1,\dfrac{1}{3}S_1 \right)$ is a $\Gamma_2$-contraction.\\

(2)$\Rightarrow$(5). Suppose (2) holds. Then
\begin{align*}
S_1 &= U_1+U_2+U_3 \\
& = (U_1+U_2)+(U_1^{-1}U_2^{-1})U_1U_2U_3 \\
& = C_1+C_2^*P\,, \quad \text{ where }
C_1=U_1+U_2\,,\,C_2=U_1U_2\,.
\end{align*}
 Also,
\[
S_2=U_1U_2+U_2U_3+U_3U_1=C_2+C_1^*P.
\]
Evidently $(C_1,C_2)$ is a $\Gamma_2$-unitary and $C_1,C_2$
commute with the unitary $P$.\\

\noindent (5)$\Rightarrow $(3) It suffices if we prove here that
$\left(\dfrac{2}{3}S_1,\dfrac{1}{3}S_2 \right)$ is a
$\Gamma_2$-contraction, because, $S_1=S_2^*P$ is obvious. Now
since $(C_1,C_2)$ is a $\Gamma_2$-unitary, there are commuting
unitaries $U_1,U_2$ such that
\[
C_1=U_1+U_2\;,\; C_2=U_1U_2.
\]
A straight-forward computation show that $(S_1,S_2,P)$ is the
symmetrization of the contractions $U_1,U_2,U_1^*U_2^*P$. Thus,
$(S_1,S_2,P)$ is a $\Gamma_3$-contraction. So, by Lemma
\ref{lem:BSR2}, $\left(\dfrac{2}{3}S_1,\dfrac{1}{3}S_2\right)$ is
a $\Gamma_2$-contraction and the proof is complete.

\end{proof}

\subsection{Model theorem for pure
$\Gamma_3$-isometries}\label{pure-isometry}

An isometry, by Wold-decomposition, can be decomposed into two
parts; a unitary and a pure isometry. See Chapter-I of classic
\cite{nagy} for a detailed description of this. A pure isometry
$V$ is unitarily equivalent to the Toeplitz operator $T_z$ on the
vectorial Hardy space $H^2(\mathcal D_{V^*})$. We shall see in
Theorem \ref{G-isometry} that an analogous Wold-type decomposition
holds for $\Gamma_3$-isometries in terms of a $\Gamma$-unitary and
a pure $\Gamma$-isometry. Theorem \ref{G-unitary} shows that every
$\Gamma$-unitary is nothing but the symmetrization of a triple of
commuting unitaries. Therefore a standard model for pure
$\Gamma_3$-isometries gives a complete vision of a
$\Gamma_3$-isometry. The following theorem describes a concrete
model for pure $\Gamma_3$-isometries.

\begin{thm}\label{model1}
Let $(\hat{S_1},\hat{S_2},\hat{P})$ be a commuting triple of
operators on a Hilbert space $\mathcal H$. If
$(\hat{S_1},\hat{S_2},\hat{P})$ is a pure $\Gamma_3$-isometry then
there is a unitary operator $U:\mathcal H \rightarrow H^2(\mathcal
D_{{\hat{P}}^*})$ such that
\[
\hat{S_1}=U^*T_{\varphi}U,\quad \hat{S_2}=U^*T_{\psi}U \text{ and
} \hat{P}=U^*T_zU,
\]
where $\varphi(z)= F_1^*+F_2z,\,\psi(z)= F_2^*+F_1z, \;
z\in\mathbb D$ and $(F_1,F_2)$ is the fundamental operator pair of
$(\hat{S_1}^*,\hat{S_2}^*,\hat{P}^*)$ such that
\begin{enumerate}
\item $(F_1,F_2)$ is almost normal, \item $\left
(\dfrac{2}{3}\varphi(z),\dfrac{1}{3}\psi(z) \right )$ is a
$\Gamma_2$-contraction for all $z\in \mathbb T$.
\end{enumerate}
Conversely, if $F_1$ and $F_2$ are two bounded operators on a
Hilbert space $E$ satisfying the above conditions, then
$(T_{F_1^*+F_2z},T_{F_2^*+F_1z},T_z)$ on $H^2(E)$ is a pure
$\Gamma_3$-isometry.
\end{thm}

\begin{proof}
Suppose that $(\hat{S_1},\hat{S_2},\hat{P})$ is a pure
$\Gamma_3$-isometry. Then $\hat{P}$ is a pure isometry and it can
be identified with the Toeplitz operator $T_z$ on $H^2(\mathcal
D_{{\hat{P}}^*})$. Therefore, there is a unitary $U$ from
$\mathcal H$ onto $H^2(\mathcal D_{{\hat{P}}^*})$ such that
$\hat{P}=U^*T_zU$. Since for $i=1,2,\, \;\hat{S_i}$ is a commutant
of $\hat{P}$, there are two multipliers $\varphi,\, \psi$ in
$H^{\infty}(\mathcal L(\mathcal D_{{\hat{P}}^*}))$ such that

\[
\hat{S_1}=U^*T_{\varphi}U \;,\; \hat{S_2}=U^*T_{\psi}U.
\]

Now $(T_{\varphi},T_{\psi},T_z)$ on $H^2(\mathcal
D_{{\hat{P}}^*})$ is a $\Gamma_3$-isometry and the triple of
commuting multiplication operators $(M_{\varphi},M_{\psi},M_z)$ on
$L^2(\mathcal D_{{\hat{P}}^*})$ is a natural $\Gamma_3$-unitary
extension of $(T_{\varphi},T_{\psi},T_z)$. The fact that
$(M_{\varphi},M_{\psi},M_z)$ on $L^2(\mathcal D_{{\hat{P}}^*})$ is
a $\Gamma_3$-unitary follows from Theorem \ref{G-unitary},
because, $(M_{\varphi},M_{\psi},M_z)$ is a $\Gamma_3$-contraction
as $(T_{\varphi},T_{\psi},T_z)$ is a $\Gamma_3$-contraction and
also $M_z$ on $L^2(\mathcal D_{{\hat{P}}^*})$ is a unitary.
Therefore, $M_{\varphi}, M_{\psi}, M_z$ commute and by Theorem
\ref{G-unitary}, we have that
\[
M_{\varphi}=M_{\psi}^*M_z \text{ and } M_{\psi}=M_{\varphi}^*M_z.
\]
So, we have
\[
\varphi(z)=G_1+G_2z \textup{ and } \psi(z)=G_2^*+G_1^*z \textup{
for some } G_1,G_2 \in\mathcal L(\mathcal D_{{\hat{T_3}}^*}).
\]
Setting ${F_1}= G_1^*$ and ${F_2}=G_2$ and by the commutativity of
$\varphi(z)$ and $\psi(z)$ we obtain
\[
[F_1,F_2]=0 \textup{ and } [F_1^*,F_1]=[F_2^*,F_2].
\]
This is same as saying that $F_1+F_2^*z$ is normal for all $z$ of
unit modulus. Therefore, $(F_1,F_2)$ is almost normal. It has been
shown in the proof of Theorem \ref{converse} that $(F_1,F_2)$ is
the FOP of the $\Gamma_3$-co-isometry
$(T_{F_1^*+F_2z}^*,T_{F_2^*+F_1z}^*,T_z^*)$. Also it follows from
Lemma \ref{lem:BSR2} that $\left
(\dfrac{2}{3}T_{\varphi(z)},\dfrac{1}{3}T_{\psi(z)} \right )$ is a
$\Gamma_2$-contraction for all $z\in \mathbb T$. Hence $\left
(\dfrac{2}{3}{\varphi(z)},\dfrac{1}{3}{\psi(z)} \right )$ is a
$\Gamma_2$-contraction for all $z\in \mathbb T$.

For the converse, we first prove that the triple of multiplication
operators $(M_{F_1^*+F_2z},M_{F_2^*+F_1z},M_z)$ on $L^2(E)$ is a
$\Gamma_3$-unitary when $F_1,F_2$ satisfy the given conditions. It
is evident that $(M_{F_1^*+F_2z},M_{F_2^*+F_1z},M_z)$ is a
commuting triple of normal operators when $(F_1,F_2)$ is almost
normal. Again since $\left
(\dfrac{2}{3}{\varphi(z)},\dfrac{1}{3}{\psi(z)} \right )$ is a
$\Gamma_2$-contraction for all $z\in \mathbb T$, it is obvious
that $\left (\dfrac{2}{3}M_{\varphi(z)},\dfrac{1}{3}M_{\psi(z)}
\right )$ is a $\Gamma_2$-contraction for all $z\in \mathbb T$.
Also $M_{F_1^*+F_2z}=M_{F_2^*+F_1z}^*M_z$ and $M_z$ on $L^2(E)$ is
unitary. Therefore, by part-(3) of Theorem \ref{G-unitary},
$(M_{F_1^*+F_2z},M_{F_2^*+F_1z},M_z)$ is a $\Gamma_3$-unitary.
Needless to say that $(T_{F_1^*+F_2z},T_{F_2^*+F_1z},T_z)$, being
the restriction of $(M_{F_1^*+F_2z},M_{F_2^*+F_1z},M_z)$ to the
common invariant subspace $H^2(E)$, is a $\Gamma_3$-isometry. Also
$T_z$ on $H^2(E)$ is a pure isometry. Thus we conclude that
$(T_{F_1^*+F_2z},T_{F_2^*+F_1z},T_z)$ is a pure
$\Gamma_3$-isometry.

\end{proof}

\subsection{A set of characterizations for the $\Gamma_3$-isometries}

\begin{lem}\label{123}
Let $U\;,\;V$ be a unitary and a pure isometry on Hilbert Spaces
$\mathcal{H}_1 \;,\; \mathcal{H}_2$ respectively, and let
$X\;:\;\mathcal{H}_1 \rightarrow \mathcal{H}_2$ be such that
$XU=VX.$ Then $X=0.$
 \end{lem}
\begin{proof}
We have , for any positive integer $n,\; XU^n\;=\;V^nX$ by
iteration. Therefore, ${U^*}^nX^*\;=\;{X^*}{V^*}^n.$ Thus $X^*$
vanishes on $Ker{V^*}^n,$ and since $\displaystyle
\bigcup_n{KerV^*}^n$ is dense in $\mathcal{H}_2$ we have $X^*=0\;
i.e,\; X=0.$
\end{proof}

\begin{thm}\label{G-isometry}
Let $S_1,S_2,P$ be commuting operators on a Hilbert space
$\mathcal H$. Then the following are equivalent:
\begin{enumerate}

\item $(S_1,S_2,P)$ is a $\Gamma_3$-isometry ; \item $P$ is
isometry, $S_1=S_2^*P$ and $(\dfrac{2}{3}S_1,\dfrac{1}{3}S_2)$ is
a $\Gamma_2$-contraction ; \item $($ Wold-Decomposition $)$: there
is an orthogonal decomposition $\mathcal H=\mathcal H_1 \oplus
\mathcal H_2$ into common invariant subspaces of $S_1,S_2$ and $P$
such that $(S_1|_{\mathcal H_1},S_2|_{\mathcal H_1},P|_{\mathcal
H_1})$ is a $\Gamma_3$-unitary and $(S_1|_{\mathcal
H_2},S_2|_{\mathcal H_2},P|_{\mathcal H_2})$ is a pure
$\Gamma_3$-isometry ; \item $(S_1,S_2,P)$ is a
$\Gamma_3$-contraction and $P$ is an isometry; \item
$(\dfrac{2}{3}S_1,\dfrac{1}{3}S_2)$ is a $\Gamma_2$-contraction
and $\rho_{ik}(\omega S_1,\omega^2S_2,\omega^3P)= 0, \; \forall
\;\omega\in \mathbb T$ and $\forall \; k \geq 3$;\\

\noindent Moreover, if the spectral radius $r(S)$ is less than $3$
then all
of the above are equivalent to :\\

\item $(\dfrac{2}{3}S_1,\dfrac{1}{3}S_2)$ is a
$\Gamma_2$-contraction and for $k\geq 3$, $(k\beta P-S_2)(k-\beta
S_1)^{-1}$ and $(k\beta^2 P-S_1)(k-\beta^2 S_2)^{-1}$ are
isometries for all $\beta\in\mathbb T$.

\end{enumerate}

\end{thm}
\begin{proof}
We prove in the following way : $(1)\Rightarrow (2)\Rightarrow
(3)\Rightarrow (1)$ , $(1)\Rightarrow
(4)\Rightarrow (5)\Rightarrow (2)$ and $(5)\Leftrightarrow (6)$.\\

\noindent $(1)\Rightarrow(2)$ Suppose that there is a
$\Gamma_3$-unitary $(\tilde S_1,\tilde S_2,\tilde P)$ on $\mathcal
K\supseteq \mathcal H$ such that $\mathcal H$ is a common
invariant subspace of $S_1,S_2$ and $P$ and $ S_1=\tilde
S_1|_{\mathcal H}, S_2=\tilde S_2|_{\mathcal H} \textup{ and }
P=\tilde P|_{\mathcal H}$. By Theorem \ref{G-unitary},
\[
\tilde S_1^*=\tilde P^*\tilde S_2 \,,\; \tilde P^*\tilde P=I \;.
\]
Taking compression to $\mathcal H$, we get
\[
S_1^*=P^*S_2\,,\; P^*P=I.
\]
Again since $(\tilde S_1,\tilde S_2,\tilde P)$ is a
$\Gamma_3$-unitary, it is a $\Gamma_3$-contraction and so by Lemma
\ref{lem:BSR2} $(\dfrac{2}{3}\tilde S_1, \dfrac{1}{3} \tilde S_2)$
is a $\Gamma_2$-contraction. Hence $(\dfrac{2}{3}S_1, \dfrac{1}{3}
S_2)$, being the restriction of $(\dfrac{2}{3}\tilde S_1,
\dfrac{1}{3} \tilde S_2)$ to the common invariant subspace
$\mathcal H$, is a $\Gamma_3$-contraction.\\

\noindent $(2)\Rightarrow (3)$ Since $P$ is an isometry, by
Wold-decomposition there is an orthogonal decomposition $\mathcal
H=\mathcal H_1\oplus \mathcal H_2$ into reducing subspaces of $P$
such that $P|_{\mathcal H_1}$ is a unitary and $P|_{\mathcal H_2}$
is a pure isometry. Therefore,
\[
P=\begin{pmatrix} U&0\\0&V \end{pmatrix} \textup{ on }\mathcal H=
\mathcal H_1 \oplus \mathcal H_2,
\]
where $U$ is a unitary and $V$ is an isometry. Let
\[
S_1=\begin{pmatrix}
S_{111} & S_{112} \\
S_{121} & S_{122}
\end{pmatrix}
\;,\; S_2=\begin{pmatrix}
S_{211}&S_{212} \\
S_{221}&S_{222}
\end{pmatrix}
\]
with respect to the decomposition $\mathcal H=\mathcal H_1\oplus
\mathcal H_2$. Now $S_1P=PS_1$ implies that
\[
\begin{pmatrix}
S_{111}&S_{112}\\S_{121}&S_{122}
\end{pmatrix} \begin{pmatrix} U&0\\0&V
\end{pmatrix}=\begin{pmatrix} U&0\\0&V
\end{pmatrix}\begin{pmatrix} S_{111}&S_{112}\\S_{121}&S_{122}
\end{pmatrix}
\]
that is
\[
\begin{pmatrix} S_{111}U&S_{112}V\\S_{121}U&S_{122}V
\end{pmatrix}=\begin{pmatrix} US_{111}&US_{112}\\VS_{121}&VS_{122}
\end{pmatrix}.
\]
So we have $S_{121}U=VS_{121}$ which makes $S_{121}=0$ by Lemma
\ref{123}. Similarly from the relation $S_2P=PS_2$ we get
$S_{221}=0$. Again the identity $S_1=S_2^*P$ provides that
\[
\begin{pmatrix}
S_{111}&S_{112}\\0&S_{122}\end{pmatrix}=\begin{pmatrix}S_{211}&S_{212}\\0&S_{222}\end{pmatrix}^*
\begin{pmatrix}U&0\\0&V\end{pmatrix} =
\begin{pmatrix}S_{211}^*U&0\\S_{212}^*U&S_{122}^*V\end{pmatrix}.
\]
This shows that $S_{212}^*U=0$ which implies that $S_{212}=0$.
Therefore, we obtain the following:
\begin{equation}\label{19}
S_{112}=S_{212}=0, \; S_{111}=S_{211}^*U \text{ and }
S_{122}=S_{222}^*V.
\end{equation}
Therefore,
\[
S_1=\begin{pmatrix}S_{111}&0\\0&S_{122}\end{pmatrix},\;
S_2=\begin{pmatrix}S_{211}&0\\0&S_{222}\end{pmatrix},
\]
with respect to the decomposition $\mathcal H=\mathcal H_1\oplus
\mathcal H_2$. Thus $\mathcal H_1, \mathcal H_2$ are common
reducing subspaces for $S_1,S_2$ and $P$ and with respect to the
decomposition $\mathcal H=\mathcal H_1 \oplus \mathcal H_2$ we
have
\begin{equation}\label{20}
(S_1,S_2,P)=(S_{111}\oplus S_{211}, S_{122} \oplus S_{222}, U
\oplus V).
\end{equation}
Again since $\left(\dfrac{2}{3}S_1,\dfrac{1}{3}S_2\right)$ is a
$\Gamma_2$-contraction, so are
$\left(\dfrac{2}{3}S_{111},\dfrac{1}{3}S_{122}\right)$ and
$\left(\dfrac{2}{3}S_{211},\dfrac{1}{3}S_{222}\right)$. Therefore,
by Theorem \ref{G-unitary}, $\left(S_{111},S_{122},U\right)$ is a
$\Gamma_3$-unitary.\\

 Again since $V$ on $\mathcal H_2$ is an isometry, it is unitarily
 equivalent to $T_z$ on $H^2(\mathcal D_{P^*})$. So, following the technique of
 the proof of Theorem \ref{model1}, we can identify $S_{211}$ and $S_{222}$ with
 $T_{\varphi}$ and $T_{\psi}$ respectively, where $\varphi(z),\psi(z)\in H^{\infty}(\mathcal D_{P^*})$.
 by the commutativity of $S_{211}$ and $S_{222}$ and by relations $S_{211}=S_{222}^*P$, $S_{222}=S_{211}^*P$, we have that
 \[
\varphi(z)=F_1^*+F_2z \text{ and  } \psi(z)=F_2^*+F_1z\,,
 \]
 where $(F_1,F_2)$ is an almost normal pair. Here we have followed the same
 technique that was applied in the proof of Theorem
 \ref{model1}. Again since $\left(\dfrac{2}{3}S_{211},\dfrac{1}{3}S_{222} \right)$ is a
 $\Gamma_2$-contraction, so is
 $\left(\dfrac{2}{3}T_{\varphi},\dfrac{1}{3}T_{\psi}\right)$ and therefore,
 $\left(\dfrac{2}{3}\varphi(z),\dfrac{1}{3}\psi(z)\right)$ is a
 $\Gamma_2$-contraction. Therefore, by Theorem \ref{model1},
 $(S_{211},S_{222},V)$ is a pure $\Gamma_3$-isometry.
 Hence $\mathcal H_1, \mathcal H_2$ are common reducing subspaces
for $S_1,S_2$ and $P$ and $(S_{111},S_{211},U)$ is a
$\Gamma_3$-unitary and $(S_{122},S_{222},V)$ is a pure
$\Gamma_3$-isometry. This is same as saying that $(S_1|_{\mathcal
H_1},S_2|_{\mathcal H_1},P|_{\mathcal H_1})$ is a
$\Gamma_3$-unitary and $(S_1|_{\mathcal H_2},S_2|_{\mathcal
H_2},P|_{\mathcal H_2})$ is a pure $\Gamma_3$-isometry.\\

\noindent $(3)\Rightarrow (1)$ Let $(S_1,S_2,P)$ have the stated
Wold-decomposition. Since
\[
(S_1|_{\mathcal H_2},S_2|_{\mathcal H_2},P|_{\mathcal H_2})
\]
is a $\Gamma_3$-isometry, by definition, there is a Hilbert space
$\mathcal K_2 \supseteq \mathcal H_2$ and a $\Gamma_3$-unitary say
$(T_1,T_2,U)$ on $\mathcal K_2$ such that $\mathcal H_2$ is a
common invariant subspace of $T_1,T_2, U$ and that
\[
T_1|_{\mathcal H_2}=S_1|_{\mathcal H_2}\,,\, T_2|_{\mathcal
H_2}=S_2|_{\mathcal H_2}\,,\, \text{ and } U|_{\mathcal
H_2}=P|_{\mathcal H_2}\,.
\]
Set $\mathcal K=\mathcal H_1\oplus \mathcal K_2$ and
\[
(\tilde{T}_1,\tilde{T}_2,\tilde{U})=(S_1|_{\mathcal H_1}\oplus
T_1, S_2|_{\mathcal H_1}\oplus T_2, P|_{\mathcal H_1}\oplus U).
\]
It follows trivially from part-(4) of Theorem \ref{G-unitary} that
the direct sum of two $\Gamma_3$-unitaries is a
$\Gamma_3$-unitary. Therefore,
$(\tilde{T}_1,\tilde{T}_2,\tilde{U})$ is a $\Gamma_3$-unitary and
it is evident that it is a $\Gamma_3$-unitary extension of
$(S_1,S_2,P)$. Therefore, $(S_1,S_2,P)$ is a
$\Gamma_3$-isometry.\\

\noindent $(1)\Rightarrow (4)$ This is obvious because
$(S_1,S_2,P)$, being a $\Gamma_3$-isometry is a
$\Gamma_3$-contraction and also is the restriction of a
$\Gamma_3$-unitary say $(\tilde S_1, \tilde S_2, \tilde P)$ where
$\tilde P$ is a unitary whose restriction to an invariant
subspace is an isometry.\\

\noindent $(4)\Rightarrow (5)$ Since $(S_1,S_2,P)$ is a
$\Gamma_3$-contraction, by Lemma \ref{lem:BSR2} $\left(
\dfrac{2}{3}S_1,\dfrac{1}{3} S_2 \right )$ is a
$\Gamma_2$-contraction. Again since $(S_1,S_2,P)$ is a
$\Gamma_3$-contraction, $\Phi_{ik}(\beta S_1, {\beta}^2S_2,
{\beta}^3P)\geq 0$ for $k\geq 3$, $\beta\in\mathbb T$ and $i=1,2$.
From the positivity of $\Phi_{1k}$ we have
\[
k^2(I-P^*P)-2k \text{ Re } \beta (S_1-S_2^*P) \geq 0.
\]
Using the fact that $P^*P=I$ we have
\[
\text{Re } \beta (S_1-S_2^*P)\leq 0 \text{ for all }
\beta\in\mathbb T.
\]
By Lemma \ref{basiclemma}, $S_1=S_2^*P$. Therefore,
\[
\Phi_{1k}(\beta S_1,{\beta}^2S_2,{\beta}^3P)=0.
\]
Similarly using the positivity of $\Phi_{2k}$ we can obtain
$\Phi_{2k}(\beta S_1,{\beta}^2S_2,{\beta}^3P)=0$.\\

\noindent $(5)\Rightarrow (2)$ We have that $\Phi_{ik}(\beta
S_1,{\beta}^2S_2,{\beta}^3P)=0$ for all $\beta\in\mathbb T$. For
$i=1$ we put $\beta=\pm 1$ and obtain $I-P^*P=0$. Hence
\[
\text{Re } \beta (S_1-S_2^*P) =0 \; \text{ for all }
\beta\in\mathbb T .
\]
Hence $S_1=S_2^*P$.\\

\noindent $(5)\Leftrightarrow (6)$ By hypothesis,
\[
\Phi_{1k}(\beta S_1,{\beta}^2S_2,{\beta}^3P) = (k-\beta
S_1)^*(k-\beta
S_1)-(k{\beta}^3P-{\beta}^2S_2)^*(k{\beta}^3P-{\beta}^2S_2)=0
\]
which implies that
\[
(k-\beta S_1)^*(k-\beta
S_1)=(k{\beta}^3P-{\beta}^2S_2)^*(k{\beta}^3P-{\beta}^2S_2).
\]
Since $r(S_1)< 3$ and $k\geq 3$, $k-\beta S_1$ is invertible and
so we have
\[
((k-\beta
S_1)^{-1})^*(k{\beta}^3P-{\beta}^2S_2)^*(k{\beta}^3P-{\beta}^2S_2)(k-\beta
S_1)^{-1}=I.
\]
Therefore, $(k{\beta}^3P-{\beta}^2S_2)(k-\beta S_1)^{-1}$ and
hence $(k{\beta}P-S_2)(k-\beta S_1)^{-1}$ is an isometry for all
$\beta\in\mathbb T$. Similarly we can show that
$(k{\beta}P-S_1)(k-\beta S_2)^{-1}$ is an isometry.\\

Conversely, let $(5)$ holds. Then $(k{\beta}P-S_2)(k-\beta
S_1)^{-1}$ is an isometry for all $\beta\in\mathbb T$. Therefore,
\[
((k-\beta
S_1)^{-1})^*(k{\beta}^3P-{\beta}^2S_2)^*(k{\beta}^3P-{\beta}^2S_2)(k-\beta
S_1)^{-1}=I
\]
or equivalently for all $\beta\in\mathbb T$,
\[
\Phi_{1k}(\beta S_1,{\beta}^2S_2,{\beta}^3P) = (k-\beta
S_1)^*(k-\beta
S_1)-(k{\beta}^3P-{\beta}^2S_2)^*(k{\beta}^3P-{\beta}^2S_2)=0
\]
Similarly we can show that $\Phi_{2k}(\beta
S_1,{\beta}^2S_2,{\beta}^3P) = 0$.

\end{proof}

\subsection{A partial converse to the Existence-Uniqueness Theorem of Fundamental operator pairs}

The existence and uniqueness of FOP is in the centre of all
results of this article (Theorem \ref{existence-uniqueness}). The
counter example we construct in section \ref{example} is a
$\Gamma_3$-contraction whose FOP is not almost normal. Here we
provide a partial converse to the existence-uniqueness theorem for
FOP.

\begin{thm}\label{converse}
Let $F_1,F_2$ be operators defined on a Hilbert space $E$ such
that $(F_1,F_2)$ is almost normal and $\left(
\dfrac{2}{3}(F_1^*+F_2z),\dfrac{1}{3}(F_2^*+F_1z) \right)$ is a
$\Gamma_2$-contraction for any $z\in\mathbb T$. Then there is a
$\Gamma_3$-contraction for which $(F_1,F_2)$ is the FOP.
\end{thm}

\begin{proof}
Let us consider the Hilbert space $H^2(E)$ and the commuting
operator triple $(T_{F_1^*+F_2z},T_{F_2^*+F_1z},T_z)$ acting on
it. We shall show that
\[
(T_{F_1^*+F_2z}^*,T_{F_2^*+F_1z}^*,T_z^*)
\]
is a $\Gamma_3$-co-isometry and $(F_1,F_2)$ is the FOP of it.
Since the pair $(F_1,F_2)$ is almost normal,
$(T_{F_1^*+F_2z},T_{F_2^*+F_1z},T_z)$ is a commuting triple and
$F_2^*+F_1z$ is normal for all $z$ of unit modulus. Clearly
$T_{\hat{F}^*+\hat{F}z}=T_{\hat{F}^*+\hat{F}z}^*T_z$ and $T_z$ is
an isometry. Again since $\left(
\dfrac{2}{3}(F_1^*+F_2z),\dfrac{1}{3}(F_2^*+F_1z) \right)$ is a
$\Gamma_2$-contraction, so is $\left(
\dfrac{2}{3}M_{F_1^*+F_2z},\dfrac{1}{3}M_{F_2^*+F_1z} \right)$,
where the multiplication operators are defined on $L^2(E)$. Also
it is obvious that the restriction of
$\left(\dfrac{2}{3}M_{F_1^*+F_2z},\dfrac{1}{3}M_{F_2^*+F_1z}
\right)$ to the common invariant subspace $H^2(E)$ is $\left(
\dfrac{2}{3}T_{F_1^*+F_2z},\dfrac{1}{3}T_{F_2^*+F_1z} \right)$.
Therefore, $\left(
\dfrac{2}{3}T_{F_1^*+F_2z},\dfrac{1}{3}T_{F_2^*+F_1z}\right)$ is a
$\Gamma_2$-contraction. Hence, by part-(2) of Theorem
\ref{G-isometry}, $(T_{F_1^*+F_2z},T_{F_2^*+F_1z},T_z)$ is a
$\Gamma_3$-isometry and consequently
$(T_{F_1^*+F_2z}^*,T_{F_2^*+F_1z}^*,T_z^*)$ is a
$\Gamma_3$-co-isometry. We now compute the FOP of
$(T_{F_1^*+F_2z}^*,T_{F_2^*+F_1z}^*,T_z^*)$. Clearly $I-T_zT_z^*$
is the projection onto the space $\mathcal D_{T_z^*}$. Now
\[
T_{F_1^*+F_2z}^*- T_{F_2^*+F_1z}T_z^*=T_{F_1+F_2^* \bar z}-
T_{F_2^*+F_1z}T_{\bar z}= T_{F_1}=(I-T_zT_z^*)F_1(I-T_zT_z^*).
\]
Similarly,
\[
T_{F_2^*+F_1z}^*-T_{F_1^*+F_2z}T_z^* =(I-T_zT_z^*)F_2(I-T_zT_z^*).
\]
Therefore, $(F_1,F_2)$ is the FOP of
$(T_{F_1^*+F_2z}^*,T_{F_2^*+F_1z}^*,T_z^*)$.

\end{proof}

\section{A necessary condition for the existence of
dilation}\label{necessary-condition}

This section is devoted to find out a set of necessary conditions
for the existence of rational dilation, that is,
$\Gamma_3$-unitary dilation to a $\Gamma_3$-contraction
(Proposition \ref{ultimate}). Since $\Gamma_3$ is polynomially
convex, rational dilation reduces to polynomial dilation on
$\Gamma_3$. So we refine the definition of $\Gamma_3$-isometric
dilation of a $\Gamma_3$-contraction.

\begin{defn}
Let $(S_1,S_2,P)$ be a $\Gamma_3$-contraction on $\mathcal H$. A
commuting triple $(T_1,T_2,V)$ defined on $\mathcal K$ is said to
be a $\Gamma_3$-isometric dilation of $(S_1,S_2,P)$ if $\mathcal H
\subseteq \mathcal K$, $(T_1,T_2,V)$ is a $\Gamma_3$-isometry and
$$ P_{\mathcal H}(T_1^{m_1}T_2^{m_2}V^n)|_{\mathcal H}=S_1^{m_1}S_2^{m_2}P^n,
\; \textup{ for all non-negative integers }m_1,m_2,n.
$$ Moreover, the dilation is called {\em minimal} if the
following holds:
$$\mathcal K=\overline{\textup{span}}\{ T_1^{m_1}T_2^{m_2}V^n h\,:\;
h\in\mathcal H \textup{ and }m_1,m_2,n\in \mathbb N \cup \{0\}
\}.$$ In a similar fashion we can define $\Gamma_3$-unitary
dilation of a $\Gamma_3$-contraction.
\end{defn}

\begin{prop}\label{exist-minimal}
If a $\Gamma_3$-contraction $(S_1,S_2,P)$ defined on $\mathcal H$
has a $\Gamma_3$-isometric dilation, then it has a minimal
$\Gamma_3$-isometric dilation.
\end{prop}
\begin{proof}
Let $(T_1,T_2,V)$ on $\mathcal K\supseteq \mathcal H$ be a
$\Gamma_3$-isometric dilation of $(S_1,S_2,P)$. Let $\mathcal K_0$
be the space defined as
$$\mathcal K_0=\overline{\textup{span}}\{ T_1^{m_1}T_2^{m_2}V^n h\,:\;
h\in\mathcal H \textup{ and }m_1,m_2,n\in \mathbb N \cup \{0\}
\}.$$ Clearly $\mathcal K_0$ is invariant under $T_1^{m_1},
T_2^{m_2}$ and $V^n$, for any non-negative integer $m_1,m_2$ and
$n$. Therefore if we denote the restrictions of $T_1,T_2$ and $P$
to the common invariant subspace $\mathcal K_0$ by $T_{11},
T_{12}$ and $V_1$ respectively, we get $T_{11}^{m_1}k=T_1^{m_1}k,
\, T_{12}^{m_2}k=T_2^{m_2}k, \textup{ and } V_1^nk=V^nk,\quad
\textup{ for any }k\in\mathcal K_0.$ Hence
$$\mathcal K_0=\overline{\textup{span}}\{ T_{11}^{m_1}T_{12}^{m_2}V_1^n h\,:\;
h\in\mathcal H \textup{ and }m_1,m_2,n\in \mathbb N \cup \{0\} \}.
$$ Therefore for any non-negative integers $m_1,m_2$ and $n$ we have
$$ P_{\mathcal H}(T_{11}^{m_1}T_{12}^{m_2}V_1^{n})h=P_{\mathcal H}(T_1^{m_1}T_2^{m_2}V^n )h,
\quad \textup{ for all }h\in\mathcal H .
$$
Now $(T_{11},T_{12},V_1)$ is a $\Gamma_3$-contraction as being the
restriction of a $\Gamma_3$-contraction $(T_1,T_2,V)$ to a common
invariant subspace $\mathcal K_0$. Again $V_1$, being the
restriction of an isometry to an invariant subspace, is also an
isometry. Therefore by Theorem \ref{G-isometry}-part (4),
$(T_{11},T_{12},V_1)$ is a $\Gamma_3$-isometry. Hence
$(T_{11},T_{12},V_1)$ is a minimal $\Gamma_3$-isometric dilation
of $(S_1,S_2,P)$.

\end{proof}

\begin{prop}\label{dilation-extension}
Let $(T_1,T_2,V)$ on $\mathcal K\supseteq \mathcal H$ be a
$\Gamma_3$-isometric dilation of a $\Gamma_3$-contraction
$(S_1,S_2,P)$. If $(T_1,T_2,V)$ is minimal, then
$(T_1^*,T_2^*,V^*)$ is a $\Gamma_3$-co-isometric extension of
$(S_1^*,S_2^*,P^*)$. Conversely, if $(T_1^*,T_2^*,V^*)$ is a
$\Gamma_3$-co-isometric extension of $(S_1^*,S_2^*,P^*)$ then
$(T_1,T_2,V)$ is a $\Gamma_3$-isometric dilation of $(S_1,S_2,P)$.
\end{prop}
\begin{proof}
We first prove that $S_1P_{\mathcal H}=P_{\mathcal H}T_1,
S_2P_{\mathcal H}=P_{\mathcal H}T_2$ and $PP_{\mathcal
H}=P_{\mathcal H}V$, where $P_{\mathcal H}:\mathcal K \rightarrow
\mathcal H$ is orthogonal projection onto $\mathcal H$. Clearly
$$\mathcal K=\overline{\textup{span}}\{ T_1^{m_1}T_2^{m_2}V^n h\,:\;
h\in\mathcal H \textup{ and }m_1,m_2,n\in \mathbb N \cup \{0\}
\}.$$ Now for $h\in\mathcal H$ we have that
\begin{align*}
S_1P_{\mathcal H}(T_1^{m_1}T_2^{m_2}V^n h)
=S_1(S_1^{m_1}S_2^{m_2}P^n h) &=S_1^{m_1+1}S_2^{m_2}P^n h\\&
=P_{\mathcal H}(T_1^{m_1+1}T_2^{m_2}V^n h)\\& =P_{\mathcal
H}T_1(T_1^{m_1}T_2^{m_2}V^n h).
\end{align*}
Thus we have $S_1P_{\mathcal H}=P_{\mathcal H}T_1$ and similarly
we can prove that $S_2P_{\mathcal H}=P_{\mathcal H}T_2$ and
$PP_{\mathcal H}=P_{\mathcal H}V$. Also for $h\in\mathcal H$ and
$k\in\mathcal K$ we have that
\[
\langle S_1^*h,k \rangle =\langle P_{\mathcal H}S_1^*h,k \rangle
=\langle S_1^*h,P_{\mathcal H}k \rangle =\langle h,S_1P_{\mathcal
H}k \rangle =\langle h,P_{\mathcal H}T_1 \rangle =\langle T_1^*h,k
\rangle .
\]
Hence $S_1^*=T_1^*|_{\mathcal H}$ and similarly
$S_2^*=T_2^*|_{\mathcal H}$ and $P^*=V^*|_{\mathcal H}$. The
converse part is obvious.

\end{proof}

\begin{prop}\label{ultimate}
Let $\mathcal H_1$ be a Hilbert space and let $(S_1,S_2,P)$ be a
$\Gamma_3$-contraction on $\mathcal H=\mathcal H_1\oplus \mathcal
H_1$ with FOP $(F_1,F_2)$ and $P$ is such that
\begin{itemize}
 \item[(i)] $ Ker (D_P)=\mathcal H_1\oplus \{0\} \textup{ and } \mathcal D_P =
 \{0\}\oplus \mathcal H_1\, ;$
\item[(ii)] $P(\mathcal D_P)=\{0\}$ and $P Ker(D_P)\subseteq
\mathcal D_P$.
\end{itemize}
If $(S_1^*,S_2^*,P^*)$ has a $\Gamma_3$-isometric dilation then
$(F_1,F_2)$ is almost normal.
\end{prop}
\begin{proof}
Let $(T_1,T_2,V)$ on a Hilbert space $\mathcal K \supseteq
\mathcal H$ be a minimal $\Gamma_3$-isometric dilation of
$(S_1^*,S_2^*,P^*)$ (such a minimal $\Gamma_3$-isometric dilation
exists by Proposition \ref{exist-minimal}) so that
$(T_1^*,T_2^*,V^*)$ is a $\Gamma_3$-co-isometric extension of
$(S_1,S_2,P)$ by Proposition \ref{dilation-extension}. Since
$(T_1,T_2,V)$ on $\mathcal K$ is a $\Gamma_3$-isometry, by Theorem
\ref{G-isometry}, it has Wold decomposition
\[
(T_1,T_2,V)=(T_{11},T_{12},U_1)\oplus (T_{21},T_{22},V_1) \text{
on } \mathcal K_1\oplus \mathcal K_2,
\]
where $(T_{11},T_{12},U_1)$ on $\mathcal K_1$ is a
$\Gamma_3$-unitary and $(T_{21},T_{22},V_1)$ on $\mathcal K_2$ is
a pure $\Gamma_3$-isometry. Since $(T_{21},T_{22},V_1)$ on
$\mathcal K_2$ is a pure $\Gamma_3$-isometry, by Theorem
\ref{model1}, $\mathcal K_2$ can be identified with $H^2(\mathcal
D_{V_1^*})$ and $T_{21},T_{22},V_1$ can be identified with
$T_{\varphi},T_{\psi},T_z$ respectively on $H^2(\mathcal
D_{V_1^*})$ for some $\varphi, \psi$ in $H^{\infty}(\mathcal
L(\mathcal D_{V_1^*}))$, where $\varphi(z)=A+Bz$ and
$\psi(z)=B^*+A^*z, \; z\in\mathbb D$, $(A^*,B)$ being the FOP of
$(T_{21}^*,T_{22}^*,V_1^*)$. Again $H^2(\mathcal D_{V_1^*})$ can
be identified with $l^2(\mathcal D_{V_1^*})$ and
$T_{\varphi},T_{\psi},T_z$ on $H^2(\mathcal D_{V_1^*})$ can be
identified with the multiplication operators
$M_{\varphi},M_{\psi},M_z$ on $l^2(\mathcal D_{V_1^*})$
respectively. So without loss of generality we can assume that
$K_2=l^2(\mathcal D_{V_1^*})$ and $T_{21}=M_{\varphi},
T_{22}=M_{\psi}$ and $V_1=M_z$ on $l^2(\mathcal D_{V_1^*})$.
On $l^2(\mathcal D_{V_1^*})$ clearly $$M_{\varphi}=\begin{bmatrix} A&0&0&\dots\\ B&A&0&\dots\\
0&B&A&\dots\\ \dots&\dots&\dots&\dots
\end{bmatrix},\; M_{\psi}=\begin{bmatrix} B^*&0&0&\dots \\ A^*&B^*&0&\dots\\
0&A^*&B^*&\dots\\\dots&\dots&\dots&\dots
 \end{bmatrix}$$ $$ \textup{and } M_z=\begin{bmatrix} 0&0&0&\dots\\I&0&0&\dots\\0&I&0&\dots\\
 \dots&\dots&\dots&\dots& \end{bmatrix}.$$
We now consider $\mathcal H$ to be a subspace of $\mathcal K$ and
$S_1,S_2,P$ defined on $\mathcal H$ to be the restrictions of
$T_1^*,T_2^*,V^*$ respectively to $\mathcal H$.\\

For the rest of the proof we denote $\mathcal D_{V_1^*}$ by $E$,
that is $\mathcal D_{V_1^*}\equiv E$.\\

\noindent \textit{Claim} $1$. $\mathcal D_P\subseteq E \oplus
\{0\}\oplus \{0\}\oplus \cdots \subseteq l^2(E)=\mathcal K_2.$

\noindent \textit{Proof of claim.} Let $h=h_1\oplus h_2 \in
\mathcal D_P\subseteq \mathcal H$, where $h_1\in \mathcal K_1$ and
$h_2=(c_0,c_1,c_2,\dots)^T \in l^2(E)=\mathcal K_2$. Since
$P(\mathcal D_P)=\{0\}$, we have
\begin{align*}
& Ph=V^*h=V^*(h_1\oplus h_2)=U_1^*h_1\oplus
M_z^*h_2=U_1^*h_1\oplus (c_1,c_2,\cdots)^T=0 \\& \Rightarrow h_1=0
\textup{ and } c_1=c_2=\dots=0.
\end{align*}
This completes the proof of \textit{Claim} $1$.\\

\noindent \textit{Claim} $2$. $Ker(D_P)\subseteq \{0\}\oplus E
\oplus \{0\}\oplus \{0\}\oplus \cdots \subseteq l^2(E)=\mathcal
K_2.$

\noindent \textit{Proof of claim.} For $k=k_1\oplus k_2 \in
Ker(D_P)\subseteq \mathcal H$, where $k_1\in\mathcal K_1$ and
$k_2=(g_0,g_1,g_2,\dots)^T \in l^2(E)=\mathcal K_2$, we have
\begin{align*}
& \qquad D_P^2k=0 \\& \Rightarrow (I-P^*P)k=P_{\mathcal
H}(I-VV^*)k=P_{\mathcal H}(k_1\oplus k_2 - k_1 \oplus
M_zM_z^*k_2)=0
\\& \Rightarrow k_1\oplus k_2 - P_{\mathcal H}(k_1\oplus
M_zM_z^*k_2)=0 \\& \Rightarrow k_1\oplus (g_0,g_1,\cdots)^T =
P_{\mathcal H}(k_1\oplus (0,g_1,g_2,\cdots)^T) \\& \Rightarrow
\|k_1\oplus (0,g_1,g_2,\cdots)^T\|\geq \|k_1 \oplus
(g_0,g_1,g_2,\cdots)^T\|
\\& \Rightarrow g_0=0.
\end{align*}
Again since $P(Ker (D_P))\subseteq \mathcal D_P$, we have for
$k=k_1\oplus (0,g_1,g_2,\dots)^T \in Ker (D_P)$,
\[
P(k_1\oplus (g_0,g_1,g_2,\dots)^T)=U_1^*k_1 \oplus
M_z^*(0,g_1,g_2,\cdots)^T =U_1^*k_1 \oplus (g_1,g_2,\cdots)\in
\mathcal D_P.
\]
Then by Claim 1, $U_1^*k_1=0$, i.e, $k_1=0$ and $g_2=k_3=\dots=0$.
Hence \textit{Claim} $2$ is established.\\

Now since $\mathcal H=\mathcal D_P \oplus Ker(D_P)$, we can
conclude that $\mathcal H \subseteq E\oplus E\oplus \{0\}\oplus
\{0\}\oplus \cdots \subseteq l^2(E)=\mathcal K_2$. Therefore
$(M_{\varphi}^*,M_{\psi}^*,M_z^*)$ on $l^2(E)$ is a
$\Gamma_3$-co-isometric extension of $(S_1,S_2,P)$.\\

We now compute the FOP of $(M_{\varphi}^*,M_{\psi}^*,M_z^*)$.

\begin{align*}
& M_{\varphi}^*-M_{\psi}M_z^*
\\& =\begin{bmatrix} A^*&B^*&0&\cdots\\ 0&A^*&B^*&\cdots\\
0&0&A^*&\cdots\\ \vdots&\vdots&\vdots&\ddots
\end{bmatrix} - \begin{bmatrix} B^*&0&0&\dots\\ A^*&B^*&0&\cdots\\ 0&A^*&B^*&\cdots\\
\vdots&\vdots&\vdots&\ddots \end{bmatrix}
\begin{bmatrix} 0&I&0&\cdots \\ 0&0&I&\cdots \\ 0&0&0&\cdots \\ \vdots&\vdots&\vdots&\ddots
\end{bmatrix} \\& = \begin{bmatrix} A^*&B^*&0&\cdots\\ 0&A^*&B^*&\cdots\\
0&0&A^*&\cdots\\ \vdots&\vdots&\vdots&\ddots
\end{bmatrix}
- \begin{bmatrix} 0&B^*&0&\cdots\\ 0&A^*&B^*&\cdots\\
0&0&A^*&\cdots\\ \vdots&\vdots&\vdots&\ddots
\end{bmatrix} \\& = \begin{bmatrix} A^*&0&0&\cdots \\ 0&0&0&\cdots \\ 0&0&0&\cdots \\
\vdots&\vdots&\vdots&\ddots \end{bmatrix}.
\end{align*}

Similarly
\[
M_{\psi}^* - M_{\varphi}M_z^* = \begin{bmatrix} B&0&0&\cdots \\ 0&0&0&\cdots \\ 0&0&0&\cdots \\
\vdots&\vdots&\vdots&\ddots \end{bmatrix}.
\]
Also
\begin{align*}
D_{M_z^*}^2 &=I-M_zM_z^* \\&
= \begin{bmatrix} I&0&0&\cdots \\ 0&0&0&\cdots\\
0&0&0&\cdots \\ \vdots&\vdots&\vdots&\ddots
\end{bmatrix} .
\end{align*}
Therefore, $\mathcal D_{M_z^*}=E\oplus \{0\}\oplus \{0\}\cdots$
and $D_{M_z^*}^2=D_{M_z^*}=I_d$ on $E\oplus \{0\}\oplus
\{0\}\cdots$. If we set
\begin{align}\label{011}
\hat{F_1}=\begin{bmatrix} A^*&0&0&\dots\\ 0&0&0&\dots\\
0&0&0&\dots\\ \dots&\dots&\dots&\dots
\end{bmatrix},\; \hat{F_2}=\begin{bmatrix} B&0&0&\dots \\ 0&0&0&\dots\\
0&0&0&\dots\\\dots&\dots&\dots&\dots
\end{bmatrix},
\end{align}
then
\[
M_{\varphi}^*-M_{\psi}M_z^*=D_{M_z^*} \hat{F_1} D_{M_z^*} \text{
and } M_{\psi}^*-M_{\varphi}M_z^*=D_{M_z^*} \hat{F_2} D_{M_z^*}.
\]
Therefore, $(\hat{F_1},\hat{F_2})$ are the FOP of
$(M_{\varphi}^*,M_{\psi}^*,M_z^*)$. We shall use a notation for
our convenience here. Let us denote
$(M_{\varphi}^*,M_{\psi}^*,M_z^*)$ by $(R_1,R_2,W)$. Therefore,
 \begin{eqnarray}
 \label{01}&R_1-R_2^*W=D_W\hat{F_1}D_W
\\ \label{02}& R_2-R_1^*W=D_W\hat{F_2}D_W.
 \end{eqnarray}\\

\noindent \textit{Claim} $3$. $\hat{F_i}D_W|_{\mathcal D_P}
\subseteq \mathcal D_P$ and $\hat{F_i}^*D_W|_{\mathcal D_P}
\subseteq \mathcal D_P$
for $i=1,2$.\\
\noindent \textit{Proof of claim.} Let
$h_0=(c_0,0,0,\cdots)^T\in\mathcal D_P$. Then
$\hat{F_1}D_Wh_0=(A^*c_0,0,0,\cdots)^T=M_{\varphi}^*h_0=R_1h_0$.
Since $R_1|_{\mathcal H}=S_1$, $R_1h_0 \in \mathcal H$. Therefore
$(A^*c_0,0,0,\cdots)^T \in \mathcal D_P$ and
$\hat{F_1}D_W|_{\mathcal D_P} \subseteq \mathcal D_P$. Similarly
we can prove that $\hat{F_2}D_W|_{\mathcal D_P} \subseteq \mathcal
D_P$.\\

We compute the adjoint of $P$. Let $(c_0,c_1,0,\cdots)^T$ and
$(d_0,d_1,0,\cdots)^T$ be two arbitrary elements in $\mathcal H$
where $(c_0,0,0,\cdots)^T, (d_0,0,0,\cdots)^T \in \mathcal D_P$
and $(0,c_1,0,\cdots)^T,(0,d_1,0,\cdots)^T\\ \in Ker(D_P)$. Now
\begin{align*}
\langle P^*(c_0,c_1,0,\cdots)^T,(d_0,d_1,0,\cdots)^T \rangle &
=\langle (c_0,c_1,0,\cdots)^T,P(d_0,d_1,0,\cdots)^T \rangle \\&
=\langle (c_0,c_1,0,\cdots)^T,W(d_0,d_1,0,\cdots)^T \rangle \\&
=\langle (c_0,c_1,0,\cdots)^T,(d_1,0,0,\cdots)^T \rangle \\&
=\langle c_0,d_1 \rangle_E \\& =\langle
(0,c_0,0,\cdots)^T,(d_0,d_1,0,\cdots)^T \rangle.
\end{align*}
Therefore
\[
P^*(c_0,c_1,0,\cdots)^T=(0,c_0,0,\cdots)^T.
\]
Now $h_0=(c_0,0,0,\cdots)^T\in\mathcal D_P$ implies that
$P^*h_0=(0,c_0,0,\cdots)^T \in \mathcal H$ and
\[
M_{\psi}^*(0,c_0,0,\cdots)^T=R_2(0,c_0,0,\cdots)^T =
(Ac_0,0,0,\cdots)^T \in \mathcal H.
\]
In particular, $(Ac_0,0,0,\cdots)^T \in \mathcal D_P$. Therefore
$\hat{F_1}^*D_Wh_0=(Ac_0,0,0,\cdots)^T \in \mathcal D_P$ and
$\hat{F_2}^*D_W|_{\mathcal D_P} \subseteq \mathcal D_P$. Similarly
we can prove that $\hat{F_2}^*D_W|_{\mathcal D_P}\subseteq
\mathcal D_P$. Hence we proved \textit{Claim} $3$.\\

\noindent \textit{Claim} $4$. $\hat{F_i}|_{\mathcal D_P}=F_i$ and $\hat{F_i}^*|_{\mathcal D_P}=F_i^*$ for $i=1,2$.\\
\noindent \textit{Proof of Claim.}
 It is obvious that $\mathcal D_P \subseteq \mathcal D_W =E\oplus\{0\}\oplus \{0\}\oplus
 \cdots$. Now since $W|_{\mathcal H}=P$ and $D_W$ is projection onto $\mathcal
 D_W$, we have
 \[
 D_W|_{\mathcal H}= D_W^2|_{\mathcal H}=D_W^2|_{\mathcal D_P} =D_P^2
 \]
 and hence $D_P^2$ is a projection onto $\mathcal D_P$. Therefore
 $D_P^2=D_P$.
 From (\ref{01}) we have
 \begin{eqnarray} \label{03} P_{\mathcal H}(R_1-R_2^*W)|_{\mathcal H}=
 P_{\mathcal H}(D_{W}\hat{F_1}D_{W})|_{\mathcal H}.
 \end{eqnarray}
 Since $(R_1,R_2,W)$ is a $\Gamma_3$-co-isometric extension of
 $(S_1,S_2,P)$, the LHS of (\ref{03}) is equal to
 $S_1-S_2^*P$.
Again since $(F_1,F_2)$ is the FOP of $(S_1,S_2,P)$, we have
 \begin{eqnarray} \label{04}
 S_1-S_2^*P=D_PF_1D_P, \quad F_1\in \mathcal L(\mathcal D_P).
 \end{eqnarray}
Since $S_1-S_2^*P$ is $0$ on the orthogonal complement of
$\mathcal D_P$, that is on $Ker(D_P)$, we have that
\begin{eqnarray} \label{05}
S_1-S_2^*P=P_{\mathcal D_P}(R_1-R_2^*W)|_{\mathcal D_P}=
 P_{\mathcal D_P}(D_{W}\hat{F_1}D_{W})|_{\mathcal D_P}.
\end{eqnarray}
Again Since $D_W|_{\mathcal D_P}=D_P$, the RHS of (\ref{05}) is
equal to $(D_{W}\hat{F_1}D_{W})|_{\mathcal D_P}$ and hence
\begin{eqnarray} \label{06} S_1-S_2^*P=(R_1-R_2^*W)|_{\mathcal D_P}=
 (D_{W}\hat{F_1}D_{W})|_{\mathcal D_P}=D_P\hat{F_1}D_P.
 \end{eqnarray}
The last identity follows from the fact (of \textit{Claim} $3$)
that $\hat{F_1}D_W|_{\mathcal D_P}\subseteq \mathcal D_P$. By the
uniqueness of $F_1$ we get that $\hat{F_1}|_{\mathcal D_P}=F_1$.
Also since $\mathcal D_P$ is invariant under $\hat{F_1}^*$ by
\textit{Claim $3$}, we have $\hat{F_1}^*|_{\mathcal D_P}=F_1^*$.
Similarly we can prove that $\hat{F_2}|_{\mathcal D_P}=F_2$ and
$\hat{F_2}^*|_{\mathcal D_P}=F_2^*$. Thus the proof of \textit{Claim} $4$ is complete.\\

Now since $(M_{\varphi},M_{\psi},M_z)$ on $l^2(E)$ is a
$\Gamma_3$-isometry,
 $M_{\varphi}$ and $M_{\psi}$ commute, that is
\[
\Small \begin{bmatrix} A&0&0&\dots\\ B&A&0&\dots\\
0&B&A&\dots\\ \dots&\dots&\dots&\dots
\end{bmatrix} \begin{bmatrix} B^*&0&0&\dots \\ A^*&B^*&0&\dots\\
0&A^*&B^*&\dots\\\dots&\dots&\dots&\dots
 \end{bmatrix} =\begin{bmatrix} B^*&0&0&\dots \\ A^*&B^*&0&\dots\\
0&A^*&B^*&\dots\\\dots&\dots&\dots&\dots\end{bmatrix}\begin{bmatrix} A&0&0&\dots\\ B&A&0&\dots\\
0&B&A&\dots\\ \dots&\dots&\dots&\dots \end{bmatrix}
\]
which implies that
\[
\Small \begin{bmatrix} AB^*&0&0&\dots\\ BB^*+AA^*&AB^*&0&\dots\\
BA^*&BB^*+AA^*&AB^*&\dots\\ \dots&\dots&\dots&\dots
\end{bmatrix} = \begin{bmatrix} B^*A&0&0&\dots \\ A^*A+B^*B&B^*A&0&\dots\\
A^*B&A^*A+B^*B&B^*A&\dots\\\dots&\dots&\dots&\dots\end{bmatrix}.
\]
Comparing both sides we obtain that
\begin{enumerate}
\item $A^*B=BA^*$ \item $A^*A-AA^*=BB^*-B^*B$. \end{enumerate}
Therefore, from (\ref{011}) we have that
\begin{enumerate}
\item $\hat{F_1}\hat{F_2}=\hat{F_2}\hat{F_1}$ \item
$\hat{F_1}^*\hat{F_1}-\hat{F_1}\hat{F_1}^*=\hat{F_2}^*\hat{F_2}-\hat{F_2}\hat{F_2}^*$.
\end{enumerate}
Taking restriction of the above two operator identities to the
subspace $\mathcal D_P$ we get
\begin{enumerate}
\item $F_1F_2=F_2F_1$ \item $F_1^*F_1-F_1F_1^*=F_2^*F_2-F_2F_2^*$.
\end{enumerate}
Therefore, by Lemma \ref{vital}, $(F_1,F_2)$ is almost normal and
the proof is complete.

\end{proof}

\section{A counter example}\label{example}

In this section we shall produce an example of a
$\Gamma_3$-contraction which satisfies the hypotheses of
Proposition \ref{ultimate} but fails to possess an almost normal
FOP.\\

Let $\mathcal H_1=l^2(E)\oplus l^2(E),\; E=\mathbb C^2$ and let
$\mathcal H=\mathcal H_1 \oplus \mathcal H_1$. Let us consider
\[
S_1=\begin{bmatrix} 0&0\\0&J \end{bmatrix},\, S_2=\begin{bmatrix}
0&0\\0&0 \end{bmatrix} \text{ and } P=\begin{bmatrix} 0&0\\Y&0
\end{bmatrix} \text{ on } \mathcal H_1\oplus \mathcal H_1,
\]
where $J=\begin{bmatrix} X&0\\0&0 \end{bmatrix}$ and
$Y=\begin{bmatrix} 0&V\\I&0 \end{bmatrix}$ on $\mathcal
H_1=l^2(E)\oplus l^2(E)$. Here $V=M_z$ and $I=I_d$ on $l^2(E)$ and
$X$ on $l^2(E)$ is defined as
\begin{align*}
X\;:&\;l^2(E)\rightarrow l^2(E) \\& (c_0,c_1,c_2,\hdots)^T \mapsto
(X_1c_0,0,0,\hdots)^T,
\end{align*}
where we choose $X_1$ on $E$ to be a non-normal contraction such
that $X_1^2=0$. For example we can choose $X_1=\begin{pmatrix}
0&\eta \\ 0&0 \end{pmatrix}$ for some $\eta >0$. Clearly $X^2=0$
and $X^*X\neq XX^*$. Since $XV=0$, $JY=0$ and thus the product of
any two of $S_1,S_2,P$ is equal to $0$. Now we unfold the
operators $S_1,S_2,P$ and write them explicitly as they are
defined on $\mathcal H = l^2(E)\oplus l^2(E)\oplus l^2(E)\oplus
l^2(E)$:
\[
S_1=\begin{bmatrix} 0&0&0&0\\0&0&0&0\\0&0&X&0\\0&0&0&0
\end{bmatrix},\; S_2= \begin{bmatrix}
0&0&0&0\\0&0&0&0\\0&0&0\emph{}&0\\0&0&0&0
\end{bmatrix} \text{ and } P=\begin{bmatrix}
 0&0&0&0\\0&0&0&0\\0&V&0&0\\I&0&0&0 \end{bmatrix}.
 \]
 We shall prove later that $(S_1,S_2,P)$ is a
 $\Gamma_3$-contraction and for time being let us assume it. Here
 \begin{align*}
 D_P^2 &=I-P^*P \\&= \begin{bmatrix} I&0&0&0\\0&I&0&0\\0&0&I&0\\0&0&0&I \end{bmatrix}-
 \begin{bmatrix} 0&0&0&I\\0&0&V^*&0\\0&0&0&0\\0&0&0&0 \end{bmatrix}
 \begin{bmatrix} 0&0&0&0\\0&0&0&0\\0&V&0&0\\I&0&0&0 \end{bmatrix} \\&=
 \begin{bmatrix} 0&0&0&0\\0&0&0&0\\0&0&I&0\\0&0&0&I \end{bmatrix} =D_P.
\end{align*}
Clearly $\mathcal D_P=\{0\}\oplus \{0\}\oplus l^2(E)\oplus
l^2(E)=\{0\}\oplus \mathcal H_1$ and $Ker(D_P)=l^2(E)\oplus
l^2(E)\oplus \{0\}\oplus \{0\}=\mathcal H_1 \oplus \{0\}$. Also
for a vector $k_0=(h_0,h_1,0,0)^T\in Ker(D_P)$ and for a vector
$k_1=(0,0,h_2,h_3)^T\in \mathcal D_P$,
\[
Pk_0=\begin{bmatrix} 0&0&0&0\\0&0&0&0\\0&V&0&0\\I&0&0&0
\end{bmatrix} (h_0,h_1,0,0)^T=(0,0,Vh_1,h_0)^T\in \mathcal D_P
\]
and
\[
Pk_1=\begin{bmatrix} 0&0&0&0\\0&0&0&0\\0&V&0&0\\I&0&0&0
\end{bmatrix} (0,0,h_2,h_3)^T = (0,0,0,0)^T.
\]

Thus $(S_1,S_2,P)$ satisfies all the conditions of Proposition
\ref{ultimate}. We now compute the FOP $(F_1,F_2)$ of
$(S_1,S_2,P)$. We have that
\begin{gather*}
S_1-S_2^*P=S_1=\begin{bmatrix} 0&0&0&0\\0&0&0&0\\0&0&X&0\\0&0&0&0
\end{bmatrix}\\ \text{ and } \quad D_PF_1D_P=\begin{bmatrix}
0&0&0&0\\0&0&0&0\\0&0&I&0\\0&0&0&I
\end{bmatrix}F_1\begin{bmatrix} 0&0&0&0\\0&0&0&0\\0&0&I&0\\0&0&0&I \end{bmatrix}.
\end{gather*}
By the uniqueness of $F_1$ we conclude that
\[
F_1=0\oplus
\begin{bmatrix} X&0\\0&0 \end{bmatrix} \text{ on }
\mathcal D_P=\{0\}\oplus \{0\}\oplus l^2(E)\oplus l^2(E).
\]
Again $S_1^*P=0$ as $X^*V=0$ and therefore $S_2-S_1^*P=0$. This
shows that the fundamental operator $F_2$, for which
$S_2-S_1^*P=D_PF_2D_P$ holds, has to be equal to $0$. Evidently
\[
F_1^*F_1 - F_1F_1^*= 0 \oplus  \begin{bmatrix} X^*X-XX^*&0\\0&0
\end{bmatrix} \neq 0 \text{ as } X^*X\neq XX^*
\]
but $F_2^*F_2 - F_2F_2^*=0$. Therefore, $[F_1^*,F_1]\neq
[F_2^*,F_2]$ and consequently $(F_1,F_2)$ is not almost normal.
This violets the conclusion of Proposition \ref{ultimate} and it
is guaranteed that the $\Gamma_3$-contraction $(S_1^*,S_2^*,P^*)$
does not have a $\Gamma_3$-isometric dilation. Since every
$\Gamma_3$-unitary dilation is necessarily a $\Gamma_3$-isometric
dilation,
$(S_1^*,S_2^*,P^*)$ does not have a $\Gamma_3$-unitary dilation.\\

Now we prove the fact that $(S_1,S_2,P)$ is a
$\Gamma_3$-contraction. Let $f(s_1,s_2,p)$ be a polynomial in the
co-ordinates of $\Gamma_3$. We show that
\[
\|f(S_1,S_2,P)\|\leq \|f\|_{\infty, \Gamma_3}.
\]
Let
\begin{equation}\label{NEQ:1}
f(s_1,s_2,p)=a_0+(a_1s_1+a_2s_2+a_3p)+ Q(s_1,s_2,p),
\end{equation}
where $Q$ is a polynomial which is either $0$ or contains only
terms of second or higher degree. We now make a change the
co-ordinates from $s_1,s_2,p$ to $z_1,z_2,z_3$ by substituting
\[
s_1=z_1+z_2+z_3\,,\, s_2=z_1z_2+z_2z_3+z_3z_1\,,\, p=z_1z_2z_3.
\]
So we have that
\begin{align}
f(s_1,s_2,p)& \nonumber =f\circ \pi_3(z_1,z_2,z_3)\\&
=a_0+a_1(z_1+z_2+z_3)+b_2(z_1z_2+z_2z_3+z_3z_1) \notag \\& \quad
+b_3(z_1z_2z_3)+Q_1(z_1,z_2,z_3),
\end{align}
where $Q_1$ is a polynomial which is either $0$ or contains terms
in $z_1,z_2,z_3$ of degree two or higher and every term in $Q_1$
contains at least one of $z_1^2,z_2^2,z_3^2$ as one of the
factors. The co-efficients $b_2,b_3$ may not be same as $a_2,a_3$
because $Q(s_1,s_2,p)$ may contain a term with $s_1^2$ and a term
with $s_1s_2$ which contribute some terms with
$z_1z_2+z_2z_3+z_3z_1$ and $z_1z_2z_3$. We rewrite $f$ in the
following way:
\[
f(s_1,s_2,p)=f\circ
\pi_3(z_1,z_2,z_3)=a_0+a_1(z_1+z_2+z_3)+R(z_1,z_2,z_3),
\]
where $R$ contains terms in $z_1,z_2,z_3$ of degree two or higher.
Now $S_1,S_2$ and $P$ are chosen in such a way that the degree two
or higher terms in $S_1,S_2,P$ vanish and so from (\ref{NEQ:1}) we
have
\[
f(S_1,S_2,P)=a_0I+a_1S_1+a_3P=\begin{bmatrix} a_0I&0\\a_3Y &
a_0I+a_1J \end{bmatrix}
\]
Since $Y$ is a contraction and $\|J\|=\dfrac{1}{4}$, it is obvious
that
\[
\left\| \begin{bmatrix} a_0I&0\\a_3Y & a_0I+a_1J
\end{bmatrix} \right\| \leq \left\| \begin{pmatrix} |a_0|&0\\
|a_3|&|a_0|+\dfrac{|a_1|}{4}
\end{pmatrix} \right\|.
\]
We divide the rest of the proof into two cases.\\

\textbf{Case 1.} When $|a_0|\leq |a_1|$.\\

We show that
\[
\left\| \begin{pmatrix} |a_0|&0\\
|a_3|&|a_0|+\dfrac{|a_1|}{4}
\end{pmatrix} \right\| \leq
\left\| \begin{pmatrix} |a_0|&0\\
|a_1|+|a_3|&|a_0|
\end{pmatrix} \right\|\,.
\]
Let $\begin{pmatrix} \epsilon \\ \delta
\end{pmatrix}$ be a unit vector in $\mathbb C^2$ such that
\[
\left\| \begin{pmatrix} |a_0|&0\\
|a_3|&|a_0|+\dfrac{|a_1|}{4} \end{pmatrix} \right\|= \left\| \begin{pmatrix} |a_0|&0\\
|a_3|&|a_0|+\dfrac{|a_1|}{4} \end{pmatrix}\begin{pmatrix} \epsilon
\\ \delta
\end{pmatrix} \right\|.
\]
Without loss of generality we can choose $\epsilon, \delta \geq 0$
because
\[
\left\| \begin{pmatrix} |a_0|&0\\
|a_3|&|a_0|+\dfrac{|a_1|}{4} \end{pmatrix}\begin{pmatrix} \epsilon
\\ \delta
\end{pmatrix} \right\|^2
 =|a_0\epsilon|^2+\left| |a_3\epsilon|+\left( |a_0|+\dfrac{|a_0|}{4} \right)\delta\right|^2
\]
and if we replace $\begin{pmatrix} \epsilon \\ \delta
\end{pmatrix}$ by $\begin{pmatrix} |\epsilon | \\ |\delta|
\end{pmatrix}$ we see that
\[
\left\| \begin{pmatrix} |a_0|&0\\
|a_3|&|a_0|+\dfrac{|a_1|}{4} \end{pmatrix}\begin{pmatrix}
|\epsilon|
\\ |\delta|
\end{pmatrix} \right\|^2 \geq
\left\| \begin{pmatrix} |a_0|&0\\
|a_3|&|a_0|+\dfrac{|a_1|}{4} \end{pmatrix}\begin{pmatrix} \epsilon
\\ \delta
\end{pmatrix} \right\|^2 \,.
\]
So, assuming $\epsilon, \delta \geq 0$ we get
\begin{align}
&\left\| \begin{pmatrix} |a_0|&0\\
|a_3|&|a_0|+\dfrac{|a_1|}{4} \end{pmatrix}\begin{pmatrix} \epsilon
\\ \delta
\end{pmatrix} \right\|^2 \notag \\
&=|a_0\epsilon|^2+\left\{ |a_3\epsilon|+\left(
|a_0|+\dfrac{|a_1|}{4} \right)\delta \right\} ^2 \notag \\
&=|a_0\epsilon|^2+|a_3\epsilon|^2+ \left\{
|a_0|^2+\dfrac{|a_0a_1|}{2}+\dfrac{|a_1|^2}{16} \right\}{\delta}^2
+
2|a_3|\left( |a_0|+\dfrac{|a_1|}{4} \right)\epsilon\delta \notag \\
&= \left\{
(|a_0|^2+|a_3|^2)\epsilon^2+|a_0|^2\delta^2+2|a_0a_3|\epsilon\delta
\right\} + \left\{ \dfrac{|a_1|^2}{16}+\dfrac{|a_0a_1|}{2}
\right\}\delta^2 + \dfrac{|a_1a_3|}{2}\epsilon\delta \,.
\label{eqn:ex1}
\end{align}
Again
\begin{align}
& \left\| \begin{pmatrix} |a_0|&0\\
|a_1|+|a_3|&|a_0| \end{pmatrix}\begin{pmatrix} \epsilon \\ \delta
\end{pmatrix} \right\|^2 \notag \\
& = |a_0\epsilon|^2+\{ (|a_1|+|a_3|)\epsilon+|a_0|\delta \}^2
\notag \\
& = |a_0|^2\epsilon^2+\{ |a_1|^2+|a_3|^2+2|a_1a_3|
\}\epsilon^2+2|a_0|(|a_1|+|a_3|)\epsilon\delta+|a_0|^2\delta^2
\notag \\
& =\left\{
(|a_0|^2+|a_3|^2)\epsilon^2+|a_0|^2\delta^2+2|a_0a_3|\epsilon\delta
\right\} +
(|a_1|^2\epsilon^2+2|a_0a_1|\epsilon\delta)+2|a_1a_3|\epsilon^2
\,. \label{eqn:ex2}
\end{align}
We now compare (\ref{eqn:ex1}) and (\ref{eqn:ex2}). If $\epsilon
\geq \delta$ then
\[
(|a_1|^2\epsilon^2+2|a_0a_1|\epsilon\delta)+2|a_1a_3|\epsilon^2
\geq \left( \dfrac{|a_1|^2}{16}+\dfrac{|a_0a_1|}{2}
\right)\delta^2+\frac{|a_1a_3|}{2}\epsilon\delta
\]
Therefore, it is evident from (\ref{eqn:ex1}) and (\ref{eqn:ex2})
that
\[
\left\| \begin{pmatrix} |a_0|&0\\
|a_3|&|a_0|+\dfrac{|a_1|}{4}
\end{pmatrix}\begin{pmatrix} \epsilon \\ \delta
\end{pmatrix} \right\|^2 \leq
\left\| \begin{pmatrix} |a_0|&0\\
|a_1|+|a_3|&|a_0|
\end{pmatrix}\begin{pmatrix} \epsilon \\ \delta
\end{pmatrix} \right\|^2\,.
\]
If $\epsilon < \delta$ we consider the unit vector $\begin{pmatrix} \delta \\
\epsilon \end{pmatrix}$ and it suffices if we show that
\[
\left\| \begin{pmatrix} |a_0|&0\\
|a_3|&|a_0|+\dfrac{|a_1|}{4}
\end{pmatrix}\begin{pmatrix} \epsilon \\ \delta
\end{pmatrix} \right\|^2 \leq
\left\| \begin{pmatrix} |a_0|&0\\
|a_1|+|a_3|&|a_0|
\end{pmatrix}\begin{pmatrix} \delta \\ \epsilon
\end{pmatrix} \right\|^2\,.
\]
A computation similar to (\ref{eqn:ex2}) gives
\begin{align}
& \left\| \begin{pmatrix} |a_0|&0\\
|a_1|+|a_3|&|a_0|
\end{pmatrix}\begin{pmatrix} \delta \\ \epsilon
\end{pmatrix} \right\|^2 \notag \\
& = |a_0|^2\delta^2+\{ |a_1|^2+|a_3|^2+2|a_1a_3|
\}\delta^2+2|a_0|(|a_1|+|a_3|)\epsilon\delta +|a_0|^2\epsilon^2
\notag \\
&= \{ |a_0|^2(\epsilon^2+\delta^2)+2|a_0a_3|\epsilon\delta \}+\{
|a_1|^2+|a_3|^2+2|a_1a_3| \}\delta^2+2|a_0a_1|\epsilon\delta  \notag \\
& = \{ |a_0|^2+2|a_0a_3|\epsilon\delta \}+\{
|a_1|^2+|a_3|^2+2|a_1a_3| \}\delta^2+2|a_0a_1|\epsilon\delta \,.
 \label{eqn:ex3}
\end{align}
In the last equality we used the fact that
$|\epsilon|^2+|\delta|^2=1$. Again from (\ref{eqn:ex1}) we have
\begin{align}
& \left\| \begin{pmatrix} |a_0|&0\\
|a_3|&|a_0|+\dfrac{|a_1|}{4} \end{pmatrix}\begin{pmatrix} \epsilon
\\ \delta
\end{pmatrix} \right\|^2 \notag \\
& = \{ |a_0|^2(\epsilon^2+\delta^2)+2|a_0a_3|\epsilon\delta \} +
\left\{ |a_3|^2\epsilon^2+\dfrac{|a_1a_3|}{2}\epsilon\delta
\right\}
+\left\{ \dfrac{|a_1|^2}{16}+\dfrac{|a_0a_1|}{2} \right\}\delta^2 \notag \\
& \leq \{ |a_0|^2(\epsilon^2+\delta^2)+2|a_0a_3|\epsilon\delta \}
+ \left\{ |a_3|^2\epsilon^2+\dfrac{|a_1a_3|}{2}\epsilon\delta
\right\} + \left\{ \dfrac{|a_1|^2}{16}+\dfrac{|a_1|^2}{2}
\right\}\delta^2 \notag \\
&=\{ |a_0|^2+2|a_0a_3|\epsilon\delta \} + \left\{
\dfrac{9|a_1|^2}{16}\delta^2+|a_3|^2\epsilon^2+\dfrac{|a_1a_3|}{2}\epsilon\delta
\right\} \label{eqn:ex4}
\end{align}
The last inequality follows from the fact that $|a_0|\leq |a_1|$.
Since $\epsilon<\delta$ we can conclude from (\ref{eqn:ex3}) and
(\ref{eqn:ex4}) that
\[
\left\| \begin{pmatrix} |a_0|&0\\
|a_3|&|a_0|+\dfrac{|a_1|}{4}
\end{pmatrix}\begin{pmatrix} \epsilon \\ \delta
\end{pmatrix} \right\|^2 \leq
\left\| \begin{pmatrix} |a_0|&0\\
|a_1|+|a_3|&|a_0|
\end{pmatrix}\begin{pmatrix} \delta \\ \epsilon
\end{pmatrix} \right\|^2\,.
\]
Therefore,
\[
\left\| f(S_1,S_2,P) \right\|\leq \left\| \begin{pmatrix} |a_0|&0\\
|a_3|&|a_0|+\dfrac{|a_1|}{4} \end{pmatrix} \right\| \leq \left\| \begin{pmatrix} |a_0|&0\\
|a_1|+|a_3|&|a_0| \end{pmatrix} \right\|.
\]

A classical result of Caratheodory and Fejer states that
\[
\inf \, \|b_0+b_1z+r(z)\|_{\infty, \overline{\mathbb D}} = \left\| \begin{pmatrix} b_0&0\\
b_1&b_0 \end{pmatrix} \right\|,
\]
where the infemum is taken over all polynomials $r(z)$ in one
variable which contain only terms of degree two or higher. For an
elegant proof to this result, see Sarason's seminal paper
\cite{sarason}, where the result is derived as a consequence of
the classical commutant lifting theorem of Sz.-Nagy and Foias (see
\cite{nagy}). Using this fact, we have
\begin{align}
\|f(S_1,S_2,P)\| & \nonumber \leq \left\|\begin{pmatrix} |a_0|&0\\
|a_1|+|a_3|&|a_0| \end{pmatrix} \right\| \\& \nonumber = \inf \,
\||a_0|+(|a_1|+|a_3|)z+r(z)\|_{\infty, \overline{\mathbb D}} \\&
\label{014} \leq
 \inf \, \| |a_0|+|a_1|(z_1+z_2+z_3)+|a_3|(z_1z_2z_3)+R(z_1,z_2,z_3) \|_{\infty,
 \Delta}\\&\nonumber
 \leq \inf \, \| |a_0|+|a_1|(z_1+z_2+z_3)+|a_2|(z_1z_2+z_2z_3+z_3z_1) \\& \label{015} \quad \quad +
 |a_3|(z_1z_2z_3)+R(z_1,z_2,z_3) \|_{\infty, \Delta} \\& \nonumber
 \leq \inf \, \| |a_0|+|a_1|(z_1+z_2+z_3)+|a_2|(z_1z_2+z_2z_3+z_3z_1) \\& \nonumber \quad \quad +
 |a_3|(z_1z_2z_3)+R(z_1,z_2,z_3) \|_{\infty, \overline{\mathbb
 D^3}}\\& \nonumber
 = \inf \, \| a_0+a_1(z_1+z_2+z_3)+c_2(z_1z_2+z_2z_3+z_3z_1) \\& \label{016} \quad \quad +
 c_3(z_1z_2z_3)+R(z_1,z_2,z_3) \|_{\infty, \overline{\mathbb
 D^3}} \\& \nonumber
\leq \| a_0+a_1(z_1+z_2+z_3)+b_2(z_1z_2+z_2z_3+z_3z_1) \\&
\label{017} \quad \quad +
 b_3(z_1z_2z_3)+Q_1(z_1,z_2,z_3) \|_{\infty, \overline{\mathbb
 D^3}} \\& \nonumber
= \|f\circ \pi_3(z_1,z_2,z_3) \|_{\infty, \overline{\mathbb
 D^3}} \\& \nonumber =\|f(s_1,s_2,p) \|_{\infty, \Gamma_3}.
\end{align}
Here $\Delta=\overline{\mathbb D}\times \{i\} \times \{-i \}
\subseteq \overline{\mathbb D^3}$. The polynomials $r(z)$ and
$R(z_1,z_2,z_3)$ range over polynomials of degree two or higher.
The inequality (\ref{014}) was obtained by putting $z_1=z, z_2=i$
and $z_3=-i$ which makes the set of polynomials
$|a_0|+|a_1|(z_1+z_2+z_3)+|a_3|(z_1z_2z_3)+R(z_1,z_2,z_3)$, a
subset of the set of polynomials $|a_0|+(|a_1|+|a_3|)z+r(z)$. The
infimum taken over a subset is always bigger than or equal to the
infimum taken over the set itself. We obtained the inequality
(\ref{015}) by applying this argument. The equality (\ref{016})
was obtained by multiplying by $\dfrac{a_0}{|a_0|}$ and replacing
$z_i$ by $\dfrac{\bar{a_0}a_1}{|a_0a_1|}z_i,\, i=1,2,3$. Clearly
$c_2=|a_2|.(\dfrac{\bar{a_0}a_1}{|a_0a_1|})^2$ and
$c_3=|a_3|.(\dfrac{\bar{a_0}a_1}{|a_0a_1|})^3$. The last
inequality (\ref{017}) was reached by choosing $R(z_1,z_2,z_3)$
suitably to be the polynomial
$(b_2-c_2)(z_1z_2+z_2z_3+z_3z_1)+(b_3-c_3)(z_1z_2z_3)+Q_1(z_1,z_2,z_3)$.\\

\textbf{Case 2.} When $|a_0|> |a_1|$.\\

It is obvious from Case 1 that
\[
\left\| \begin{pmatrix} |a_0|&0\\
|a_3|&|a_0|+\dfrac{|a_1|}{4} \end{pmatrix} \right\|  \leq \left\| \begin{pmatrix} |a_0|&0 \\
|a_3|&|a_0|+\dfrac{|a_0|}{4} \end{pmatrix} \right\| \leq \left\| \begin{pmatrix} |a_0|&0\\
|a_0|+|a_3|&|a_0| \end{pmatrix} \right\| \,.
\]
Therefore,
\begin{align}
\|f(S_1,S_2,P)\| & \nonumber \leq \left\|\begin{pmatrix} |a_0|&0\\
|a_0|+|a_3|&|a_0| \end{pmatrix} \right\| \\& \nonumber = \inf \,
\||a_0|+(|a_0|+|a_3|)z+r(z)\|_{\infty, \overline{\mathbb D}} \\&
\nonumber \leq
 \inf \, \| |a_0|+(|a_0|+|a_3|)(z_1z_2z_3)+R(z_1,z_2,z_3) \|_{\infty,
 \Delta}\\&\nonumber
 \leq \inf \, \| |a_0|+|a_1|(z_1+z_2+z_3)z_2z_3+(|a_0|+|a_3|)(z_1z_2z_3)\\
 &\label{018} \quad \quad +R(z_1,z_2,z_3) \|_{\infty,
 \Delta}\\
 &\nonumber
 =\inf \, \| |a_0|+|a_1|(z_1+z_2+z_3)+(|a_0|+|a_3|)(z_1z_2z_3)\\
 & \nonumber \quad \quad +R(z_1,z_2,z_3) \|_{\infty,
 \Delta}\\
 &\nonumber
 \leq \inf \, \| |a_0|+|a_1|(z_1+z_2+z_3)+|a_2|(z_1z_2+z_2z_3+z_3z_1) \\& \nonumber \quad \quad +
 (|a_0|+|a_3|)(z_1z_2z_3)+R(z_1,z_2,z_3) \|_{\infty, \Delta} \\& \nonumber
 \leq \inf \, \| |a_0|+|a_1|(z_1+z_2+z_3)+|a_2|(z_1z_2+z_2z_3+z_3z_1) \\& \nonumber \quad \quad +
 (|a_0|+|a_3|)(z_1z_2z_3)+R(z_1,z_2,z_3) \|_{\infty, \overline{\mathbb
 D^3}}\\& \nonumber
 = \inf \, \| a_0+a_1(z_1+z_2+z_3)+d_2(z_1z_2+z_2z_3+z_3z_1) \\& \label{019} \quad \quad +
 d_3(z_1z_2z_3)+R(z_1,z_2,z_3) \|_{\infty, \overline{\mathbb
 D^3}} \\& \nonumber
\leq \| a_0+a_1(z_1+z_2+z_3)+b_2(z_1z_2+z_2z_3+z_3z_1) \\&
\label{020} \quad \quad +
 b_3(z_1z_2z_3)+Q_1(z_1,z_2,z_3) \|_{\infty, \overline{\mathbb
 D^3}} \\& \nonumber
= \|f\circ \pi_3(z_1,z_2,z_3) \|_{\infty, \overline{\mathbb
 D^3}} \\& \nonumber =\|f(s_1,s_2,p) \|_{\infty, \Gamma_3}.
\end{align}
Here the notations used are as same as they were in case 1. The
inequality (\ref{018}) holds because $|a_1|(z_1+z_2+z_3)z_2z_3$ is
a polynomial that contains terms of degree two or higher which
makes the set of polynomials
$|a_0|+|a_1|(z_1+z_2+z_3)z_2z_3+(|a_0|+|a_3|)(z_1z_2z_3)+R(z_1,z_2,z_3)$
, a subset of the set of polynomials
$|a_0|+(|a_0|+|a_3|)(z_1z_2z_3)+R(z_1,z_2,z_3)$. The equality
(\ref{019}) was obtained by multiplying by $\dfrac{a_0}{|a_0|}$
and replacing $z_i$ by $\dfrac{\bar{a_0}a_1}{|a_0a_1|}z_i,\,
i=1,2,3$. Clearly $d_2=|a_2|.(\dfrac{\bar{a_0}a_1}{|a_0a_1|})^2$
and $d_3=(|a_0|+|a_3|).(\dfrac{\bar{a_0}a_1}{|a_0a_1|})^3$. The
last inequality (\ref{020}) was reached by choosing
$R(z_1,z_2,z_3)$ suitably to be the polynomial
$(b_2-d_2)(z_1z_2+z_2z_3+z_3z_1)+(b_3-d_3)(z_1z_2z_3)+Q_1(z_1,z_2,z_3)$.\\

\section{Conditional dilation}\label{conditional-dilation}

In the previous section, we have seen that there are
$\Gamma_3$-contractions whose FOPs are not almost normal. A class
of such $\Gamma_3$-contractions do not dilate. In this section, we
shall see that if the FOPs of a $\Gamma_3$-contraction
$(S_1,S_2,P)$ and its adjoint $(S_1^*,S_2^*,P^*)$ satisfy the
almost normality condition, then $(S_1,S_2,P)$ possesses a
$\Gamma_3$-unitary dilation. In fact, almost normality of the FOP
of $(S_1,S_2,P)$ is sufficient to have such a $\Gamma_3$-unitary
dilation, because, we shall see that if the FOP satisfies the
almost normality condition then such a $\Gamma_3$-contraction can
be dilated to a $\Gamma_3$-isometry and every $\Gamma_3$-isometry
can be extended to a $\Gamma_3$-unitary. Here we shall provide an
explicit construction of $\Gamma_3$-unitary dilation to such
$\Gamma_3$-contractions. Before going to the construction of
dilation, we list out a few important properties of the FOPs which
we shall use in the proof of the dilation theorem. Throughout this
section, we shall use a result (whose proof could be found in
chapter-I in \cite{nagy}) from one variable operator theory.
\begin{equation}\label{nagy-foias}
PD_P=D_{P^*}P, \textup{ for any contraction } P \textup{ on a
Hilbert space}.
\end{equation}
We shall also use the definitions of the FOPs, that is,
\begin{align}\label{funda-repeat}
& S_1-S_2^*P= D_{P}F_1D_{P},\;\; S_2-S_1^*P= D_{P}F_2D_{P}
\\& \label{funda-repeat1} S_1^*-S_2P^*=
D_{P^*}{F_1}_*D_{P^*} \;\; S_2^*-S_1P^*= D_{P^*}{F_2}_*D_{P^*}.
\end{align}

\begin{lem}\label{funda-properties}
Let $(S_1,S_2,P)$ be a $\Gamma_3$-contraction on a Hilbert space
$\mathcal H$. Let $(F_1,F_2)$ and $({F_1}_*, {F_2}_*)$ be
respectively the FOPs of $(S_1,S_2,P)$ and $(S_1^*,S_2^*,P^*)$.
Then

\begin{enumerate}
\item $PF_i={F_i}_*^*P|_{\mathcal D_{P}}$ and
$P^*{F_i}_*={F_i}^*P^*|_{\mathcal D_{P^*}}$ for $i=1,2$ \item
$D_{P}S_1=F_1D_{P}+F_2^*D_{P}P$ and
$D_{P}S_2=F_2D_{P}+F_1^*D_{P}P$ \item
$S_1D_{P^*}=D_{P^*}{F_1}_*^*+PD_{P^*}{F_2}_*$ and
$S_2D_{P^*}=D_{P^*}{F_2}_*^*+PD_{P^*}{F_1}_*$ \item
$S_1^*S_1-S_2^*S_2=D_{P}(F_1^*F_1-F_2^*F_2)D_{P}$, when
$[F_1,F_2]=0$ \item
${S_1}_*^*{S_1}_*-{S_2}_*^*{S_2}_*=D_{P^*}({F_1}_*^*{F_1}_*-{F_2}_*^*{F_2}_*)D_{P^*}$,
when $[{F_1}_*,{F_2}_*]=0$ \item $\omega(F_2+F_1^*z)\leq 3$ and
$\omega({F_2}_*^*+{F_1}_*z)\leq 3$ for all $z\in\mathbb T$.

\end{enumerate}
\end{lem}

\begin{proof}
\textbf{(1).} It suffices if we show $PF_1={F_1}_*^*P|_{\mathcal
D_{P}}$ because the proof to the other identities are same. For
$D_{P}h \in \mathcal D_{P}$ and $D_{P^*}h^{\prime}\in \mathcal
D_{P^*}$, we have by virtue of (\ref{funda-repeat1}),
\begin{align*}
\langle PF_1D_{P}h,D_{P^*}h^{\prime} \rangle =\langle
D_{P^*}PF_1D_{P}h,h^{\prime} \rangle &=\langle
PD_{P}F_1D_{P}h,h^{\prime}\rangle \\&=\langle
P(S_1-S_2^*P)h,h^{\prime} \rangle \\& =\langle
(S_1-PS_2^*)Ph,h^{\prime} \rangle \\&=\langle
D_{P^*}{F_1}_*^*D_{P^*}Ph,h^{\prime} \rangle \\&=\langle
{F_1}_*^*PD_{P}h,D_{P^*}h^{\prime} \rangle.
\end{align*}
\textbf{(2).}
$D_{P}(F_1D_{P}+F_2^*D_{P}P)=(S_1-S_2^*P)+(S_2^*-P^*S_1)P=D_{P}^2S_1$,
by (\ref{funda-repeat}). Therefore,
$D_{P}S_1=F_1D_{P}+F_2^*D_{P}P$ because both LHS and RHS are
defined from $\mathcal H$ to $\mathcal D_{P}$. The proof to
the other identity is similar.\\

\noindent \textbf{(3).}
$(D_{P^*}{F_1}_*^*+PD_{P^*}{F_2}_*)D_{P^*}=(S_1-PS_2^*)+P(S_2^*-S_1P^*)=S_1D_{P^*}^2$,
by (\ref{funda-repeat1}). Therefore,
$S_1D_{P^*}=D_{P^*}{F_1}_*^*+PD_{P^*}{F_2}_*$
and the proof to the other identity is similar.\\

\noindent \textbf{(4).} We have $S_1^*S_2^*P=S_2^*S_1^*P$ which by
(\ref{funda-repeat}) and (\ref{funda-repeat1}) implies that
\[
S_1^*(S_1-D_{P}F_1D_{P})=S_2^*(S_2-D_{P}F_2D_{P}).
\]
Therefore, we have that
\begin{align*}
S_1^*S_1-S_2^*S_2 &=S_1^*D_{P}F_1D_{P}-S_2^*D_{P}F_2D_{P}
\\&
=(D_{P}F_1^*+P^*D_{P}F_2)F_1D_{P}
\\&-(D_{P}F_2^*+P^*D_{P}F_1)F_2D_{P}, \textup{ [by part-(1)
of this lemma]} \\& = D_{P}(F_1^*F_1-F_2^*F_2)D_{P}, \text{ when }
[F_1,F_2]=0.
\end{align*}
\textbf{(5).} Same as (4).\\

\noindent \textbf{(6).} By Theorem 3.1, $\omega(F_1+F_2z)\leq 3$
for every $z\in\mathbb T$. Therefore, the inequalities follow from
Lemma \ref{basicnrlemma1}.
\end{proof}

\begin{thm}\label{main-dilation-theorem}
 Let $(S_1,S_2,P)$ be a $\Gamma_3$-contraction defined on a Hilbert space $\mathcal
 H$ such that the FOPs $(F_1,F_2)$ and $(F_{1*},F_{2*})$ of $(S_1,S_2,P)$
 and $(S_1^*,S_2^*,P^*)$ respectively are almost normal. Let
 $\mathcal K=\cdots\oplus\mathcal D_{P}\oplus\mathcal D_{P}\oplus\mathcal D_{P}\oplus\mathcal H\oplus
 \mathcal D_{P^*}\oplus\mathcal D_{P^*}\oplus\mathcal D_{P^*}\oplus\cdots$ and let $(R_1,R_2,U)$
 be a triple of operators defined on $\mathcal K$ by
\begin{eqnarray}\label{2.3}
&R_1 =\footnotesize \left[
\begin{array}{ c c c c|c|c c c c}
\bm{\ddots}&\vdots &\vdots&\vdots   &\vdots  &\vdots& \vdots&\vdots&\vdots\\
\cdots&0&F_1&F_2^*  &0&  0&0&0&\cdots\\
%\hline
\cdots&0&0&F_1  &F_2^*D_{P}&  -F_2^*P^*&0&0&\cdots\\ \hline

\cdots&0&0&0   &S_1&   D_{P^*}{F_2}_*&0&0&\cdots\\ \hline

\cdots&0&0&0   &0&  {F_1}_*^*& {F_2}_*&0&\cdots\\
\cdots&0&0&0   &0&  0&{F_1}_*^*&{F_2}_*&\cdots\\
\vdots&\vdots&\vdots&\vdots&\vdots&\vdots&\vdots&\vdots&\bm{\ddots}\\
\end{array} \right]\,,\\&
\label{2.4} R_2 = \footnotesize \left[
\begin{array}{ c c c c|c|c c c c}
\bm{\ddots}&\vdots &\vdots&\vdots   &\vdots  &\vdots& \vdots&\vdots&\vdots\\
\cdots&0&F_2&F_1^*  &0&  0&0&0&\cdots\\
%\hline
\cdots&0&0&F_2  &F_1^*D_{P}&  -F_1^*P^*&0&0&\cdots\\ \hline

\cdots&0&0&0   &S_2&   D_{P^*}{F_1}_*&0&0&\cdots\\ \hline

\cdots&0&0&0   &0&  {F_2}_*^*& {F_1}_*&0&\cdots\\
\cdots&0&0&0   &0&  0&{F_2}_*^*&{F_1}_*&\cdots\\
\vdots&\vdots&\vdots&\vdots&\vdots&\vdots&\vdots&\vdots&\bm{\ddots}\\
\end{array} \right] \\&
\text{ \large and }\quad U = \left[
\begin{array}{ c c c c|c|c c c c}
\bm{\ddots}&\vdots &\vdots&\vdots   &\vdots  &\vdots& \vdots&\vdots&\vdots\\
\cdots&0&0&I  &0&  0&0&0&\cdots\\
%\hline
\cdots&0&0&0  &D_{P}&  -{P}^*&0&0&\cdots\\ \hline

\cdots&0&0&0   &P&   D_{P^*}&0&0&\cdots\\ \hline

\cdots&0&0&0   &0&  0& I&0&\cdots\\
\cdots&0&0&0   &0&  0&0&I&\cdots\\
\vdots&\vdots&\vdots&\vdots&\vdots&\vdots&\vdots&\vdots&\bm{\ddots}\\
\end{array} \right].
\end{eqnarray}
Then $(R_1,R_2,U)$ is a minimal $\mathbb E$-unitary dilation of
$(S_1,S_2,P)$.
\end{thm}

 \begin{proof}
It is evident from Sz.-Nagy-Foias model theory for contraction
(see Chapter-I and Chapter-II of \cite{nagy}) that $U$ is the
minimal unitary dilation of $P$. Also it is obvious from the block
matrices of $R_1,R_2$ and $U$ that
$$ P_{\mathcal H}(R_1^{m_1}R_2^{m_2}U^n)|_{\mathcal H}=S_1^{m_1}S_2^{m_2}P^n
\textup{ for all integers } m_1,m_2,n,$$ which proves that
$(R_1,R_2,U)$ dilates $(S_1,S_2,P)$. The minimality of the
$\Gamma_3$-unitary dilation follows from the fact that $\mathcal
K$ and $U$ are respectively the minimal unitary dilation space and
minimal unitary dilation of $P$. Therefore, in order to prove that
$(R_1,R_2,U)$ is a minimal $\Gamma_3$-unitary dilation of
$(S_1,S_2,P)$, we need to show that $(R_1,R_2,U)$ is a
$\Gamma_3$-unitary. By virtue of Theorem \ref{G-unitary}, it
suffices to show the following steps:
\begin{enumerate}
\item $R_1R_2=R_2R_1$ \item $R_iU=UR_i \quad i=1,2 $ \item
$R_1=R_2^*U$ \item $\left( \dfrac{2}{3}R_1,\dfrac{1}{3}R_2
\right)$ is a $\Gamma_2$-contraction.
\end{enumerate}
\textbf{Step 1.} $R_1R_2 =$
\begin{equation*}
\tiny \left[
\begin{array}{ c c c|c|c c c }
\bm{\ddots} &\vdots&\vdots   &\vdots  &\vdots& \vdots&\vdots\\
\cdots & F_1F_2& F_1F_1^*+F_2^*F_2 &F_2^*F_1^*D_{P}&  -F_2^*F_1^*P&0&\cdots\\
%\hline
\cdots&0&F_1F_2  & {\normalsize \substack{F_1F_1^*D_{P}\\
+F_2^*D_{P}S_2}}&
{\normalsize \substack{-F_1F_1^*P^*+F_2^*D_{P}D_{P^*}{F_1}_*\\ -F_2^*P^*{F_2}_*^*}} &-F_2^*P^*{F_1}_*&\cdots\\
\hline

\cdots&0&0 &S_1S_2&
S_1D_{P^*}{F_1}_*+D_{P^*}{F_2}_*{F_2}_*^*&D_{P^*}{F_2}_*{F_1}_*&\cdots\\
\hline

\cdots&0&0 &0&  {F_1}_*^*{F_2}_*^*& {F_1}_*^*{F_1}_*+{F_2}_*{F_2}_*^*&\cdots\\
\cdots &0&0 &0&  0&{F_1}_*^*{F_2}_*^*& \cdots\\
\vdots &\vdots&\vdots&\vdots&\vdots&\vdots &\bm{\ddots}\\
\end{array} \right]
\end{equation*}
and $R_2R_1 =$
\begin{equation*}
\tiny \left[
\begin{array}{ c c c|c|c c c }
\bm{\ddots} &\vdots&\vdots   &\vdots  &\vdots& \vdots&\vdots\\
\cdots &F_2F_1&F_2F_2^*+F_1^*F_1  &F_1^*F_2^*D_{P}&  -F_1^*F_2^*P&0&\cdots\\
%\hline
\cdots&0&F_2F_1  & {\normalsize \substack{F_2F_2^*D_{P}\\
+F_1^*D_{P}S_1}}&
{\normalsize \substack{-F_2F_2^*P^*+F_1^*D_{P}D_{P^*}{F_2}_*\\ -F_1^*P^*{F_1}_*^*}}&-F_1^*P^*{F_2}_*&\cdots\\
\hline

\cdots&0&0 &S_2S_1&
S_2D_{P^*}{F_2}_*+D_{P^*}{F_1}_*{F_1}_*^*&D_{P^*}{F_1}_*{F_2}_*&\cdots\\
\hline

\cdots&0&0 &0&  {F_2}_*^*{F_1}_*^*& {F_2}_*^*{F_2}_*+{F_1}_*{F_1}_*^*&\cdots\\
\cdots &0&0 &0&  0&{F_2}_*^*{F_1}_*^*& \cdots\\
\vdots &\vdots&\vdots&\vdots&\vdots&\vdots &\bm{\ddots}\\
\end{array} \right].
\end{equation*}
For proving $R_1R_2$ and $R_2R_1$ to be equal, it suffices to
verify the equality of the entities $(-1,0), (0,1), (-1,1)$ in the
matrices of $R_1R_2$ and $R_2R_1$ because the other entities are
equal by the given conditions $(1)$ and $(2)$ of the theorem.
Therefore we have to show the following operator identities.
\begin{itemize}
\item[($a_1$)]$S_1D_{P^*}{F_1}_*+D_{P^*}{F_2}_*{F_2}_*^*=S_2D_{P^*}{F_2}_*+D_{P^*}{F_1}_*{F_1}_*^*$,
\item[($a_2$)]$F_1F_1^*D_{P}+F_2^*D_{P}S_2=F_2F_2^*D_{P}+F_1^*D_{P}S_1$,
\item[($a_3$)]$-F_1F_1^*P^*+F_2^*D_{P}D_{P^*}{F_1}_*-F_2^*P^*{F_2}_*^*\\
=-F_2F_2^*P^*+F_1^*D_{P}D_{P^*}{F_2}_*-F_1^*P^*{F_1}_*^*$.

\end{itemize}

$(a_1).$ We apply part-(3) of Lemma \ref{funda-properties} and get
\begin{align*}
S_1D_{P^*}{F_1}_*+D_{P^*}{F_2}_*{F_2}_*^* &
=(D_{P^*}{F_1}_*^*+PD_{P^*}{F_2}_*){F_1}_*
+D_{P^*}{F_2}_*{F_2}_*^* \\&
=D_{P^*}({F_1}_*^*{F_1}_*+{F_2}_*{F_2}_*^*)+
PD_{P^*}{F_2}_*{F_1}_*.
\end{align*}
Similarly
$F_2D_{P^*}{F_2}_*+D_{P^*}{F_1}_*{F_1}_*^*=D_{P^*}({F_2}_*^*{F_2}_*+{F_1}_*{F_1}_*^*)
+ PD_{P^*}{F_1}_*{F_2}_*$
and now we apply the hypotheses of the theorem.\\

$(a_2).$ We have, by part-(2) of Lemma \ref{funda-properties} that
\begin{align*}
F_1F_1^*D_{P}+F_2^*D_{P}S_2
&=F_1F_1^*D_{P}+F_2^*(F_2D_{P}+F_1^*D_{P}P)\\&
=(F_1F_1^*+F_2^*F_2)D_{P}+F_2^*F_1^*D_{P}P.
\end{align*}
Similarly
$F_2F_2^*D_{P}+F_1^*D_{P}S_1=(F_2F_2^*+F_1^*F_1)D_{P}+F_1^*F_2^*D_{P}P$
and the equality follows from the hypotheses of the theorem.\\
$(a_3).$ By virtue of Lemma \ref{funda-properties}- part-(1), both
of LHS and RHS are defined from $\mathcal D_{P^*}$ to $\mathcal
D_{P}$. Let $T_1=$LHS and $T_2=$RHS. Therefore, by
(\ref{funda-repeat}) and (\ref{funda-repeat1}) we have that
\begin{align*}
D_{P}(T_2-T_1)D_{P^*} & =
D_{P}(F_1F_1^*-F_2F_2^*)D_{P}P^*-P^*D_{P^*}({F_1}_*{F_1}_*^*-{F_2}_*{F_2}_*^*)D_{P^*}
\\& \quad +
D_{P}F_1^*D_{P}D_{P^*}{F_2}_*D_{P^*}-D_{P}F_2^*D_{P}D_{P^*}{F_1}_*D_{P^*}
\\& =
(S_1^*S_1-S_2^*S_2)P^*-P^*(S_1S_1^*-S_1S_2^*)\\& \quad
+(S_1^*-P^*S_2)(S_2^*-S_1P^*) -(S_2^*-P^*S_1)(S_1^*-S_2P^*)\\&=0.
\end{align*}
The first equality follows from the hypotheses of the theorem and
also by using part-(4) and part-(5) of Lemma
\ref{funda-properties}.\\

\noindent \textbf{Step 2.} We now show that $R_1U=UR_1$.
\begin{equation*}\label{2.5}
R_1U= \footnotesize \left[
\begin{array}{ c c c c|c|c c c c }
\bm{\ddots} &\vdots &\vdots&\vdots   &\vdots  &\vdots& \vdots& \vdots& \vdots\\
\cdots &0&F_1&F_2^* &0&0&0&0&\cdots\\
\cdots &0&0&F_1&F_2^*D_{P}&-F_2^*P^*&0&0&\cdots\\
%\hline
\cdots&0&0&0&F_1D_{P}+F_2^*D_{P}P &F_2^*D_{P}D_{P^*}-F_1P^*
& -F_2^*P^*&0&\cdots\\
\hline

\cdots&0&0&0 &S_1P& S_1D_{P^*}&D_{P^*}{F_2}_*&0&\cdots\\
\hline

\cdots&0&0 &0&0&0&  {F_1}_*^*& {F_2}_*&\cdots\\
\cdots&0&0 &0&0&0&0& {F_1}_*^*&\cdots\\

\cdots &0&0   &0&  0&0&0&0&\cdots\\
\vdots &\vdots&\vdots&\vdots&\vdots&\vdots&\bm{\ddots}\\
\end{array} \right]
\end{equation*}
and
\begin{equation*}\label{2.5}
UR_1= \footnotesize \left[
\begin{array}{ c c c c|c|c c c c }
\bm{\ddots} &\vdots &\vdots&\vdots   &\vdots  &\vdots& \vdots& \vdots& \vdots\\
\cdots &0&F_1&F_2^*&0 &0&0&0&\cdots\\
\cdots &0&0&F_1&F_2^*D_{P}&-F_2^*P^*&0&0&\cdots\\
%\hline
\cdots &0&0&0&D_{P}S_1&D_{P}D_{P^*}{F_2}_*-P^*{F_1}_*^*
& -P^*{F_2}_*&0&\cdots\\
\hline

\cdots &0&0&0 &PS_1& PD_{P^*}{F_2}_*+D_{P^*}{F_1}_*^*&D_{P^*}{F_2}_*&0&\cdots\\
\hline

\cdots &0&0 &0&0&0&  {F_1}_*^*& {F_2}_*&\cdots\\
\cdots &0&0 &0&0&0&0& {F_1}_*^*&\cdots\\

\cdots &0&0   &0&  0&0&0&0&\cdots\\
\vdots &\vdots&\vdots&\vdots&\vdots&\vdots&\bm{\ddots}\\
\end{array} \right].
\end{equation*}
The equality of the entities in the positions $(-1,2),\,(-1,0)$
and $(0,1)$ of $R_1U$ and $UR_1$ follows from part-(1), part-(2)
and part-(3) of Lemma \ref{funda-properties}. Therefore, for
showing the equality of $R_1U$ and $UR_1$ we have to verify that
$F_2^*D_{P}D_{P^*}-F_1P^*=D_{P}D_{P^*}{F_2}_*-P^*{F_1}_*^*$. Let
$T=(F_2^*D_{P}D_{P^*}-F_1P^*)-(D_{P}D_{P^*}{F_2}_*-P^*{F_1}_*^*)$.
Then $T$ maps $\mathcal D_{P^*}$ into $\mathcal D_{P}$. Now
\begin{align*}
D_{P}TD_{P^*}
&=D_{P}F_2^*D_{P}D_{P^*}^2-D_{P}F_1P^*D_{P^*}+D_{P}P^*{F_1}_*^*D_{P^*}
-D_{P}^2D_{P^*}{F_2}_*D_{P^*}
\\&
=(S_2^*-P^*S_1)(I-PP^*)-(S_1-S_2^*P)P^*\\& \quad +P^*(S_1-PS_2^*)
-(I-P^*P)(S_2^*-S_1P^*) \\& =0.
\end{align*}
We used (\ref{nagy-foias}), (\ref{funda-repeat}) and
(\ref{funda-repeat}). Hence $R_1U=UR_1$.\\

\noindent \textbf{Step 3.} We now show that $R_1=R_2^*U$.
\begin{eqnarray}
&R_2^*U= \notag \\&
 \SMALL \left[
\begin{array}{ c c c c|c|c c c c}
\bm{\ddots}&\vdots &\vdots&\vdots   &\vdots  &\vdots& \vdots&\vdots&\vdots\\
%\hline
\cdots&F_1&F_2^*&0&0&0&0&0&\cdots\\
\cdots&0&F_1&F_2^*&0&0&0&0&\cdots\\
\hline

\cdots&0&0&D_{P}F_1&S_2^*&0&0&0&\cdots\\
\hline

\cdots&0&0&-PF_1&  {F_1}_*^*D_{P^*}& {F_2}_*&0&0&\cdots\\
\cdots&0&0&0&0&{F_1}_*^*&{F_2}_*&0&\cdots\\
\vdots&\vdots&\vdots&\vdots&\vdots&\vdots&\vdots&\vdots&\bm{\ddots}\\
\end{array} \right]
\left[
\begin{array}{ c c c|c|c c c }
\bm{\ddots} &\vdots&\vdots   &\vdots  &\vdots& \vdots&\vdots\\
\cdots&0&I  &0&  0&0&\cdots\\
%\hline
\cdots&0&0  &D_{P}&  -{P}^*&0&\cdots\\ \hline

\cdots&0&0   &{P}&   D_{{P}^*}&0&\cdots\\ \hline

\cdots&0&0   &0&  0& I&\cdots\\
\cdots&0&0   &0&  0&0&\cdots\\
\vdots&\vdots&\vdots&\vdots&\vdots&\vdots&\bm{\ddots}\\
\end{array} \right] \notag \\&
= \SMALL \left[
\begin{array}{ c c c c|c|c c c c }
\bm{\ddots} &\vdots &\vdots&\vdots   &\vdots  &\vdots& \vdots& \vdots& \vdots\\
\cdots &0&F_1&F_2^* &0&0&0&0&\cdots\\
\cdots &0&0&F_1&F_2^*D_{P}&-F_2^*P^*&0&0&\cdots\\

\hline

\cdots&0&0&0 &S_2^*P+D_{P}F_1D_{P}& S_2^*D_{P^*}-D_{P}F_1P^*&0&0&\cdots\\
\hline

\cdots&0&0 &0&{F_1}_*^*D_{P^*}P-PF_1D_{P}& {F_1}_*^*D_{P^*}^2+T_3F_1P^*& {F_2}_*&0&\cdots\\
\cdots&0&0 &0&0&0&{F_1}_*^*& {F_2}_*&\cdots\\

\vdots &\vdots&\vdots&\vdots&\vdots&\vdots&\bm{\ddots}\\
\end{array} \right]
\end{eqnarray}

In order to prove $R_1=R_2^*U$, we need to show the following
steps because the other equalities follow from
(\ref{funda-repeat}) and (\ref{funda-repeat1}).
\begin{itemize}
\item[($c_1$)]${F_1}_*^*D_{P^*}P=PF_1D_{P^*}$, \item[($c_2$)]
$D_{P^*}{F_2}_*=S_2^*D_{P^*}-D_{P}F_1P^*$, \item[($c_3$)]
${F_1}_*^*D_{P^*}^2+PF_1P^*={F_1}_*^*$.
\end{itemize}
The identity $(c_1)$ follows from part-(1) of Lemma
\ref{funda-properties} together with (\ref{nagy-foias}).\\
$(c_2).$  Let $J_1=D_{P^*}{F_2}_*+D_{P}F_1P^*$. Now
\begin{align*}
J_1D_{P^*} =D_{P^*}{F_2}_*D_{P^*}+D_{P}F_1P^*D_{P^*}
&=(S_2^*-S_1P^*)+D_{P}F_1D_{P}P^*\\&
=(S_2^*-S_1P^*)+(S_1-S_2^*P)P^*\\&=S_2^*D_{P^*}^2.
\end{align*}
We used (\ref{nagy-foias}), (\ref{funda-repeat}) and
(\ref{funda-repeat1}) here. Since $J$ is defined from $D_{P^*}$ to
$\mathcal H$, $(b_2)$ is
established.\\
$(c_3). \quad
{F_1}_*^*D_{P^*}^2+PF_1P^*={F_1}_*^*(I-PP^*)+{F_1}_*^*PP^*={F_1}_*^*$.\\

\noindent \textbf{Step 4.} We first show that $R_2$ is a normal
operator. We have $R_1=R_2^*U$ from step 3 and so $R_2=R_1^*U$ by
Fuglede's theorem, \cite{Fuglede}. Thus,
\begin{align*}
R_2R_2^*=R_1^*UR_2^*=R_1^*R_2^*U=R_2^*R_1^*U=R_2^*R_2
\end{align*}
and $R_2$ is normal. Therefore, $r(R_2)=\omega(R_2)=\|R_2\|$.
Suppose that the matrix of $R_2$ with respect to the decomposition
$l^2(\mathcal D_{P})\oplus \mathcal H \oplus l^2(\mathcal
D_{P^*})$ of $\mathcal K$ is $ \left[
\begin{array}{ccc}
B_1 & B_2 & B_3\\
0 & S_2 & B_4\\
0& 0& B_5  \end{array} \right]. $ Since $(F_1,F_2)$ and $({F_1}_*,
{F_2}_*)$ are FOPs, by part-(6) of Lemma \ref{funda-properties},
$\omega(F_2+F_1^*z)$ and $\omega({F_2}_*^*+{F_1}_*z)$ are not
greater than $3$. Therefore,
\[
r(B_1)\leq 1 \textup{ and } r(B_5)\leq 1.
\]
Also $\|S_2\|\leq 3$. So by Lemma 1 of \cite{hong}, which states
that the spectrum of an operator of the form $\begin{bmatrix}
X&Y\\0&Z
\end{bmatrix}$ is a subset of $\sigma(X)\cup \sigma (Z)$, we have
\[
\sigma(R_2)\subseteq \sigma(B_1)\cup\sigma(S_2)\cup\sigma({B_5}).
\]
Therefore, $r(R_2)\leq 3$. Hence, $r(R_2)=\|R_2\|\leq 3$ and
$(R_1,R_2,U)$ is a $\Gamma_3$-unitary.

\end{proof}

\begin{thm}\label{isometric-dilation}
 Let $\mathcal N \subseteq \mathcal K$ be defined as $\mathcal N=\mathcal H\oplus l^2({\mathcal
 D_P})$. Then $\mathcal N$ is a common invariant subspace of $R_1,R_2,U$ and
 $(T_1,T_2,V)=(R_1|_{\mathcal N},R_2|_{\mathcal N},U|_{\mathcal N})$
 is a minimal $\Gamma_3$-isometric dilation of $(S_1,S_2,P)$.
 \end{thm}
 \begin{proof}
This theorem could be treated as a corollary of the previous
theorem. It is evident from the matrices of $R_1,R_2$ and $U$ that
$\mathcal N=\mathcal H\oplus l^2({\mathcal
 D_P})= H\oplus \mathcal
D_P\oplus\mathcal D_P\oplus\cdots$ is a common invariant subspace
of $R_1,R_2$ and $U$. Therefore by the definition of
$\Gamma_3$-isometry, the restriction of $(R_1,R_2,U)$ to the
common invariant subspace $\mathcal N$, i.e. $(T_1,T_2,V)$ is a
$\Gamma_3$-isometry. The matrices of $T_1,T_2$ and $V$ with
respect to the decomposition $\mathcal H\oplus \mathcal
D_P\oplus\mathcal D_P\oplus\cdots$ of $\mathcal N$ are the
following:
\begin{gather*}
T_1=\begin{bmatrix} S_1&0&0&0&\cdots\\F_2^*D_P&F_1&0&0&\cdots\\0&F_2^*&F_1&0&\cdots\\0&0&F_2^*&F_1&\cdots\\
\vdots&\vdots&\vdots&\vdots&\ddots \end{bmatrix}\;,\quad
T_2=\begin{bmatrix} S_2&0&0&0&\cdots\\F_1^*D_P&F_2&0&0&\cdots\\0&F_1^*&F_2&0&\cdots\\0&0&F_1^*&F_2&\cdots\\
\vdots&\vdots&\vdots&\vdots&\ddots \end{bmatrix}\;,\\
V=\begin{bmatrix}
P&0&0&0&\cdots\\D_P&0&0&0&\cdots\\0&I&0&0&\cdots\\0&0&I&0&\cdots\\
\vdots&\vdots&\vdots&\vdots&\ddots \end{bmatrix}.
\end{gather*}
It is obvious from the matrices of $T_1,T_2$ and $V$ that the
adjoint of $(T_1,T_2,V)$ is a $\Gamma_3$-co-isometric extension of
$(S_1^*,S_2^*,P^*)$. Therefore by Proposition
\ref{dilation-extension}, $(T_1,T_2,V)$ is a $\Gamma_3$-isometric
dilation of $(S_1,S_2,P)$. The minimality of this
$\Gamma_3$-isometric dilation follows from the fact that $\mathcal
N$ and $V$ are respectively the minimal isometric dilation space
and minimal isometric dilation of $P$. Hence the proof is
complete.
\end{proof}
\begin{rem}
The minimal $\Gamma_3$-unitary $(R_1,R_2,U)$ described in Theorem
\ref{main-dilation-theorem} is a minimal $\Gamma_3$-unitary
extension of $(T_1,T_2,V)$ given in Corollary
\ref{isometric-dilation}. The reason is that for any
$\Gamma_3$-unitary extension $(\hat R_1,\hat R_2,\hat U)$ of
$(T_1,T_2,V)$, $\hat U$ is the minimal unitary extension of $V$.
\end{rem}

As a consequence of the dilation theorems, we arrived at a
sufficient condition for a triple $(S_1,S_2,P)$ to become a
$\Gamma_3$-contraction.

\begin{thm}\label{sufficient1} Let $S_1,S_2,P$ be commuting
operators on a Hilbert space $\mathcal H$ with $\|S_i\|\leq 3$ and
$\|P\|\leq 1$. Let $F_1,F_2$ be two commuting bounded operators on
$\mathcal D_{P}$ with $\omega(F_1+ F_2z)\leq 3$ for all
$z\in\mathbb T$ such that
\[
S_1-S_2^*P=D_{P}F_1D_{P} \text{ and } S_2-S_1^*P=D_{P}F_2D_{P}.
\]
If $(F_1,F_2)$ is almost normal then $\Gamma_3$ is a complete
spectral set for $(S_1,S_2,P)$ and hence $(S_1,S_2,P)$ is a
$\Gamma_3$-contraction.
\end{thm}

\begin{proof}
By Lemma \ref{basicnrlemma1}, we have that
\[
\omega(F_1^*+F_2z)\leq 3 \text{ and } \omega(F_2^*+F_1z)\leq 3 \,.
\]
So, we can construct $T_1,T_2,V$ as in Theorem
\ref{isometric-dilation} so that $(T_1,T_2,V)$ is a
$\Gamma_3$-isometric dilation of $(S_1,S_2,P)$. Since every
$\Gamma_3$-isometry is nothing but the restriction of a
$\Gamma_3$-unitary to a joint invariant subspace, $(T_1,T_2,V)$
can be extended to a $\Gamma_3$-unitary which will become a
$\Gamma_3$-unitary dilation of $(S_1,S_2,P)$. Obviously the
restriction of $(T_1^*,T_2^*,V^*)$ to $\mathcal H$ gives back
$(S_1^*,S_2^*,P^*)$. Since the restriction of a
$\Gamma_3$-contraction to a joint invariant subspace is also a
$\Gamma_3$-contraction, $(S_1^*,S_2^*,P^*)$ is a
$\Gamma_3$-contraction. Therefore, $(S_1,S_2,P)$ is also a
$\Gamma_3$-contraction. Also since $(S_1,S_2,P)$ has normal
$b\Gamma_3$-dilation, $\Gamma_3$ is a complete spectral set for
$(S_1,S_2,P)$.

\end{proof}

\section{A functional model for a class of
$\Gamma_3$-contractions}\label{functional-model}

In this section, with the help of the dilation theorems proved in
the previous section, we construct a concrete and explicit
functional model for the class of $\Gamma_3$-contractions
$(S_1,S_2,P)$ for which the adjoint $(S_1^*,S_2^*,P^*)$ has almost
normal FOP. The following result is necessary for the proof of the
model theorem.

\begin{prop}\label{easyprop1}
If $T$ is a contraction and $V$ is its minimal isometric dilation
then $T^*$ and $V^*$ have defect spaces of same dimension.
\end{prop}
\begin{proof}
Let $T$ and $V$ be defined on $\mathcal H$ and $\mathcal K$. Since
$V$ is the minimal isometric dilation of $T$ we have
\[
\mathcal K=\overline{\text{span}} \{p(V)h:\; h\in\mathcal H \text{
and }p \text{ is any polynomial in one variable }\}.
\]
The defect spaces of $T^*$ and $V^*$ are respectively $\mathcal
D_{T^*}=\overline{\textup{Ran}}\;(I-TT^*)^{\frac{1}{2}}$ and
$D_{V^*}=\overline{\textup{Ran}}\; (I-VV^*)^{\frac{1}{2}}$. Let
$\mathcal N= \overline{\textup{Ran}}\;
(I-VV^*)^{\frac{1}{2}}|_{\mathcal H}$. For $h\in\mathcal H$ and
$n\geq 1$, we have
\[
(I-VV^*)V^nh=V^nh-VV^*V^nh=0, \text{ as } V \text{ is an
isometry}.
\]
Therefore, $(I-VV^*)p(V)h=p(0)(I-VV^*)h$ for any polynomial $p$ in
one variable. So $(I-VV^*)k\in\mathcal N$ for any $k\in\mathcal
K$. This shows that $\overline{\textup{Ran}}(I-VV^*)\subseteq
\mathcal N$ and
hence $\overline{\textup{Ran}}(I-VV^*)=\mathcal D_{V^*}=\mathcal N$.\\

We now define for $h\in\mathcal H$,
\begin{align*}
L:\; & \text{Ran}(I-TT^*)^{\frac{1}{2}}\,\rightarrow\,
\text{Ran}(I-VV^*)^{\frac{1}{2}}
\\& (I-TT^*)^{\frac{1}{2}}h \mapsto
(I-VV^*)^{\frac{1}{2}}h.
\end{align*}
We prove that $L$ is an isometry. Since $V^*$ is co-isometric
extension of $T^*$, $TT^*=P_{\mathcal H}VV^*|_{\mathcal H}$ and
thus we have $ (I_{\mathcal H}-TT^*)=P_{\mathcal H}(I_{\mathcal
K}-VV^*)|_{\mathcal H}$, that is, $D_{P^*}^2=P_{\mathcal
H}D_{V^*}^2|_{\mathcal H}.$ Therefore, for $h\in\mathcal H$,
\begin{align*}
\|D_{T^*}h\|^2=\langle D_{P^*}^2h,h \rangle=\langle P_{\mathcal
H}D_{V^*}^2h,h \rangle =\langle D_{V^*}^2h,h
\rangle=\|D_{V^*}h\|^2,
\end{align*}
and $L$ is an isometry and this can clearly be extended to a
unitary from $\mathcal D_{T^*}$ to $\mathcal D_{V^*}$. Hence
proved.
\end{proof}

\begin{thm}\label{model2}
 Let $(S_1,S_2,P)$ be a $\Gamma_3$-contraction on a Hilbert space $\mathcal
 H$ such that $(S_1^*,S_2^*,P^*)$ has almost normal FOP $({F_1}_*,{F_2}_*)$.
 Let $(\hat T_1,\hat T_2,\hat V)$ on $\mathcal N_*=\mathcal H\oplus \mathcal D_{P^*}\oplus\mathcal D_{P^*}\oplus
 \hdots$ be defined as
 \begin{gather*}
 \hat T_1=\begin{bmatrix}
 S_1 & D_{P^*}{F_2}_* & 0 & 0 & \cdots \\
 0 & {F_1}_*^* & {F_2}_* & 0 & \cdots \\
 0 & 0 & {F_1}_*^* & {F_2}_* & \cdots \\
 0 & 0 & 0 & {F_1}_*^* & \cdots \\
 \vdots & \vdots & \vdots & \vdots & \ddots
 \end{bmatrix}
 \;,\;
 \hat T_2=\begin{bmatrix}
 S_2 & D_{P^*}{F_1}_* & 0 & 0 & \cdots \\
 0 & {F_2}_*^* & {F_1}_* & 0 & \cdots \\
 0 & 0 & {F_2}_*^* & {F_1}_* & \cdots \\
 0 & 0 & 0 & {F_2}_*^* & \cdots \\
 \vdots & \vdots & \vdots & \vdots & \ddots
 \end{bmatrix}
 \;,\\
\text{and}\quad \hat V=\begin{bmatrix}
P & D_{P^*} & 0 & 0 & \cdots\\
0 & 0 & I & 0 & \cdots\\
0 & 0 & 0 & I & \cdots \\
0 & 0 & 0 & 0 & \cdots\\
\vdots & \vdots & \vdots & \vdots & \ddots
\end{bmatrix}\,.
\end{gather*}
 Then
 \begin{enumerate}
 \item $(\hat T_1,\hat T_2,\hat V)$ is a $\Gamma_3$-co-isometry, $\mathcal H$ is a
 common invariant subspace of $\hat T_1,\hat T_2,\hat V$ and $\hat T_1|_{\mathcal H}=S_1$,
 $\hat T_2|_{\mathcal H}=S_2$ and $\hat V|_{\mathcal H}=P$\,;
 \item there is an orthogonal decomposition $\mathcal N_*=\mathcal N_1\oplus \mathcal N_2$
 into reducing subspaces of $\hat T_1$, $\hat T_2$ and $\hat V$ such that
 $(\hat T_1|_{\mathcal N_1},\hat T_2|_{\mathcal N_1},\hat V|_{\mathcal N_1})$ is a
 $\Gamma_3$-unitary and $(\hat T_1|_{\mathcal N_2},\hat T_2|_{\mathcal N_2},\hat V|_{\mathcal N_2})$ is a
 pure $\Gamma_3$-co-isometry\,;
 \item $\mathcal N_2$ can be identified with $H^2(\mathcal D_{\hat V})$, where $D_{\hat V}$ has same dimension
 as of $\mathcal D_P$. The operators $\hat T_1|_{\mathcal N_2}$, $\hat T_2|_{\mathcal N_2}$ and $\hat V|_{\mathcal
 N_2}$ are respectively unitarily equivalent to $T_{B_1+B_2^*\bar z}$, $T_{B_2+B_1^*\bar z}$ and $T_{\bar
 z}$ defined on $H^2(\mathcal D_{\hat V})$, $(B_1,B_2)$ being the FOP of $(\hat T_1,\hat T_2,\hat V)$.
 \end{enumerate}
 \end{thm}
 \begin{proof}
 Since the FOP $({F_1}_*,{F_2}_*)$ is almost normal,
 by Corollary \ref{isometric-dilation}, we have that $({\hat T_1}^*,{\hat T_2}^*,{\hat V}^*)$
 is minimal $\Gamma_3$-isometric dilation of $(S_1^*,S_2^*,P^*)$,
 where ${\hat V}^*$ is the minimal isometric dilation of $P^*$.
 Therefore by Proposition \ref{dilation-extension}, $(\hat T_1,\hat T_2,\hat V)$ is
 $\Gamma_3$-co-isometric extension of $(S_1,S_2,P)$. So we have that $\mathcal H$ is a
 common invariant subspace of $\hat T_1,\hat T_2$ and $\hat V$ and
 $\hat T_1|_{\mathcal H}=S_1\,,\; \hat T_2|_{\mathcal H}=S_1\,,\; \hat V|_{\mathcal H}=P$.
 Again since $({\hat T_1}^*,{\hat T_2}^*,\hat V^*)$ is a $\Gamma_3$-isometry, by
 Wold decomposition (see Theorem \ref{G-isometry}, part-(4)),
 there is an orthogonal decomposition $\mathcal N_*=\mathcal N_1\oplus\mathcal N_2$ into
 reducing subspaces of $\hat T_1,\hat T_2$ and $\hat V$ such that
 $(\hat T_1|_{\mathcal N_1},\hat T_2|_{\mathcal N_1},\hat V|_{\mathcal N_1})$ is a
 $\Gamma_3$-unitary and $(\hat T_1|_{\mathcal N_2},\hat T_2|_{\mathcal N_2},\hat V|_{\mathcal N_2})$
 is a pure $\Gamma_3$-co-isometry. If we denote
 $(\hat T_1|_{\mathcal N_1},\hat T_2|_{\mathcal N_1},\hat V|_{\mathcal N_1})
 = (T_{11},T_{12},V_1)$ and $(\hat T_1|_{\mathcal N_2},\hat T_2|_{\mathcal N_2},\hat V|_{\mathcal N_2})
 = (T_{21},T_{22},V_2)$ then with respect to the orthogonal decomposition $\mathcal K_*=\mathcal K_1\oplus \mathcal K_2$
 we have
 \[
 \hat T_1=
 \begin{bmatrix}
 T_{11} & 0 \\
 0 & T_{21}
 \end{bmatrix}\;,\;
 \hat T_2=
 \begin{bmatrix}
 T_{12} & 0 \\
 0 & T_{22}
 \end{bmatrix}\;
 \text{ and }
 \hat V=
 \begin{bmatrix}
 V_1 & 0 \\
 0 & V_2
 \end{bmatrix}.
 \]
 The fundamental equations
\begin{align*}
\hat T_1-{\hat T_2}^*\hat V &=D_{\hat V}X_1D_{\hat V}\,,\\
\hat T_2-{\hat T_1}^*\hat V &=D_{\hat V}X_2D_{\hat V}
\end{align*}
 of $(\hat T_{1},\hat T_{2},\hat V)$ clearly become
 \begin{align*}
 \begin{bmatrix}
 T_{11}-T_{12}^*V_1 & 0 \\
 0 & T_{21}-T_{22}^*V_2
 \end{bmatrix}
 &=\begin{bmatrix}
 0 & 0 \\
 0 & D_{V_2}X_{12}D_{V_2}
 \end{bmatrix},\quad
 X_1=\begin{bmatrix} X_{11}\\X_{12}\end{bmatrix}\\
 \text{ and }\quad
\begin{bmatrix}
 T_{12}-T_{11}^*V_1 & 0 \\
 0 & T_{22}-T_{21}^*V_2
 \end{bmatrix}
 &=\begin{bmatrix}
 0 & 0 \\
 0 & D_{V_2}X_{22}D_{V_2}
 \end{bmatrix},\quad
 X_2=\begin{bmatrix} X_{21}\\X_{22}\end{bmatrix}.
 \end{align*}
 Since $\mathcal D_{\hat V}=\mathcal D_{V_2}$, $(\hat T_1,\hat T_2,\hat V)$ and $(T_{21},T_{22},V_2)$ have the
 same FOP. Now we apply Theorem \ref{model1} to
 the pure $\Gamma_3$-isometry $(T_{21}^*,T_{22}^*,V_2^*)$ and get the following:
 \begin{enumerate}
 \item $\mathcal N_2$ can be identified with $H^2(\mathcal D_{V_2})=H^2(\mathcal
 D_{\hat V})$;
 \item The operators $T_{21}^*,T_{22}^*$ and $V_2^*$ are respectively unitarily equivalent to
 $T_{B_1^*+B_2 z}$, $T_{B_2^*+B_1 z}$ and $T_z$ defined on $H^2(\mathcal D_{\hat V})$,
 $(B_1,B_2)$ being the FOP of $(\hat T_1,\hat T_2,\hat V)$.
 \end{enumerate}
 Therefore, $(\hat T_1|_{\mathcal N_2},\hat T_2|_{\mathcal N_2},\hat V|_{\mathcal
 N_2})$ is unitarily equivalent to $(T_{B_1+B_2^*\bar z},T_{B_2+B_1^*\bar z},T_{\bar
 z})$ defined on $H^2(\mathcal D_{\hat V})$. Also since $\hat V^*$ is the minimal
 isometric dilation of $P^*$ by Proposition \ref{easyprop1}, $\mathcal D_{\hat V}$
 and $\mathcal D_P$ have same dimension.
\end{proof}

\section{Some important classes of $\Gamma_3$-contractions and their
dilations}\label{further-dilation}

In the previous section we saw that almost normality of the
fundamental operator pair is sufficient for a
$\Gamma_3$-contraction to possess rational dilation. Also in the
section before that, where we saw that some
$\Gamma_3$-contractions did not dilate because their FOPs were not
almost normal. Here we shall see that there are
$\Gamma_3$-contractions which dilate even
without having an almost normal FOP.\\

Before going to the examples, we need to say a few words about
operator theory on the symmetrized bidisc because throughout this
section we shall explore a connection between the operator theory
on $\Gamma_2$ and $\Gamma_3$. We mention here that in stead of
denoting the closed symmetrized bidisc by $\Gamma_2$, we shall
follow notations from the existing literature and denote it by
$\Gamma$. So, the closed symmetrized bidisc is defined as
\[
\Gamma_2=\Gamma=\{(z_1+z_2,z_1z_2)\in\mathbb C^2\,:\, |z_i|\leq
1\,, i=1,2\}.
\]
Operator theory on the symmetrized bidisc has been extensively
studied in \cite{ay-jfa, ay-jot, tirtha-sourav, tirtha-sourav1,
sourav, sarkar}.

\begin{defn}A pair of commuting operators $(S,P)$ on a Hilbert space
$\mathcal H$ for which $\Gamma$ is a spectral set is called a
$\Gamma$-\textit{contraction}.
\end{defn}

Let us consider the map
\begin{align*}
\varrho \,:\,& \mathbb C^2 \rightarrow \mathbb C^3 \\&
(z_1,z_2)\mapsto (z_1,z_2,0).
\end{align*}
This map embeds $\Gamma$ inside $\Gamma_3$ in the following way.
\begin{lem}
Let $\Gamma_3^0=\{(s_1,s_2,p)\in\Gamma_3: p=0 \}$. Then
$\varrho(\Gamma)=\Gamma_3^0$.
\end{lem}
\begin{proof}
We have that $\varrho(z_1,z_2)=(z_1,z_2,0)$ for all $(z_1,z_2)$ in
$\mathbb C^2$. Let $(s,p)\in \Gamma$. Then there are points
$\lambda_1,\lambda_2$ in the closed unit disc $\overline{\mathbb
D}$ such that $(s,p)=(\lambda_1+\lambda_2,\lambda_1\lambda_2)$.
Now clearly the point $(s,p,0)$, which is the image of $(s,p)$
under $\varrho$, is the symmetrization of the points
$\lambda_1,\lambda_2,0$ of $\overline{\mathbb D}$. Therefore,
$(s,p,0)\in\Gamma_3$ and in particular $(s,p,0)$ is in
$\Gamma_3^0$.

Conversely, let $(s_1,s_2,0)\in\Gamma_3^0$. Since $(s_1,s_2,0)$ is
a point of $\Gamma_3$, there are points $z_1,z_2,z_3$ in
$\overline{\mathbb D}$ such that $
\pi_3(z_1,z_2,z_3)=(s_1,s_2,0)$. Now $z_1z_2z_3=0$ implies that at
least one of $z_1,z_2,z_3$ is $0$. Let us assume without loss of
generality that $z_3=0$. Then $s_1=z_1+z_2$ and $s_2=z_1z_2$. This
shows that $(s_1,s_2)\in\Gamma$. Hence the proof is complete.

\end{proof}

\begin{lem}\label{implem}
If $(S,P)$ is a $\Gamma$-contraction then $(S,P,0)$ is a
$\Gamma_3$-contraction.
\end{lem}

\begin{proof}
Let $p$ be a polynomial in 3-variables $z_1,z_2,z_3$ and let
$p_1(z_1,z_2)=p(z_1,z_2,0)$. Then $p_1$ is a polynomial in
2-variables $z_1,z_2$ and $p_1(z_1,z_2)=p\circ \varrho(z_1,z_2)$
Now
\begin{align*} \|p(S,P,0)\|=\| p_1(S,P)\| & \leq \|p_1\|_{\infty,
\Gamma}, \quad \textup{since }(S,P) \textup{ is a }\Gamma\textup{-contraction},\\
& =\|p\circ \varrho\|_{\infty,\Gamma}\\ & =\|p\|_{\infty,
\varrho(\Gamma)}\\& \leq \|p\|_{\infty, \Gamma_3}.
\end{align*}
Therefore $(S,P,0)$ is a $\Gamma_3$-contraction.
\end{proof}

We recall here a remarkable result of Agler and Young about
$\Gamma$-contractions which will be useful.
\begin{thm}\label{nicethm}
Let $(S,P)$ be a pair of commuting operators such that $\|P\|<1$
and the spectral radius of $S$ is less than $2$. Then $(S,P)$ is a
$\Gamma$-contraction if and only if
$\omega(D_P^{-1}(S-S^*P)D_P^{-1})\leq 1. $
\end{thm}
For details of the above result one can see Corollary 1.9 in
\cite{ay-jot}.

\begin{eg}
If $(S,P)$ is a $\Gamma$-contraction then $(S,P,0)$ is a
$\Gamma_3$-contraction. We know that if $\Gamma$ is a spectral set
for $(S,P)$ then $\Gamma$ is a complete spectral set for $(S,P)$
too. We now show that $\Gamma_3$ is a complete spectral set for
$(S,P,0)$. Let $\boldsymbol f =[f_{ij}]_{m\times n}$ be a
matricial polynomial in three variables $z_1,z_2,z_3$. Let
$f_{ij}^{\prime}(z_1,z_2)=f_{ij}(z_1,z_2,0)$ and
$\boldsymbol{f^{\prime}}=[f_{ij}^{\prime}]_{m\times n}$. Now
\begin{align*}
\|\boldsymbol f(S,P,0) \|=\|[f_{ij}(S,P,0)]_{m\times n} \|
&=\|[f_{ij}^{\prime}(S,P)]_{m\times n} \| \\& \leq
\sup_{(z_1,z_2)\in\Gamma}\|\boldsymbol{f^{\prime}}(z_1,z_2)\| \\&
=\sup_{(z_1,z_2)\in\Gamma}\|[f_{ij}^{\prime}(z_1,z_2)] \| \\&
=\sup_{(z_1,z_2,0)\in\Gamma_3}\|[f_{ij}^{\prime}(z_1,z_2,z_3)]\|
\\& \leq \|\boldsymbol f \|_{\infty, \Gamma_3}\,.
\end{align*}
Thus $\Gamma_3$ is a complete spectral set for $(S,P,0)$. So we
get a class of $\Gamma_3$-contractions which always have
$\Gamma_3$-unitary dilation.\\

Note that the FOP for such $\Gamma_3$-contractions are just
$(S,P)$ which may or may not be almost normal. Indeed, if we
choose $P$ to be non-normal with $\|P\|<1$ and $S$ to be normal
with norm of $S$ being sufficiently small so that the norm of
$D_P^{-1}(S-S^*P)D_P^{-1}$ is less than $1$. Then by Theorem
\ref{nicethm}, $(S,P)$ is a $\Gamma$-contraction. Since $S$ is
normal and $P$ is non-normal we have
\[
S^*S-SS^*\neq P^*P-PP^*
\]
 and hence the FOP $(S,P)$ is not almost
 normal. So we get a class of $\Gamma_3$-contractions that dilate
 to the distinguished boundary despite the fact that their FOPs are not almost normal. Thus, the almost
 normality of the FOP of a $\Gamma_3$-contraction is not necessary to
have a $\Gamma_3$-unitary dilation.
\end{eg}

\begin{eg}
In \cite{BSR}, Biswas and Shyam Roy described a technique of
obtaining $\Gamma_3$-contractions from $\Gamma$-contractions.
Indeed, Lemma 2.10 in \cite{BSR} shows that we can obtain a
$\Gamma_3$-contraction from a $\Gamma$-contraction $(S,P)$ by
symmetrizing a scalar times identity operator with the existing
$\Gamma$-contraction in the following way.\\

\noindent \textbf{Lemma.} {\em Let $(S,P)$ be a
$\Gamma$-contraction, then $(\alpha I +S,\alpha S+P,\alpha P)$ is
a $\Gamma_3$-contraction for all
$\alpha \in\overline{\mathbb D}$.}\\

We now start with a $\Gamma$-contraction $(S,P)$. By the above
lemma $(I+S,S+P,P)$ is a $\Gamma_3$-contraction. Let $F$ be the
fundamental operator of $(S,P)$. We now compute the FOP
$(F_1,F_2)$ of $(I+S,S+P,P)$.
\begin{align*}
(I+S)-(S+P)^*P=(I-P^*P)+(S-S^*P) &=D_P^2+D_PFD_P \\& =
D_P(I+F)D_P\,.
\end{align*}
Also
\[
(S+P)-(I+S)^*P = S-S^*P=D_PFD_P\,.
\]
Therefore, $(F_1,F_2)=(I+F,F)$. Clearly $(F_1,F_2)$ is almost
normal. Therefore, by Theorem \ref{main-dilation-theorem},
$(I+S,S+P,P)$ has normal $b\Gamma_3$-dilation.
\end{eg}

\begin{eg}
In \cite{JH}, Holbrook has shown that the multivariate von
Neumann's inequality holds for any number of $2\times 2$ commuting
contraction matrices, i.e, if $C_1,\hdots,C_n$ are commuting
$2\times 2$ matrices with $\|C_k\|\leq 1$ for all $k$ and if
$f:\mathbb D^n\rightarrow \mathbb D$ is analytic, then
$\|f(C_1,\hdots,C_n) \|\leq 1$. Moreover, any $n$-tuple of
commuting $2\times 2$ contractions has simultaneous commuting
unitary dilation. See Proposition 2 and Proposition 3 in \cite{JH}
for a proof to these results.

So, unlike the general case, $\mathbb D^3$ is a spectral set for
any three $2\times 2$ commuting contractions $C_1,C_2,C_3$. Let
$(U_1,U_2,U_3)$ be a commuting unitary dilation of
$(C_1,C_2,C_3)$. It is obvious that the symmetrization of
$U_1,U_2,U_3$, i.e, $\pi_3(U_1,U_2,U_3)$ is a $\Gamma_3$-unitary
dilation of the $\Gamma_3$-contraction $\pi_3(C_1,C_2,C_3)$.
\end{eg}

The almost normality of FOP of a $\Gamma_3$-contraction is
sufficient but not necessary for rational dilation. We have
success of dilation in cases when FOPs are not almost normal. We
could not determine the whole class of $\Gamma_3$-contractions
which dilate without having almost normal FOPs. So, determining
the entire class of $\Gamma_3$-contractions which dilate needs
further investigations.\\

\noindent \textbf{Acknowledgement.} The author greatly appreciates
Orr Shalit's help with his invaluable knowledge of multivariable
operator theory without which this article would never be
completed.\\

\vspace{0.72cm}

%\noindent \textbf{Acknowledgement.} We are grateful to Orr Shalit
%for providing warm and generous hospitality at Ben-Gurion
%University, Be'er Sheva, Israel.

\end{document}